\documentclass[11pt]{preprint}
\usepackage[full]{textcomp}
\usepackage[osf]{newtxtext}
\usepackage{comment}

\usepackage{amssymb}
\usepackage{mathtools}
\usepackage{hyperref}
\usepackage{breakurl}
\usepackage{mhenvs}
\usepackage{mhequ}
\usepackage{mhsymb}
\usepackage{booktabs}
\usepackage{tikz}
\usepackage{mathrsfs}
\usepackage{longtable}
\usepackage{graphicx}
\usepackage{subfig}
\usetikzlibrary{plotmarks}
\usetikzlibrary{decorations.pathreplacing}
\usetikzlibrary{decorations.pathmorphing}
\usetikzlibrary{fit}

\usepackage{float}
\usepackage{microtype}
\usepackage{comment}
\usepackage{wasysym}
\usepackage{centernot}
\usepackage{enumitem}
\usepackage{bm}
\usepackage{stackrel}

\newcommand{\eqlaw}{\stackrel{\mbox{\tiny law}}{=}}

\DeclareMathAlphabet{\mathbbm}{U}{bbm}{m}{n}

\DeclareFontFamily{U}{BOONDOX-calo}{\skewchar\font=45 }
\DeclareFontShape{U}{BOONDOX-calo}{m}{n}{
  <-> s*[1.05] BOONDOX-r-calo}{}
\DeclareFontShape{U}{BOONDOX-calo}{b}{n}{
  <-> s*[1.05] BOONDOX-b-calo}{}
\DeclareMathAlphabet{\mcb}{U}{BOONDOX-calo}{m}{n}
\SetMathAlphabet{\mcb}{bold}{U}{BOONDOX-calo}{b}{n}

\setlist{noitemsep,topsep=4pt}

\newcommand{\mcA}{\mathcal{A}}

\newcommand{\mcR}{\mathcal{R}}
\newcommand{\mcC}{\mathcal{C}}
\newcommand{\mcB}{\mathcal{B}}

\newcommand{\mcL}{\mathcal{L}}
\newcommand{\mcT}{\mathcal{T}}

\newcommand{\mcD}{\mathcal{D}}

\newcommand{\mcP}{\mathcal{P}}
\newcommand{\mcG}{\mathcal{G}}
\newcommand{\mcX}{\mathcal{X}}

\newcommand{\mcZ}{\mathcal{Z}}

\newtheorem{assumption}{Assumption}

\newcommand{\mbbG}{\mathbb{G}}
\newcommand{\mbbV}{\mathbb{V}}
\newcommand{\mbbW}{\mathbb{W}}
\newcommand{\mbbE}{\mathbb{E}}

\newcommand{\mbu}{\mathbf{u}}

\newcommand{\mbi}{\mathbf{i}}

\newcommand{\T}{\mathbb{T}}

\def\${|\!|\!|}

\def\scal#1{{\langle#1\rangle}}
\def\vscal#1{{(#1)}}
\def\bvscal#1{{\Big(#1\Big)}}

\newcommand{\mfT}{\mathfrak{T}}

\newcommand{\mft}{\mathfrak{t}}

\newcommand{\mfS}{\mathfrak{S}}

\newcommand{\mfG}{\mathfrak{G}}

\def\gr#1{#1\textnormal{-gr}}
\def\rnorm#1{#1;\rho}

\tikzset{
dot/.style={circle,fill=black,inner sep=0pt, minimum size=1.2mm}
}

\makeatletter
\DeclareRobustCommand{\cev}[1]{%
  {\mathpalette\do@cev{#1}}%
}
\newcommand{\do@cev}[2]{%
  \vbox{\offinterlineskip
    \sbox\z@{$\m@th#1 x$}%
    \ialign{##\cr
      \hidewidth\reflectbox{$\m@th#1\vec{}\mkern4mu$}\hidewidth\cr
      \noalign{\kern-\ht\z@}
      $\m@th#1#2$\cr
    }%
  }%
}
\makeatother

\newcommand{\mrd}{\mathop{}\!\mathrm{d}}

\newcommand{\bonds}{\mathscr{B}}
\newcommand{\obonds}{\overline{\mathscr{B}}}

\newcommand{\g}{\mathbf{g}}

\DeclareMathOperator{\Trace}{Tr}

\newcommand{\hol}{\textnormal{hol}}
\newcommand{\Haus}{\mathrm{H}}

\def\emptyset{{\centernot\ocircle}}

\usetikzlibrary{shapes.misc}
\usetikzlibrary{shapes.symbols}
\usetikzlibrary{shapes.geometric}
\usetikzlibrary{decorations}
\usetikzlibrary{decorations.markings}
\usetikzlibrary{arrows.meta}

\usetikzlibrary{calc}

\colorlet{darkblue}{blue!90!black}
\colorlet{darkred}{red!90!black}
\colorlet{darkgreen}{green!70!black}

\def\${|\!|\!|}

\def\Var{\mathrm{Var}}

\def\multiset#1{\mathcal{M} ( #1 ) }

\def\bphi{\boldsymbol{\phi}}

\def\E{\mathbb{E}}
\def\P{\mathbb{P}}
\def\R{\mathbb{R}}
\def\C{\mathbb{C}}
\def\Z{\mathbb{Z}}
\def\N{\mathbb{N}}

\title{Gauge field marginal of an Abelian Higgs model}
\author{Ajay~Chandra$^1$ and Ilya~Chevyrev$^2$}

\institute{Imperial College London, \email{a.chandra@imperial.ac.uk} \and University of Edinburgh, \email{ichevyrev@gmail.com}}

\date{\today}
\begin{document}
\maketitle
\begin{abstract}
We study the gauge field marginal of an Abelian Higgs model with Villain action defined on a 2D lattice in finite volume. Our first main result, which holds for gauge theories on arbitrary finite graphs and does not assume that the structure group is Abelian, is a loop expansion of the Radon--Nikodym derivative of the law of the gauge field marginal with respect to that of the pure gauge theory. This expansion is similar to the one of Seiler~\cite{Seiler82} but holds in greater generality and uses a different graph theoretic approach. Furthermore, we show ultraviolet stability for the gauge field marginal of the model in a fixed gauge. More specifically, we show that moments of the H{\"o}lder--Besov-type norms introduced in~\cite{Chevyrev19YM} are bounded uniformly in the lattice spacing. This latter result relies on a quantitative diamagnetic inequality that in turn follows from the loop expansion and elementary properties of Gaussian random variables.
\\[.4em]
\noindent {\small \textit{Keywords:} Lattice gauge theory, loop expansion, Abelian Higgs model, diamagnetic inequality, gauge fixing, ultraviolet stability}\\
\noindent {\small\textit{MSC classification:} 60D05, 81T13, 81T25.} 
\end{abstract}
\setcounter{tocdepth}{1}

\tableofcontents

\section{Introduction}\label{sec: intro}

The Abelian Higgs model is one of the simplest examples of a gauge theory. 
As a Euclidean quantum field theory, it is given formally by the Gibbs-type probability measure
\begin{equ}\label{eq:YMH_measure}
\mu(\mcD A, \mcD \Phi) \propto e^{-S(A,\Phi)}\mcD A \mcD \Phi\;,
\end{equ}
where $A$ is a gauge field (connection) on a $U(1)$-principal bundle $P\to M$ and $\Phi$ is a Higgs field (section) over a complex line bundle carrying a unitary representation of $U(1)$.
We will be concerned with the case that $M=[0,1]^2$, a $2$-manifold with boundary,
in which case $P$ is trivial and, after fixing a global section, $A$ and $\Phi$ can be represented as a $1$-form $A\colon M\to \R^{2}$ and a function $\Phi\colon M\to \C$.

Above, $S$ is an action of Yang--Mills--Higgs-type
\begin{equ}\label{eq:action}
S(A,\Phi) = \int_M |\mrd A|^2 + |\mrd_A\Phi|^2 + V(|\Phi|)\;,
\end{equ}
where $\mrd A = \partial_1 A_2-\partial_2 A_1$ is the curvature of $A$, $\mrd_A \colon \Phi \to \sum_{i=1}^2(\partial_i \Phi + \mbi A_i\Phi)\mrd x_i$ is the associated covariant derivative,
and $M=[0,1]^2$ is equipped with the Lebesgue measure.
The function $V\colon[0,\infty)\to\R$ is a potential, typically chosen as the `sombrero potential' $\Phi(x)=(x^2-1)^2$. 
Finally, $\mcD A \mcD \Phi$ is a purely formal Lebesgue measure on the space of (gauge field, Higgs field) pairs.
The terms $|\mrd_A \Phi|^2$ and (possibly) $V(|\Phi|)$ are non-quadratic, representing particle interactions, which render the theory non-trivial.

Making rigorous sense and studying properties of the measure $\mu$ has a history dating back to the late 70's.
Restricting to the sombrero potential,
the first successful programme to construct $\mu$ in both the continuum (ultraviolet) and infinite volume (infrared) limits is due to Brydges--Fr\"ohlich--Seiler~\cite{BFS79,BFS80,BFS81} (see also the monograph~\cite{Seiler82}).
Balaban~\cite{Balaban81I,Balaban81II,Balaban1983III} established lower and upper bounds on the vacuum energy in the 2D and 3D continuum limits in finite volume, a result later generalised by King~\cite{King86I,King86II} to full convergence in the continuum and infinite volume.
The above works employ a lattice regularisation.
Another approach based on stochastic quantisation was recently initiated in~\cite{Shen21, CCHS2d, CCHS3d} (see also~\cite{Chevyrev22} for a survey), which have given meaning to the Langevin dynamic of the (non-Abelian) version of $\mu$.

We mention also work on pure gauge theory (without Higgs) in the non-Abelian case in 2D~\cite{GKS89, Driver89, Fine91, Sengupta97, Levy03}, in which case the model is exactly solvable,
on the Abelian case in 3D and 4D~\cite{Gross83,Driver87}, and on the non-Abelian case in 3D and 4D~\cite{Balaban85IV,Balaban89,MRS93}. See also the survey~\cite{Chatterjee18}.
There have been a number of recent results in gauge theory (both on the lattice and the continuum) from the side of probability theory, see e.g.~\cite{Cao20, Sourav_flow, Sourav_state, CPS23, Chatterjee19,Chatterjee20Wilson,Chatterjee21Confine,Chevyrev19YM,Forsstrom21,FLV20,FLV21,FLV22,GS21,KasselLevy21,Shen22I, Shen22II,ShenZhuZhu24}.

Two issues that enter in the construction of~\eqref{eq:YMH_measure} are renormalisation and gauge invariance.
The issue of renormalisation appears in all (singular) quantum field theories and requires the addition of counterterms to the Lagrangian~\eqref{eq:action} for the measure~\eqref{eq:YMH_measure} to be well-defined and/or non-trivial when an ultraviolet cutoff is removed.
To wit, in the case $V(x)=x^4$, one requires replacing $S(A,\Phi)$ in~\eqref{eq:action} by
\begin{equ}
\int_M |\mrd A|^2 + |\mrd_A \Phi|^2 + |\Phi|^4 - \infty |\Phi|^2\;,
\end{equ}
where `$\infty$' stands for a constant that is diverging as the ultraviolet cutoff (e.g. lattice spacing) is sent to zero.
The issue of gauge invariance is specific to gauge theories and arises from the invariance of the action~\eqref{eq:action} under the infinite dimensional gauge group $\{g\colon M\to U(1)\}$.
See Section~\ref{sec:model} for a lattice version of this gauge group action and~\cite[Sec.~1]{CCHS2d} for a discussion on the consequences of gauge-invariance on the construction of $\mu$.

In this paper, we study lattice approximations of the Abelian Higgs models on the square $M=[0,1]^2$ with Villain action.
Specifically, we consider the probability measure defined as the law of the random variable $(\g,\bphi)$ with expectation
\begin{equ}\label{eq:disc_YMH}
\E[f(\g,\bphi)] \propto \int_{\mcG \times \C^{\mathring\Lambda}} f(g,\phi) \prod_{p} Q(g(\d p)) e^{\scal{\phi,\Delta_g \phi} - \int_{\mathring\Lambda} V(|\phi|)} \mrd g \mrd \phi\;,
\end{equ}
for any bounded measurable $f\colon \mcG\times \C^{\mathring\Lambda} \to \R$
(the proportionality constant is independent of $f$).
Here, $\Lambda$ is a lattice in $M$ with spacing $2^{-N}$ for $N\geq 1$, $\mcG$ is the space of gauge fields, i.e. functions from bonds of $\Lambda$ into $U(1)$,
$\mathring \Lambda$ is the set of non-boundary nodes of $\Lambda$,
$V\colon[0,\infty) \to \R$ is a suitable potential,
e.g. $V(|\phi|)=|\phi|^4 - C|\phi|^2$,
and
\begin{equ}\label{eq:heat_kernel}
Q(x) = e^{2^{-2N}\Delta}(x) \propto \sum_{n\in \Z} e^{-(x+ n)^2/2^{-2N+1}} \;,
\end{equ}
is the heat kernel on $U(1)$ at time $2^{-2N}$ (we identify here $x\in U(1)$ with an element of $[0,1)$).
The product in~\eqref{eq:disc_YMH} is over all plaquettes $p$ of the lattice $\Lambda$ and $\int_{\mathring{\Lambda}}$ denotes integration over $\mathring{\Lambda}$ against the counting measure weighted by $2^{-2N}$, which approximates the continuum Lebesgue measure. A complete definition of the model is given in Section~\ref{sec:model}.

The framework for this model is partly drawn from~\cite{BFS79, Levy03,Kenyon11,KasselLevy21}.
If the Higgs field is absent, this model agrees with the Abelian pure gauge theory of~\cite{Levy03,Levy06} (see Section~\ref{subsec:hol_bounds}), which has a Gaussian nature and which is exactly solvable in the continuum limit.
The probability measure~\eqref{eq:disc_YMH} is a lattice version of~\eqref{eq:YMH_measure} - the relationship between $A$ and $\g$ that one should keep in mind is $\g(x,x+2^{-N}e_j) \approx e^{\mbi 2^{-N}A_j(x)}$ for $j=1,2$.
We emphasise that the potential $V$ \textit{can} depend on the lattice spacing, and therefore handles the case of renormalised potentials necessary for a non-trivial continuum limit.
Furthermore, the presence of the Higgs field destroys the Gaussian nature of the pure gauge theory and no exact solvability properties of the continuum limit of $\g,\bphi$ are known.

Our main object of study is the gauge field marginal $\g$.
Our first main result, which applies to any  Higgs-type model,
provides a loop expansion for the integral
\begin{equ}\label{eq:loop_exp}
\int_{H^{\mathring\Lambda}} e^{\scal{\phi,\Delta_g \phi} - \int_{\mathring\Lambda} V(|\phi|)} \mrd \phi = \sum_\ell c_\ell \Trace \hol(g,\ell)\;,\qquad c_\ell \geq 0\;,
\end{equ}
where $H$ is a (complex or real) Hilbert space, $g$ is a gauge field of arbitrary operators on $H$,
the sum is over all loops in $\mathring \Lambda$, $c_\ell\geq 0$ are \textit{explicit non-negative} constants, and $\hol(g,\ell)$ is the holonomy (ordered product) of $g$ around $\ell$.
We state and prove this result in Section~\ref{sec:loop}.
The proof uses an inductive procedure to integrate out the lattice sites and we develop
a graph theoretic framework to keep track of the constants $c_\ell$ as the induction proceeds.
The loop expansion~\eqref{eq:loop_exp} generalises and clarifies the results of~\cite[Thm.~2.6 \& Lem.~2.9]{Seiler82}
and~\cite[Sec.~3]{BFS79}.\footnote{See also~\cite[Thm.~1]{BFS79_Nuc_Phys_B} where an expansion similar to~\eqref{eq:loop_exp} is claimed and~\cite{BFS79,BFS80} are referenced for a proof, but the proofs therein assume $V$ is quadratic and we are unable to locate a proof in these works for more general potentials $V$.}
Specialising to the case $G=U(1)$,
this loop expansion allows us to write the law of $\g$ in~\eqref{eq:disc_YMH} in terms of the pure gauge theory with a Radon--Nikodym derivative of `positive type'.

The second main result, which uses~\eqref{eq:loop_exp}, provides moment estimates on gauge-invariant observables of $\g$.
These moment estimates can be seen as a quantitative diamagnetic inequality and are summarised in Corollaries~\ref{cor:moment_bound_hol} and~\ref{cor:moment_bound_plaq_sum} in Section~\ref{sec:diamagnetic}.

Finally, we show moment bounds, uniform in the lattice spacing, on a gauge-fixed version of $\g$ in a H{\"o}lder--Besov-type space.
This result implies ultraviolet stability of the gauge-field marginal and is summarised in the following theorem.
\begin{theorem}\label{thm:moments_YM}
There exists a $\mfG$-valued random variable $\mbu$, measurable with respect to $\g$,
such that $\E[|\log \g^{\mbu}|_{\beta}^q] \leq C$ for all $\beta\in (0,1)$ and $q>0$, where $C$ depends only on $\beta,q$ and \emph{not} on $N$.
\end{theorem}
We give the proof of Theorem~\ref{thm:moments_YM} in Section~\ref{sec:moments}.
Above, $|\cdot|_\beta$ is a norm on discrete $1$-forms introduced in~\cite{Chevyrev19YM} (see also~\cite{CCHS2d} for a natural version of this norm in the continuum). We recall the definition of $|\cdot|_\beta$ in
Section~\ref{sec:1-forms}.
We also refer to Section~\ref{sec:model} for the definition of the gauge group $\mfG$ and the gauge transformation $\g^\mbu$,
and to Section~\ref{subsec:notation}
for the definition of $\log \colon U(1)\to [-\pi,\pi)$.

The main step in the proof of Theorem~\ref{thm:moments_YM} is to combine
the quantitative diamagnetic inequality (Corollary~\ref{cor:moment_bound_plaq_sum})
with a gauge fixing procedure in Section~\ref{sec:gauge_fix}.
This procedure is a simplification and improvement of the one given in~\cite[Sec.~4]{Chevyrev19YM}, but is restricted to the Abelian case.
Our results are furthermore more quantitative than those of~\cite{Chevyrev19YM}, which,
in our notation, only show tightness of $\log \g^\mbu$ rather than moment estimates.
See, however,~\cite[Sec.~9]{CS23} where moment estimates were also recently shown in the non-Abelian 2D pure Yang--Mills case and applied to show that
the Yang--Mills measure on $\T^2$ is invariant for the associated stochastic quantisation equation (tightness of $\log \g^\mbu$ is, in general, insufficient to show this invariance).

A consequence of Theorem~\ref{thm:moments_YM} is that piecewise constant extensions of $\log \g^\mbu$ are tight in the H\"older--Besov space $\CC^{\beta-1}$ for $\beta<1$ (see~\cite[Prop.~3.21]{Chevyrev19YM}).
Furthermore, $\log \g^\mbu$ converges in law along subsequences as $N\to\infty$ in $\Omega^1_\beta$ for $\beta\in (0,1)$ in the same sense as in~\cite[Thm~3.26]{Chevyrev19YM}.\footnote{Since the
$\g$ live on different lattices indexed by $N$, making this statement precise is not entirely trivial. Since this is not our main focus, we refrain from doing so here.}
We recall that $(\Omega^1_\beta,|\cdot|_\beta)$ is a Banach space of distributional $1$-forms that embeds into $\CC^{\beta-1}$ and on which holonomies and Wilson loops are well-defined and continuous (for axis-parallel paths).
Compared to~\cite{BFS81,King86I},
Theorem~\ref{thm:moments_YM} has a simpler proof
and yields a stronger form
of ultraviolet stability for the gauge-field marginal $\g$ (furthermore, only gauge-invariant fields are considered in~\cite{BFS81,King86I}) but makes no statement on the stability of the Higgs field $\bphi$.
We also mention~\cite{Klimek87} where tightness (but in a gauge-invariant sense and in a much weaker space)
is shown for a non-Abelian gauge field coupled to a fermionic field
(one can also view \cite[Eq.~(24), App.~B]{Klimek87} as a version of Corollary~\ref{cor:moment_bound_hol}
for fermionic coupling and for loops with a single plaquette).
\begin{remark}
For simplicity, we consider only Dirichlet boundary conditions on the Higgs field and free boundary conditions on the gauge field in~\eqref{eq:disc_YMH}.
Our method can be adapted to other boundary conditions, such as periodic or half-periodic
(see Remark~\ref{rem:periodic} for subtleties that arise in the periodic case).
It could also be used to treat general graphs with a certain structure -- the key property that we use is that there is a distinguished collection of plaquettes (faces) and the pure gauge theory (without the Higgs field) is determined by a centred Gaussian vector indexed by the plaquettes as in Section~\ref{subsec:hol_bounds}.
See also Remark~\ref{rem:generalisation}.
We furthermore believe our methods can be adapted to take the infinite volume limit, but we do not explore this here.

Note that in~\eqref{eq:disc_YMH} we use the Villain (heat kernel) action;
this choice allows us to construct $\g$ from a Gaussian random field and the Radon--Nikodym derivative~\eqref{eq:loop_exp}. 
It would be interesting to find an extension of our results to other actions, such as the Wilson action, where the exact Gaussian structure is unavailable.

Finally, we believe the methods of this paper can be used to analyse the case when the Higgs field $\bphi$ takes values in $\R^N$ or $\C^N$ provided that the structure group remains Abelian.
\end{remark}
\subsection{Notation}\label{subsec:notation}

We collect some frequently used notation.
Let $\mbi = \sqrt{-1}$ be the imaginary unit and $U(1)=\{z\in\C \,: \,|z|=1\}$ be the circle group.
We let $\log \colon U(1) \to [-\pi,\pi)$ be the map sending $e^{\mbi x}$ to $x$ (we thus identify $\R$ with the Lie algebra of $U(1)$).
We let $\N = \{0,1,\ldots\}$ denote the non-negative integers.

By $x\lesssim y$ we mean that $x \leq Cy$ for some universal constant $C\geq 0$.
When $C$ depends on certain parameters, e.g. $\alpha,\beta,\ldots$ we write $x\lesssim_{\alpha,\beta,\ldots} y$. 
When writing expectations we often use bold symbols for random variables and non-bold symbols for integration variable, e.g. writing $\E[f(\mathbf{X})] = \int f(X) \mrd \mu(X)$.

\section{Definition of the model}
\label{sec:model}

In this section, we give the definitions necessary to define the model~\eqref{eq:disc_YMH}.
We return to these definitions in Section~\ref{sec:diamagnetic} and Section~\ref{sec:loop} can be read without them.

We equip $M\eqdef [0,1]^2$ with its natural geodesic distance $|x-y|$.
We let $\partial M$ denote the boundary of $M$.
For every integer $N \geq 1$,
we define
$\Lambda = \{k2^{-N} \,:\, k= 0,\ldots, 2^{N}\}^2\subset M$, the lattice of with mesh size $2^{-N} $.
We let $\partial \Lambda = \Lambda\cap \partial M$ denote the boundary nodes of $\Lambda$ and $\mathring \Lambda = \Lambda\setminus \partial\Lambda$ the non-boundary nodes.
We equip $\Lambda$ and $\mathring{\Lambda}$ with the counting measure weighted by $2^{-2N}$ and denote the correspondingly integrals for $\phi\in \C^{\Lambda}$ simply by $\int_{\Lambda} \phi = 2^{-2N}\sum_{x\in\Lambda}\phi_x$,
and likewise for $\mathring{\Lambda}$.
\begin{remark}
We will mostly suppress the dependence on $N$ in most of our notation.
When this dependence becomes important, we will write $\Lambda^{(N)}$, $\mathring\Lambda^{(N)}$, etc.
\end{remark}
We define $\bonds = \bonds^{(N)}$ as the collection of all 
oriented bonds, i.e. all pairs $(x,y)\in\Lambda^2$ for which $|x-y|=2^{-N}$.
We let $\mcG$ denote the set of maps $g\in U(1)^\bonds$ such that $g_{xy}= \bar g_{yx}$.
An element of $\mcG$ is called a \emph{gauge field}.

The vector space $\C^{\mathring\Lambda}$ is the space of Higgs field configurations on $\Lambda$.
We identify $\C^{\mathring\Lambda}$ with a subspace of $\C^\Lambda$ by setting $\phi_x\eqdef 0$ for all $\phi\in\C^{\mathring\Lambda}$ and $x\in\partial\Lambda$.
We equip $\C^{\Lambda}$ with the inner product $\scal{\cdot,\cdot}$
defined by
\begin{equ}\label{eq:inner_prod_def}
\scal{\phi,\phi'} =
\int_\Lambda
\phi\overline{\phi'} = 2^{-2N}\sum_{x\in\Lambda} \phi_x\overline{\phi'_{x}}\;.
\end{equ}
We similarly equip $\C^\bonds$ with the inner product
\[
\scal{\eta,\eta'} \eqdef
2^{-2N} \sum_{b \in \bonds}
\eta_{b}\overline{\eta'_{b}}\;.
\]
\begin{definition}
For $g \in \mcG$, the \emph{covariant derivative} $\mrd_g \colon \C^\Lambda \to \C^{\bonds}$ is defined for $(x,y)\in\bonds$ by
\[
(\mrd_g \phi)(x,y) = 2^N (g_{xy}\phi_y - \phi_x)\;.
\]
\end{definition}
Note that the adjoint operator $\mrd_g^* \colon \C^\bonds\to \C^\Lambda$ is given by
\begin{equ}
\mrd_g^*\eta (x) = 2^N\sum_{y\to x}\{g_{xy}\eta_{yx}-\eta_{xy}\} \;,
\end{equ}
where $y\to x$ denotes that $y\in\Lambda$ with $(y,x)\in\bonds$.
Indeed
\begin{equs}
\scal{\mrd_g\phi,\eta}
= 2^{-N}\sum_{x\in\Lambda}\sum_{y\to x} (g_{xy}\phi_y - \phi_x) \bar\eta_{xy}
= 2^{-N}\sum_x \phi_x\sum_{y\to x} \overline{(g_{xy}\eta_{yx}-\eta_{xy} )} 
\end{equs}
where the final equality follows from $g_{xy}=\bar g_{yx}$.
\begin{definition}
We define the \emph{covariant Laplacian} associated with $g\in\mcG$ as the negative semi-definite
self-adjoint linear operator
\begin{equ}
\Delta_g
\eqdef -\frac12\pi_{\mathring\Lambda} \circ \mrd_g^*\circ\mrd_g
\colon \C^{\mathring\Lambda} \to \C^{\mathring\Lambda}\;,
\end{equ}
where $\pi_{\mathring\Lambda}\colon\C^{\Lambda}\to\C^{\mathring\Lambda}$ is the canonical projection.
\end{definition}
Note that, for all $x\in\mathring\Lambda$,
\begin{equ}\label{eq:Delta_g}
\Delta_g\phi(x) = 2^N \frac12 \sum_{y\to x} \{\mrd_g\phi(x,y)-g_{xy}\mrd_g\phi(y,x)\}
= 2^{2N}\sum_{y\to x} \{g_{xy}\phi_y - \phi_x\}\;,
\end{equ}
and that $\Delta_g$ defines a positive semi-definite symmetric quadratic form on $\C^{\mathring\Lambda}$
\begin{equ}
- \scal{\phi, \Delta_{g} \phi'} = \frac12\scal{\mrd_g \phi, \mrd_g \phi'} =
\sum_{y\to x}
(g_{xy}\phi_y - \phi_x)\overline{(g_{xy}\phi'_y - \phi'_x )}\;.
\end{equ}
Remark that our definition of the Laplacian agrees (up to a scalar multiple) with the definitions in~\cite[Sec.~2J]{KasselLevy21} and~\cite[Sections~3,~8]{Kenyon11}.
\begin{definition}\label{def:gauge_transforms}
A \emph{gauge transformation} is a map $u\colon \Lambda\to U(1)$.
We denote by $\mfG$ the set of such maps.
Every $u\in\mfG$ acts on $\mcG$ by
\begin{equ}
g_{xy} \mapsto (u\cdot g)_{xy} \eqdef g^u_{xy} \eqdef u_x g_{xy} u_y^{-1}\;,
\end{equ}
on $\C^{\Lambda}$
by
\begin{equ}
\phi_x \mapsto (u\cdot \phi)_x \eqdef \phi^u_x \eqdef u_x \phi_x\;,
\end{equ}
and on $\C^\bonds$ by
\begin{equ}
\eta_{xy} \mapsto (u\cdot\eta)_{xy}\eqdef \eta^u_{xy} \eqdef u_x\eta_{xy}\;.
\end{equ}
\end{definition}
Note that every $u\in\mfG$ is an isometry on $\C^\Lambda$ and $\C^{\bonds}$,
and that we have the covariant identities
\begin{equ}\label{eq:gauge_covar}
(\mrd_g \phi)^u
= \mrd_{g^u} \phi^u \;,\quad
(\mrd^*_g \eta)^u
= \mrd^*_{g^u} \eta^u \;,\quad
(\Delta_g\phi)^u 
= \Delta_{g^u} \phi^u \;.
\end{equ}
\begin{definition}\label{def:loops}
Denote by
\begin{equ}\label{eq:loops_def}
\mcb{L} = \{(\ell_0,\ell_1,\ldots,\ell_n) \in \Lambda^{n+1}\,:\,
n\geq 0\,,\,
\ell_0=\ell_n\,,\,
|\ell_i-\ell_{i+1}|=2^{-N}
\}
\end{equ}
the set of all loops in $\Lambda$.
We define  the subset $\mathring{\mcb{L}}\subset \mcb{L}$
that contains loops only visiting sites in $\mathring\Lambda$, i.e. we restrict to $\ell_i \in\mathring{\Lambda}$ in~\eqref{eq:loops_def}.

A \emph{plaquette} of $\Lambda$ is a loop $p\in\mcb{L}$ of the form 
\begin{equ}\label{eq:plaq}
p=(x,x+2^{-N}e_1,x+2^{-N}e_1+2^{-N}e_2,x+2^{-N}e_2,x)\;,
\end{equ}
where $x\in\Lambda$.
Let $\mathbf{P} \subset \mcb{L}$ denote the set of plaquettes of $\Lambda$.

For $g\in\mcG$ and $\ell=(\ell_0,\ldots,\ell_n)\in\mcb{L}$, let $\hol(g,\ell)\in U(1)$
denote the holonomy (i.e. ordered product) of $g$ around $\ell$ defined by
\begin{equ}
\hol(g,\ell) = g(\ell_0,\ell_1)g(\ell_1,\ell_2)\ldots g(\ell_{n-1},\ell_n)\;.
\end{equ}
For $p\in\mathbf{P}$, we use the shorthand $g(\d p) = \hol(g,p)$.
\end{definition}
%
%
In the rest of the article, we take $V\colon [0,\infty)\to \R$ such that $\int_0^\infty e^{\alpha x^2 -V(x)}\mrd x < \infty$ for all $\alpha>0$
(this condition is used in the proof of Lemma~\ref{lemma:specific_loop_expansion} but can be relaxed), e.g. $V(x)=x^4-C^{(N)}x^2$ for any $C^{(N)}\in \R$.
With these notations, the probability measure~\eqref{eq:disc_YMH} is now defined.

\section{Loop expansion revisited}
\label{sec:loop}

The main results of this section are Theorems~\ref{thm:loop_expansion} and~\ref{thm:real_loop_expansion} which show
a loop expansion of the type~\eqref{eq:loop_exp}
in a completely general setting for complex (resp. real) representations of a gauge field $g$ with a structure group that is not necessarily Abelian or compact.

As in~\cite[Thm.~2.6 \& Lem.~2.9]{Seiler82}, we develop our loop expansion by iteratively integrating out the Higgs field site by site.
However the argument of \cite{Seiler82} is based on imposing conditions on the representation of the gauge group $G$ after integration over $G$ - see ``Property (R)'' of \cite[Def.~2.7]{Seiler82}. 
Our approach does not use any properties of the representation nor integration over $G$; in fact, we even allow the ``gauge field'' $g_{e}$ to act on the Higgs vector space $H$ as a non-invertible operator. 
Our main tool is computing moments of the uniform distribution on a complex (resp. real) sphere using Wick's rule for complex (resp. real) Gaussians, see\footnote{The proof of~\cite[Thm.~2.6]{Seiler82} gives little detail and relatively inexplicit formulae so it is hard to compare our methods further. We make the minor point that, in the complex case, Lemma~\ref{lemma:spherical_integral}  indicates assumption~\cite[Eq.~(2.13)]{Seiler82} looks unreasonable as it is written.} Lemmas~\ref{lemma:spherical_integral} and~\ref{lemma:spherical_integral_real}. 

We fix a choice of finite dimensional complex (resp. real) Hilbert space $H$ and write $(\cdot,\cdot)$ for the corresponding sesquilinear (resp. bilinear) inner product on $H$ and $|\cdot|$ for the norm. 
We define $\hat{H}$ to be the unit sphere in $H$, that is $\hat{H} = \{ \varphi \in H: |\varphi| = 1\}$. 
We write $\mrd \varphi$ for Lebesgue measure on $H$ (giving volume $1$ to any cube formed by an orthonormal basis) and write $\mrd \hat{\varphi}$ for the induced measure on $\hat{H}$ coming from  $\mrd \varphi$.

\begin{assumption}\label{assumption:single_site}
Our ``single site'' measures for the Higgs field will be measures $\mrd \rho^{\lambda}$ on $H$ which are of the form
\[
\int_{H} f(\phi)\ \mrd \rho^{\lambda}(\phi)
=
\int_{0}^{\infty}  \int_{\hat{H}} 
f(s \hat{\phi})\ \mrd \hat{\phi}\ \mrd \lambda(s)\;,
\]
where the radial measure $\lambda$ is always assumed to have better\footnote{This is a simple assumption to state but stronger than what we actually need. The actual integrability criterion is being able to expand the exponentials on the left-hand sides of Theorems~\ref{thm:loop_expansion} and~\ref{thm:real_loop_expansion} and interchange the integral and sum.} than Gaussian tails, i.e. we assume that, for any $\alpha > 0$,  $\int_{0}^{\infty} e^{\alpha s^2} \mrd \lambda(s) < \infty$. 
\end{assumption}

\begin{definition}\label{def:oriented_muligraph}
An oriented multi-graph $\mathbb{G}=(\mbbV,\mbbE)$ with abstract edge set consists of finite sets $\mbbV$ and $\mbbE$ where $\mbbE$ is endowed with two maps $\mbbE\ni e\mapsto \underline e\in\mbbV$ and $\mbbE\ni e\mapsto \overline e\in\mbbV$.
Elements of $\mbbV$ are called vertices and elements of $\mbbE$ are called edges.
The vertices $\underline e$ and $\overline e$ are called the starting and ending points of $e$, respectively.

The term multi-graph (as opposed to graph) here refers to the fact that we are allowed to have distinct $e_1,e_2 \in \mathbb{E}$, with $\underline{e_1} = \underline{e_2}$ and $\overline{e_1} = \overline{e_2}$.

Note that we allow $\underline e = \overline e$, that is we allow self-loops in $\mathbb{G}$ and we write $\mathbb{E}^{o} \subset \mathbb{E}$ for the set of self-loops, and $\vec{\mathbb{E}} = \mbbE \setminus \mbbE^{o}$ for the set of edges of $\mbbE$ which are not self-loops.

\end{definition}
We fix for both Sections~\ref{subsec:complex_case} and~\ref{subsec:real_case} below an oriented finite multi-graph $\mathbb{G} = (\mathbb{V},\mathbb{E})$. 

\subsection{The complex case}\label{subsec:complex_case}

In this subsection we assume $H$ is a finite dimensional complex Hilbert space. 
We start with the following lemma regarding integration on $\hat{H}$. 
\begin{lemma}\label{lemma:spherical_integral}
For any $N, M \in \N$ and vectors $a_{1},\dots,a_{M}, b_{1},\dots,b_{N} \in H$, we have 
\begin{equ}
\int_{\hat{H}} \prod_{i=1}^M \vscal{a_i,\hat{\varphi}}\prod_{j=1}^N \vscal{\hat{\varphi},b_j} \ \mrd \hat{\varphi} 
=
\delta_{NM} K^{\C}_{N} \sum_{\sigma\in\mfS_{N}} 
\prod_{i=1}^{N} \vscal{a_i,b_{\sigma(i)}}\;,
\end{equ}
where $\mathfrak{S}_{N}$ is the set of permutations on $N$ elements and 
\begin{equ}\label{constant_KN}
K^{\C}_{N} \eqdef \frac{2 \pi^{\dim_\C(H)}}{(N+\dim_\C(H) - 1)!}\;,
\end{equ}
where $\dim_\C(H)$ is the dimension of $H$ as a complex vector space.
\end{lemma}
\begin{proof}
Let $Z$ be a standard complex Gaussian on $H$ - that is, $Z$ is centred and for any $a,b \in H$, we have $\E[(a,Z)(Z,b)] = \vscal{a,b}$ and $\E[(a,Z)(b,Z)] = 0$.
Then a simple scaling argument shows that
\begin{equs}
\E\Big[ &
\prod_{i=1}^M \vscal{a_i,Z}\prod_{j=1}^N \vscal{Z,b_j} 
\Big]
=
\frac{1}{\pi^{\dim_\C(H)}}
\int_{H} 
\prod_{i=1}^M \vscal{a_i,\varphi}\prod_{j=1}^N \vscal{\varphi,b_j} e^{-|\varphi|^{2}} \mrd \varphi\\
{}&=
\frac{1}{\pi^{\dim_\C(H)}}
\int_{0}^{\infty}
r^{N + M + 2 \dim_\C(H) - 1} e^{-r^2}\  \mrd r \enskip \times
\int_{\hat{H}}  \prod_{i=1}^M \vscal{a_i,\hat{\varphi}}\prod_{j=1}^N \vscal{\hat{\varphi},b_j}\; \mrd \hat{\varphi} \;.
\end{equs}
The desired result then follows by using Wick's rule for complex Gaussians on the left-hand side and observing that
\[
\int_{0}^{\infty} 
r^{2N + 2 \dim_\C(H) - 1} e^{-r^2}\ \mrd r 
= (N+\dim_\C(H) -1 )!/2\;.
\]
\end{proof}
We now introduce some definition and notation for graphs that appear in our expansion. 

\begin{definition}\label{def:OG}
Give a finite set $V$, an \textit{oriented graph} on $V$ is an oriented multi-graph with vertex set $V$ and set of edges $E$ and where we concretely realize $E \subset V^2$, enforcing that, for every $e \in E$, $e = (\underline{e},\overline{e})$. 
Note that, once $V$ is fixed, we often identify the oriented graph $G$ on $V$ with the choice of edge set $G=E \subset V^2$.
\end{definition}

Given three disjoint sets $A,B,C$, we define $\hat{\mcb{G}}_{\C}(A,B,C)$ to be the set of all oriented graphs\footnote{Given two distinct vertices  $b_1, b_2 \in B$, there could be two edges with end points $\{b_1,b_2\}$. However, we use still call this an oriented \textit{graph} as opposed to \textit{multi-graph} since these two edges would have different orientations.}  on $V = A \sqcup B \sqcup C$ where we enforce that 
each vertex in $A$ is incident only to a single outgoing edge, each vertex in $C$
is incident only to a single incoming edge, and each vertex in $B$ is incident to precisely one outgoing edge and one incoming edge (we allow these edges to be the same, that is we allow a vertex in $B$ to have a self-loop).
See Figure~\ref{fig:Graph} for an example.

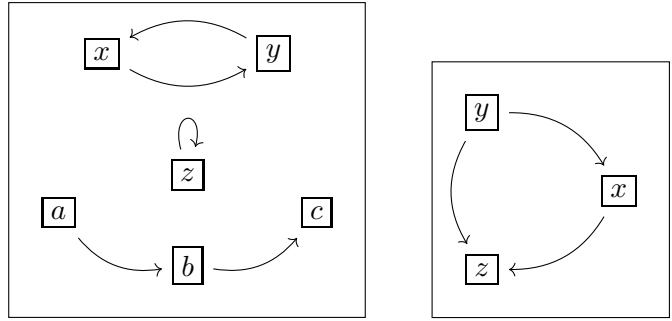
\begin{figure}[htbp]
\centering
\begin{tikzpicture}[scale=0.6]
    \node (x) at (20:2) {$\boxed{y}$};
    \node (y) at (160:2) {$\boxed{x}$};
    \node (z) at (270:2) {$\boxed{z}$};

    \draw[->] (x) to[bend right] (y);
    \draw[->] (y) to[bend right] (x);
    \draw[->] (z) edge[loop above] (z);

    \node (a) at (-135:4) {$\boxed{a}$};
    \node (b) at (-90:4) {$\boxed{b}$};
    \node (c) at (-45:4) {$\boxed{c}$};

    \draw[->] (a) to[bend right] (b);
    \draw[->] (b) to[bend right] (c);

    \node[draw, fit=(x) (a) (b) (c), inner sep=0.3cm] {};
\end{tikzpicture}
\qquad
\begin{tikzpicture}[scale=0.6]
    \node (x) at (0:2) {$\boxed{x}$};
    \node (y) at (120:2) {$\boxed{y}$};
    \node (z) at (240:2) {$\boxed{z}$};

    \draw[->] (x) to[bend left] (z);
    \draw[->] (y) to[bend left] (x);
    \draw[->] (y) to[bend right] (z);

    \node[draw, fit=(x) (y) (z), inner sep=0.3cm] {};
\end{tikzpicture}
\caption{In the left box is an example of $G\in \hat{\mcb{G}}_{\C}(A,B,C)$ with $A=\{a\}$, $B=\{b,x,y,z\}$, and $C=\{c\}$, with arrows indicating edges.
In the right box is an example of a oriented multi-graph $G$ on $B=\{x,y,z\}$ which is \emph{not} an element of $\hat{\mcb{G}}_{\C}(\emptyset,B,\emptyset)$
because the vertex $z$ is incident to two incoming edges and $y$ is incident to two outgoing edges.}
\label{fig:Graph}
\end{figure}

It follows that every $G \in \hat{\mcb{G}}_{\C}(A,B,C)$ consists of a collection of oriented paths (with starting vertex in $A$ and ending vertex in $C$) and closed oriented loops (with all vertices in $B$), and every vertex is incident to either a single path or a single loop. 
We call such a graph an oriented contraction graph. 

We define $\mathcal{P}(G)$ to be the set of oriented paths in the graph $G$ and $\mathcal{L}(G)$ to be the set of oriented loops. 
\begin{remark}
To give some context and motivation, vertices in a contraction graph $G$ will later index paths in the original graph $\mbbG$.
An edge in $G$ will indicate concatenation and thus a path (resp. loop) in $G$ will correspond to a `path of paths' (resp. `loop of paths') in $\mbbG$ - see the discussion before Lemma~\ref{lem:full_loop_expansion}.
\end{remark}
Given a tuple of vectors $v = (v_{u}: u \in A \sqcup C)$ with $v_{u} \in H$ along with a tuple of linear operators $M = (M_{b} : b \in B)$ with $M_{b} \in L(H,H)$ we define,  for $\gamma \in \mathcal{P}(G)$ and $\ell \in \mathcal{L}(G)$, 
\begin{equ}\label{eq:path_loop_observables}
\gamma(M,v)
=
\bvscal{ v_{\underline{\gamma}}, \Big(\prod_{b \in \mathring{\gamma}} M_{b} \Big) v_{\overline{\gamma}}}
\quad
\text{and}
\quad
\ell(M) = \Trace \Big( \prod_{b \in \ell}  M_{b} \Big)\;.
\end{equ}
Above, given $\gamma \in \mathcal{P}(G)$, we write $\underline{\gamma} \in A$ for the starting vertex of $\gamma$ and $\overline{\gamma} \in C$ for the ending vertex. 
We write $\mathring{\gamma}$ for the (possibly empty) set of intermediate vertices of $B$ visited by $\gamma$, and the order in  $\prod_{b \in \mathring{\gamma}}$ is determined by $\gamma$. 
Similarly, given a loop $\ell \in \mathcal{L}(G)$ the product $\prod_{b \in \ell}$ is an ordered (up to cyclic permutation) product over all the vertices in $\ell$. 
Also recall that the $\big( \cdot, \cdot \big)$ appearing in the definition of $\gamma(M,v)$ denotes the inner product on $H$, in particular here $\gamma(\cdot,\cdot)$ and $\ell(\cdot)$ take values in $\mathbb{C}$. 
%
%
\begin{lemma}\label{lem:integrate_higgs_onesite}
With the notation introduced above, we have
\begin{equs}[eq:integrating_higgs_onesite_complex]
\int_{H} &
\Big( 
\prod_{a \in A} \vscal{v_{a},\varphi}
\Big)
\Big(
\prod_{b \in B} \vscal{\varphi, M_{b} \varphi}
\Big)
\Big(
\prod_{c \in C} \vscal{\varphi, v_{c}}
\Big)\; \mrd \rho^{\lambda}(\varphi) \\
{}&=
C^{\C,\lambda}_{|A| + |B|}
\sum_{G \in \hat{\mcb{G}}_{\C}(A,B,C)}
\prod_{\gamma \in \mathcal{P}(G)}
\gamma(M,v)
\prod_{\ell \in \mathcal{L}(G)}
\ell(M)\;,
\end{equs}
where, for $j \in \N$,
\[
C^{\C,\lambda}_{j} = K^{\C}_{j} \int_{0}^{\infty} s^{2j}\  \mrd \lambda(s)\;.
\]  
\end{lemma}
\begin{proof}
We fix an orthonormal basis $(e_{i})_{i=1}^{k}$ of $H$.
Without loss of generality, we may assume that, for $u \in A \sqcup C$, $v_{u} = e_{i_{u}}$ for some $1 \le i _{u} \le k$. 
We can then write the left-hand side of \eqref{eq:integrating_higgs_onesite_complex}  as 
\begin{equs}
\sum_{ \substack{(i_{b}: b \in B)
\\
1 \le i_{b} \le k
}} &
\int_{0}^{\infty} 
\int_{\hat{H}}
s^{|A|+2|B| + |C|}
\prod_{a \in A} \vscal{e_{i_{a}},  \hat{\varphi} }
\prod_{b \in B} \vscal{e_{i_b} ,\hat{\varphi}}
\\
{}&
\qquad \qquad \times
\prod_{b \in B} \vscal{\hat{\varphi}, M_{b} e_{i_{b}}}
\prod_{c \in C} \vscal{ \hat{\varphi},  e_{i_{c}}}\
\mrd \hat{\varphi}\ 
\mrd \lambda(s)\;.
\end{equs}
Above, we wrote $\varphi = s \hat{\varphi}$ and decomposed the $\varphi$ integration into integration over $s$ and $\hat{\varphi}$. 
We also expanded each factor of $\vscal{ \hat{\phi}, M_{b} \hat{\phi}}$ using our orthonormal basis. 
 
We can then apply Lemma~\ref{lemma:spherical_integral} to perform the integral over $\hat{\varphi}$ which gives us
\begin{equs}
K^{\C}_{|A|+|B|}
\sum_{ \substack{\sigma: A \sqcup B \rightarrow B \sqcup C 
\\
\mathrm{bijection}}}
\Bigg[
\sum_{
\substack{
(i_{b}: b \in B)
\\
1 \le i_{b} \le k
}}
\prod_{u \in A \sqcup B} \vscal{e_{i_{u}} ,w_{\sigma(u)}}
\Bigg]\;,
\end{equs}
where $w_{b} = M_b e_{i_{b}}$ for $b\in B$ and $w_{c} = e_{i_{c}}$ for $c\in C$.

Note that the sum above is empty unless $|A| = |C|$. 
The bijections $\sigma\colon A \sqcup B \rightarrow B \sqcup C$ are in one-to-one correspondence with graphs in $\hat{\mcb{G}}_{\C}(A,B,C)$ - given $\sigma$ we define $G(\sigma)$ to have an edge going from $u \in A \sqcup B$ to $v \in B \sqcup C$ if and only if $\sigma(u) = v$. 
Upon fixing $\sigma$, the summand in brackets above factorises as a product indexed by paths in $\mathcal{P}(G(\sigma))$ and loops in $\mathcal{L}(G(\sigma))$. 

Given $\gamma \in \mathcal{P}(G(\sigma))$, if $\gamma$ visits the sites $a,b_{1},b_{2},\dots,b_{n},c$ in that order, then we get a corresponding factor
\begin{equs}
\sum_{i_{b_{1}}=1}^{k} \cdots \sum_{i_{b_{n}}=1}^{k}&
\vscal{e_{i_{a}},M_{b_{1}} e_{i_{b_{1}}}} 
\vscal{e_{i_{b_{1}}}, M_{b_{2}} e_{i_{b_{2}}}}
\cdots
\vscal{e_{i_{b_{n-1}}}, M_{b_{n}} e_{i_{b_{n}}}}
\vscal{e_{i_{b_{n}}},e_{i_{c}}}\\
{}&=
\vscal{e_{i_{a}},M_{b_1} \cdots M_{b_{n}} e_{i_{c}}}\;.
\end{equs}

Given $\ell \in \mathcal{L}(G(\sigma))$, if we freeze a basepoint $b_{1}$ so the loops visits $b_{1},\dots,b_{n}$, then we get a corresponding factor
\begin{equs}
\sum_{i_{b_{1}}=1}^{k} \cdots \sum_{i_{b_{n}}=1}^{k} &
\vscal{e_{i_{b_{1}}}, M_{b_{2}} e_{i_{b_{2}}}}
\vscal{e_{i_{b_{2}}}, M_{b_{3}} e_{i_{b_{3}}}}
\cdots
\vscal{e_{i_{b_{n-1}}}, M_{b_{n}} e_{i_{b_{n}}}}
\vscal{ e_{i_{b_{n}}}, M e_{i_{b_{1}}}}\\
{}& =
\Trace \big(  M_{b_{1}} \cdots M_{b_{n}} \big)\;.
\end{equs}
\end{proof}
\subsubsection{Typed oriented graph isomorphisms}\label{subsec:tog-iso}
It will be convenient to use a ``multiset'' formulation of Lemma~\ref{lem:integrate_higgs_onesite}.
\begin{notation}\label{not:multisets}
Given a set $\mcb{S}$, a multiset $\mathcal{A}$ of elements from $\mcb{S}$ is a collection of elements of $\mcb{S}$ where we allow elements to appear with multiplicity. 
The set of multisets of $\mcb{S}$ is then given by $\N^{\mcb{S}}$. 
Equivalently, $\mathcal{A} \in \N^{\mcb{S}}$ can be realized as an abstract set $A$ with a map $\iota: A \rightarrow \mcb{S}$ that associates to each element  $a \in A$ a choice of \textit{type} in $\iota(a) \in \mcb{S}$ with the constraint that for each $s \in \mcb{S}$, $| \{a \in A: \iota(a) = s\}| = \mathcal{A}_{s}$. 
We call $(A,\iota)$ a typed set.
We also often just write $\CA$ instead of $A$.

Given $\mathcal{A} = (\mathcal{A}_{s}:s \in \mcb{S}) \in \N^{\mcb{S}}$ we write $|\mathcal{A}| \eqdef \sum_{s \in \mcb{S}} |\mathcal{A}_{s}|$ for the size of $\mathcal{A}$ and $\mathcal{A}! = \prod_{s \in \mcb{S}} \mathcal{A}_{s}!$ . 
We will also write sums and products indexed by multisets, in particular a sum or product over $\mathcal{A} \in \N^{\mcb{S}}$ will be a sum or product of $|\mathcal{A}|$ summands or factors.  
We write $\multiset{\mcb{S}}$ for the collection of all finite multisets of elements of $\mcb{S}$, that is $\multiset{\mcb{S}} = \{ \mathcal{A} \in \N^{\mcb{S}}: |\mathcal{A}| < \infty\}$. 
Given $\mcb{S}' \supset \mcb{S}$, we view $\multiset{\mcb{S}'} \supset \multiset{\mcb{S}}$ by just imposing that any $\CA \in \multiset{\mcb{S}}$ has $\CA_{s'} = 0$ for every $s' \in \mcb{S}' \setminus \mcb{S}$. 
Finally, given $\mathcal{A}, \mathcal{B} \in \multiset{\mcb{S}}$, we define $\mathcal{A} \sqcup \mathcal{B} \in \multiset{\mcb{S}}$ by setting $( \mathcal{A} \sqcup \mathcal{B} )_{s} = \mathcal{A}_{s} + \mathcal{B}_{s}$ for each $s \in \mcb{S}$. 
\end{notation}

We will work with oriented graphs on multi-sets, which leads us to the notion of typed oriented graph. 

\begin{definition}
A typed oriented graph $G$ is an oriented graph on a typed set $(V,\iota)$. 
\end{definition}

We introduce disjoint sets of types $\mcb{A}$, $\mcb{B}$, $\mcb{C}$.

\begin{notation}\label{rem:graphs_on_multiset}
We now ``overload'' the notation $\hat{\mcb{G}}_{\C}(\cdot,\cdot,\cdot)$ that originally took sets as arguments to also take multi-sets as arguments.
For any $\mathcal{A} \in \multiset{\mcb{A}}$, $\mathcal{B} \in \multiset{\mcb{B}}$, and $\mathcal{C} \in \multiset{\mcb{C}}$, we  write $\hat{\mcb{G}}_{\C}(\mathcal{A},\mathcal{B},\mathcal{C})$ 
for the set of all typed oriented graphs with vertex set $\mathcal{A} \sqcup \mathcal{B} \sqcup \mathcal{C}$ where each vertex in $\mathcal{A}$ is incident only to a single outgoing edge, each vertex in $\mathcal{C}$
is incident only to a single incoming edge, and each vertex in $\mathcal{B}$ is incident to precisely one outgoing edge and one incoming edge (as before, we allow a vertex in $\mathcal{B}$ to have a self-loop). 

Recall that the multi-sets $\mathcal{A}$, $\mathcal{B}$, and $\mathcal{C}$ can be associated to some fixed abstract sets $A,B.C$ along with a type map $\iota$ on $A \sqcup B \sqcup C$ as described in Notation~\ref{not:multisets}. 
Once this has been done, the data constituting an element  $G \in \hat{\mcb{G}}_{\C}(\mathcal{A},\mathcal{B}, \mathcal{C})$ is simply an element of $ \tilde{G} \in \hat{\mcb{G}}_{\C}(A,B,C)$ along with the type map $\iota$ given above.  
\end{notation}

We will want to count symmetries of typed oriented graphs and have a notion of typed oriented graphs being identical at the level of types -- this motivates the definitions below. 

\begin{definition}
Given a typed set $(V,\iota)$, a type permutation of $(V,\iota)$ is a bijection $f\colon V \rightarrow V$ satisfying $\iota \circ f = \iota$. 
\end{definition}

Note that any bijection $f\colon V \rightarrow V$ induces a corresponding bijection $\mathbf{f}$ on the set $V \times V$ by setting $e = (\underline{e},\overline{e}) \mapsto \mathbf{f}(e) = (f(\underline{e}),f(\overline{e}))$. 

\begin{definition}
Given two typed oriented graphs $G_1, G_2 \subset V \times V$ on the typed vertex set $(V,\iota)$ we say that a type permutation $f$ is a tog\footnote{\ ``tog'' is just an abbreviation for typed oriented graph.}--isomorphism if the induced bijection $\mathbf{f}$ on $V \times V$ induced by $f$ maps $G_1$ to $G_2$.
If such a tog-isomorphism exists we call $G_1$ and $G_2$ tog-isomorphic. 

In the case where $G_1 = G_2= G$, we call $f$ a tog-automorphism of $G$. 
\end{definition}
Given $G \in \hat{\mcb{G}}_{\C}(A,B,C)$, we denote by $S(G)$ the number of tog-automorphisms on $G$. 
We also define $\mcb{G}_{\C}(\mathcal{A},\mathcal{B},\mathcal{C})$ to be the set of tog-isomorphism classes in $\hat{\mcb{G}}_{\C}(\mathcal{A},\mathcal{B},\mathcal{C})$.

\begin{remark}\label{rem:graph_overload}
Note that we often use the same character $G$ to either reference a ``concrete'' graph $G \in \hat{\mcb{G}}_{\C}(\mathcal{A},\mathcal{B},\mathcal{C})$ or an tog-isomorphism class of concrete graphs of $G \in   \mcb{G}_{\C}(\mathcal{A},\mathcal{B},\mathcal{C})$ --  in any discussion, which type of object we are working with should always be clear. 

We will make several (minor) overloads of notations defined for elements of $ \hat{\mcb{G}}_{\C}(\mathcal{A},\mathcal{B},\mathcal{C})$ and $ \mcb{G}_{\C}(\mathcal{A},\mathcal{B},\mathcal{C})$. 
For instance, for $G \in \hat{\mcb{G}}_{\C}(\mathcal{A},\mathcal{B},\mathcal{C})$ we defined $S(G)$ to be the number of tog-automorphisms of $G$. 
On the other hand, for $G \in \mcb{G}_{\C}(\mathcal{A},\mathcal{B},\mathcal{C})$ we define $S(G)=S(\hat{G})$ for any $\hat{G} \in  \hat{\mcb{G}}_{\C}(\mathcal{A},\mathcal{B},\mathcal{C})$ belonging to the tog-isomorphism class $G$ - it is easy to see this is indeed well-defined.
\end{remark} 
\begin{example}\label{ref:example_graphs}
Suppose $\mcb{A}=\mcb{C}=\emptyset$ (thus $\mathcal{A}=\mathcal{C}=\emptyset$) and
$\mcb{B}=\{1,2\}$.
\begin{enumerate}[label=(\roman*)]
\item \label{pt:1st_eg} Suppose $\mathcal{B}=\{x,y,z\}$ where $x,y,z \in \{1\}$ are of the same type.
Consider the graph $G\in \hat{\mcb{G}}_{\C}(\mathcal{A},\mathcal{B},\mathcal{C})$ consisting of the single oriented loop\footnote{ Identifying $G$ with its edge set, we would write $G =\{ (x,y),(y,z),(z,x)\}$.} $(xyz)$.
Then $S(G)=3$ since $x$ can be mapped to any of $x,y,z$ in a tog-automorphism, and this determines the tog-automorphism uniquely.
The total number of oriented graphs tog-isomorphic to $G$ is $2$, the other graph being the loop $(xzy)$.
See Figure~\ref{ref:fig1}.

\item \label{pt:2nd_eg} Suppose $\mathcal{B}=\{x,y,z\}$ but now $x\in \{1\}$ and $y,z\in\{2\}$. Consider $G$ as in~\ref{pt:1st_eg}. Then $S(G)=1$ and the total number of graphs tog-isomorphic to $G$ is again $2$.
See Figure~\ref{ref:fig2}.

\item \label{pt:3rd_eg} Consider $\mathcal{B}=\{w,x,y,z\}$ with $x,w\in\{1\}$ 
and $y,z\in \{2\}$.
Consider the graph $G\in \hat{\mcb{G}}_{\C}(\mathcal{A},\mathcal{B},\mathcal{C})$ consisting of the single oriented loop $(wyxz)$.
Then $S(G)=2$ because we can simultaneously interchange $w\leftrightarrow x$ and $y\leftrightarrow z$ in a tog-automorphism for $G$. 
The number of graphs tog-isomorphic to $G$ is $2$, with the second one being $(wzxy)$.
See Figure~\ref{ref:fig3}.

\item Consider the same setting as in~\ref{pt:3rd_eg} but now with the graph $G$ consisting of the single oriented loop $(wxyz)$.
Then $S(G)=1$ because any map that interchanges $w\leftrightarrow x$ (resp. $y\leftrightarrow z$) will change the oriented structure of the graph.
The number of graphs tog-isomorphic to $G$ is $4$, the other three being the loops
$(xwyz),(xwzy),(wxzy)$.
\end{enumerate}
\end{example}
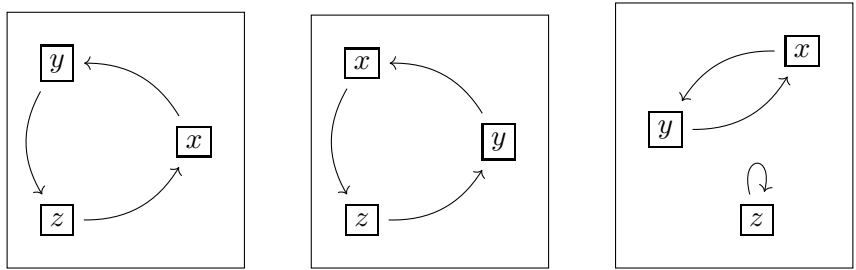
\begin{figure}[htbp]
\centering
\begin{tikzpicture}[scale=0.6] 
    \node (x) at (0:2) {$\boxed{x}$};
    \node (y) at (120:2) {$\boxed{y}$};
    \node (z) at (240:2) {$\boxed{z}$};

    \draw[->] (x) to[bend right] (y);
    \draw[->] (y) to[bend right] (z);
    \draw[->] (z) to[bend right] (x);
    \node[draw, fit=(x) (y) (z), inner sep=0.3cm] {};
\end{tikzpicture}
\qquad
\begin{tikzpicture}[scale=0.6]
    \node (y) at (0:2) {$\boxed{y}$};
    \node (x) at (120:2) {$\boxed{x}$};
    \node (z) at (240:2) {$\boxed{z}$};

    \draw[->] (y) to[bend right] (x);
    \draw[->] (x) to[bend right] (z);
    \draw[->] (z) to[bend right] (y);
    \node[draw, fit=(x) (y) (z), inner sep=0.3cm] {};
\end{tikzpicture}
\qquad
\begin{tikzpicture}[scale=0.6]
    \node (x) at (60:2) {$\boxed{x}$};
    \node (y) at (180:2) {$\boxed{y}$};
    \node (z) at (270:2) {$\boxed{z}$};

    \draw[->] (x) to[bend right] (y);
    \draw[->] (y) to[bend right] (x);
    \draw[->] (z) edge[loop above] (z);
    \node[draw, fit=(x) (y) (z), inner sep=0.3cm] {};
\end{tikzpicture}
\caption{In the first box we draw $G$ from Example~\ref{ref:example_graphs}\ref{pt:1st_eg}. 
We use shapes around vertices to indicate type, the squares indicating all vertices are type $1$. 
Here $S(G)=3$, the tog-automorphisms arise from cyclic permutations on $xyz$.  
The second graph $\tilde{G}$ is distinct from $G$, but tog-isomorphic to $G$ - they are indistinguishable after removing Roman letter labels on vertices and keeping only the squares.  
The third graph $\bar{G}$ is not tog-isomorphic to either $G$ or $\tilde{G}$.}
\label{ref:fig1}
\end{figure}

\begin{figure}[htbp]
\centering
\begin{tikzpicture}[scale=0.6] 
    \node (x) at (0:2) {$\boxed{x}$};
    \node (y) at (120:2) {$\begin{tikzpicture}[baseline=(y.base), scale=0.5]
        \node (y) [draw, regular polygon, regular polygon sides=3, inner sep=0.6pt] {$y$};
        \end{tikzpicture}$};
    \node (z) at (240:2) {$\begin{tikzpicture}[baseline=(z.base), scale=0.5]
        \node (z) [draw, regular polygon, regular polygon sides=3, inner sep=1pt] {$z$};
        \end{tikzpicture}$};

    \draw[->] (x) to[bend right] (y);
    \draw[->] (y) to[bend right] (z);
    \draw[->] (z) to[bend right] (x);
    \node[draw, fit=(x) (y) (z), inner sep=0.3cm] {};
\end{tikzpicture}
\qquad
\begin{tikzpicture}[scale=0.6]
    \node (y) at (0:2) {$\begin{tikzpicture}[baseline=(y.base), scale=0.5]
        \node (y) [draw, regular polygon, regular polygon sides=3, inner sep=0.6pt] {$y$};
        \end{tikzpicture}$};
    \node (x) at (120:2) {$\boxed{x}$};
    \node (z) at (240:2) {$\begin{tikzpicture}[baseline=(z.base), scale=0.5]
        \node (z) [draw, regular polygon, regular polygon sides=3, inner sep=1pt] {$z$};
        \end{tikzpicture}$};

    \draw[->] (y) to[bend right] (x);
    \draw[->] (x) to[bend right] (z);
    \draw[->] (z) to[bend right] (y);
    \node[draw, fit=(x) (y) (z), inner sep=0.3cm] {};
\end{tikzpicture}
\caption{In the left box we draw $G$ from Example~\ref{ref:example_graphs}\ref{pt:2nd_eg}. 
We use shapes at the vertices to indicate type, with squares for type $1$ and triangles for type $2$. 
The only allowable type permutation on vertices swaps $y \leftrightarrow z$ and leaves $x$ fixed - this swap gives the graph $\bar{G}$ on the right
which is distinct from $G$ and it follows that $S(G) = 1$.
$G$ and $\bar{G}$ are tog-isomorphic which can be seen by dropping letters and only keeping the squares and triangles. 
}
%
\label{ref:fig2}
\end{figure}
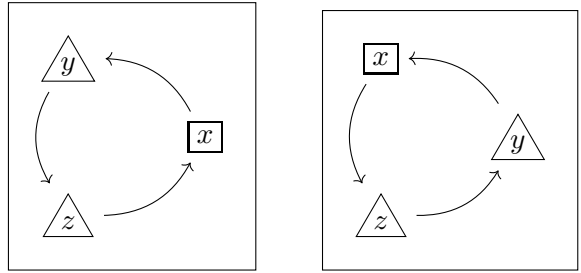

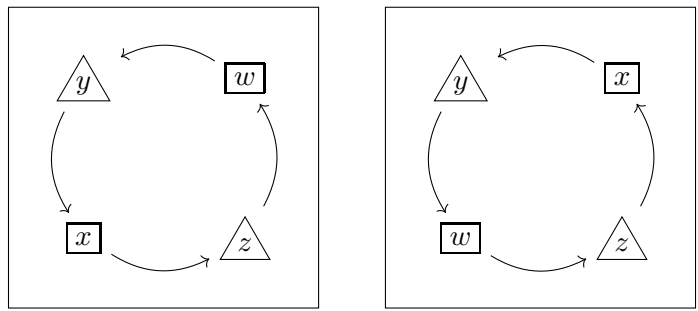
\begin{figure}[htbp]
\centering
\begin{tikzpicture}[scale=0.6]
    \node (w) at (45:2.5) {$\boxed{w}$};
    \node (y) at (135:2.5) {$\begin{tikzpicture}[baseline=(y.base), scale=0.5]
        \node (y) [draw, regular polygon, regular polygon sides=3, inner sep=0.6pt] {$y$};
        \end{tikzpicture}$};
    \node (x) at (225:2.5) {$\boxed{x}$};
    \node (z) at (315:2.5) {$\begin{tikzpicture}[baseline=(z.base), scale=0.5]
        \node (z) [draw, regular polygon, regular polygon sides=3, inner sep=1pt] {$z$};
        \end{tikzpicture}$};

    \draw[->] (w) to[bend right] (y);
    \draw[->] (y) to[bend right] (x);
    \draw[->] (x) to[bend right] (z);
    \draw[->] (z) to[bend right] (w);

    \node[draw, fit=(w) (y) (x) (z), inner sep=0.5cm] {};
\end{tikzpicture}
\qquad
\begin{tikzpicture}[scale=0.6]
    \node (x) at (45:2.5) {$\boxed{x}$};
    \node (y) at (135:2.5) {$\begin{tikzpicture}[baseline=(y.base), scale=0.5]
        \node (y) [draw, regular polygon, regular polygon sides=3, inner sep=0.6pt] {$y$};
        \end{tikzpicture}$};
    \node (w) at (225:2.5) {$\boxed{w}$};
    \node (z) at (315:2.5) {$\begin{tikzpicture}[baseline=(z.base), scale=0.5]
        \node (z) [draw, regular polygon, regular polygon sides=3, inner sep=1pt] {$z$};
        \end{tikzpicture}$};

    \draw[->] (x) to[bend right] (y);
    \draw[->] (y) to[bend right] (w);
    \draw[->] (w) to[bend right] (z);
    \draw[->] (z) to[bend right] (x);

    \node[draw, fit=(w) (y) (x) (z), inner sep=0.5cm] {};
\end{tikzpicture}
\caption{In the left box we draw $G$ from Example~\ref{ref:example_graphs}\ref{pt:3rd_eg}, again drawing squares for type $1$ and triangles for type $2$. 
On the right, we draw the only graph distinct from $G$ which is tog-isomorphic to $G$. }
\label{ref:fig3}
\end{figure}

Given $G \in \mcb{G}_{\C}(\mathcal{A},\mathcal{B},\mathcal{C})$, we have a corresponding notion of tog-isomorphism class of paths and loops in $G$. 
Since paths in $G$ cannot intersect themselves, a tog-isomorphism class of a path in $G$ is just a finite sequence of types of vertices, with the first type taken from $\mcb{A}$, the final type taken from $\mcb{C}$, and the intermediate ones taken from $\mcb{B}$. 
Since loops in $G$ are simple, a tog-isomorphism class of a loop in $G$ is an equivalence class of finite sequences of elements of $\mcb{B}$ where one views finite sequences as equivalent if they differ by a cyclic permutation (this is because loops do not have distinguished basepoints).
We will call a tog-isomorphism class of a path (resp. loop) a type of path (resp. loop). 
We write $\mathcal{P}(G)$ to be the multiset of types of paths in $G$ and $\mathcal{L}(G)$ the multiset of types of loops in $G$ -- note that here we are performing an overload of notation as mentioned in Remark~\ref{rem:graph_overload} as we already used this notation in a different context when instead $G \in \hat{\mcb{G}}_{\C}(\mathcal{A},\mathcal{B},\mathcal{C})$ --
see the discussion following Definition~\ref{def:OG}.

Since such types of loops $\ell$ or paths $\gamma$ can themselves be seen as tog-isomorphism classes of oriented contraction graphs, they are well-defined without needing to make reference to a larger graph $G$. 
Moreover, it makes sense to write $S(\ell)$ and $S(\gamma)$. 
In particular, for any path  $\gamma$ one has $S(\gamma) = 1$ - since $\gamma$ is an oriented path with initial vertex in $\mathcal{A}$ and final vertex in $\mathcal{C}$, its vertices come with a total order which forces $S(\gamma)=1$.
\begin{remark}\label{rem:automorphisms of loops}
In general, for any oriented loop $\ell$ consisting of $n$ vertices, the automorphism group of $\ell$ is a subgroup of the cyclic group $C_n$.

Calling the vertices $\{x_{0},\dots,x_{n-1}\}$ and writing the loop $(x_{0} \cdots x_{n-1})$ with basepoint $x_0$,
each non-trivial automorphism corresponds to a choice of new basepoint $x_j$ such that the new sequence of types that appears in our representation of the loop agrees with the old one, that is for $0 \le k \le n-1$, the type of $x_k$ is the same as that of $x_{k + j \mod n}$.
In particular, if all the vertices are of the same type, then the tog-automorphism group is the full cyclic group $C_n$.
\end{remark} 
An $\mcb{A} \sqcup \mcb{C}$ vector assignment is a tuple $v = (v_{u}: u \in \mcb{A} \sqcup \mcb{C})$ with $v_{u} \in H$, while 
a $\mcb{B}$ operator assignment is a tuple $M = (M_{b} : b \in \mcb{B})$ with $M_{b} \in L(H,H)$.
For $G \in \mcb{G}_{\C}(\mathcal{A},\mathcal{B},\mathcal{C})$
and any $\gamma \in \mathcal{P}(G)$ and $\ell \in \mathcal{L}(G)$, we define, for assignments as given above, $\gamma(M,v)$ and $\ell(M)$ as in~\eqref{eq:path_loop_observables} where the components of  $v$ or $M$ are chosen based on type. 
Observe that $\gamma(M,v)$, $\ell(M)$, and $S(\ell)$ depend only on the type of $\gamma$ and $\ell$.
We also write \[
\hat{\ell}(M) \eqdef \ell(M)/S(\ell)\;.
\]

We can then reformulate Lemma~\ref{lem:integrate_higgs_onesite} as the following. 
\begin{lemma}\label{lemma:integrating_onesite_multiset}
With the notation as given above, for any $\mcb{A} \sqcup \mcb{B}$ vector assignment $v$ and $\mcb{B}$ operator assignment $M$, 
\begin{equs}
{}&\frac{1}{\mathcal{A}!\mathcal{B}!\mathcal{C}!}
\int_{H} 
\Big( 
\prod_{a \in \mathcal{A}} \vscal{v_{a},\phi}
\Big)
\Big(
\prod_{b \in \mathcal{B}} \vscal{\phi, M_{b} \phi}
\Big)
\Big(
\prod_{c \in \mathcal{C}} \vscal{\phi, v_{c}}
\Big)\ 
\mrd \rho^{\lambda}(\phi)\\
{}&=
C^{\C,\lambda}_{|\mathcal{A}| + |\mathcal{B}|}
\sum_{G \in \mcb{G}_{\C}(\mathcal{A},\mathcal{B},\mathcal{C})}
\frac{1}{\mathcal{P}(G)! \mathcal{L}(G)!}
\prod_{\gamma \in \mathcal{P}(G)}
\gamma(M,v)
\prod_{\ell \in \mathcal{L}(G)}
\hat{\ell}(M)\;.
\end{equs}	
\end{lemma}

\begin{proof}
We first claim that every  $G\in\mcb{G}_\C(\mcA,\mcB,\mcC)$ has
\[
\frac{\mcA!\mcB!\mcC!}
{\mcP(G)!\mcL(G)! \prod_{\ell \in \mcL(G)} S(\ell)}
\]
representatives in $\hat{\mcb{G}}_\C(\mcA,\mcB,\mcC)$. 
The factor $\mcA!\mcB!\mcC!$ comes from going from distinguishable to indistinguishable vertices, but this is an overcounting that is corrected by dividing by the number of permutations of paths and loops of the same type and then the number of automorphisms within each loop and path. 
The conclusion then follows by~\eqref{eq:integrating_higgs_onesite_complex}.
\end{proof}

\subsubsection{Lattice paths and loops in the complex case}

Returning to our graph $\mathbb{G}$, we define $\mcb{P}_{\C}$ to be the set of all lattice\footnote{Here and in the real case of Section~\ref{subsec:real_case}, we use the word `lattice' to distinguish paths and loops that traverse $\mbbG$ rather than the loops and paths coming from the application of Lemmas~\ref{lem:integrate_higgs_onesite} and~\ref{lem:integrate_higgs_onesite_real}.} paths. 
In this subsection we enforce that
\begin{enumerate}[label=(\roman*)]
\item every lattice path contains at least one edge (and edges can be traversed multiple times),
\item lattice paths are oriented and they must traverse each edge in the way that respects its orientation, and
\item any lattice path consisting of more than one edge cannot start or end with an element of $\mathbb{E}^{o}$ (a self-loop edge). 
\end{enumerate}
Concretely, we let $\widehat{\mcb{P}}_{\C}$ be the set of all non-empty finite sequences $(e_{1},\dots,e_{n})$ of elements of $\mathbb{E}$ such that, for $1 \le i  \le n-1$ one has $\overline{e_{i}} = \underline{e_{i+1}}$.
We then set $\mcb{P}_{\C}$ to be the set all of sequences in $\widehat{\mcb{P}}_{\C}$ that are either of length $1$ or  with the property that they neither start nor end with an element of $\mathbb{E}^{o}$.
Any $\gamma\in \widehat{\mcb{P}}_\C$ has starting and ending vertices denoted $\underline{\gamma}, \overline{\gamma} \in  \mbbV$ respectively.

We define $\mcb{L}_{\C}$ to be the set of all lattice loops.
In this subsection
\begin{enumerate}[label=(\roman*)]
\item every lattice loop consists of at least one edge (and edges can be traversed multiple times),
\item lattice loops do not have a distinguished basepoint, and
\item lattice loops are oriented, in particular they traverse edges in $\mbbG$  respecting their orientation. 
\end{enumerate}
Concretely, $\mcb{L}_{\C}$ is the set of equivalence classes of $\widetilde{\mcb{P}}_{\C} =  \{ \gamma \in \widehat{\mcb{P}}_{\C}: \underline{\gamma} = \overline{\gamma}\}$ when one quotients by cyclic permutations.

\begin{remark}\label{rem:self_loop}
A self-loop  $e \in \mbbE^{o}$ determines a lattice path $\gamma = (e) \in \mcb{P}_{\C}$ of length $1$ and additionally, for all $n \ge 1$, we have a  
corresponding lattice loop $\ell$ of length $n$ which is given by an equivalence class with a single element, that is $\ell = \{ \underbrace{(e,\dots,e)}_{n \text{ times}} \}$. 
We treat the above lattice path and lattice loop of length $1$ as distinct. 
Writing $\underline{e} = \overline{e} = a \in \mbbV$, in terms of the integrand in Lemma~\ref{lem:full_loop_expansion} below,
the lattice path above corresponds to a factor $\gamma(M,\phi) = (\phi_{a},M_{e} \phi_{a})$ whereas the lattice loop of length $n$ corresponds to a factor of $\ell(M) = \Trace(M_{e}^n)$.
\end{remark}

We now describe how to define symmetry factors $S(\gamma)$ and $S(\ell)$ for $\gamma \in \widehat{\mcb{P}}_{\C}$ and $\ell \in \mcb{L}_{\C}$ analogous to those given for paths and loops in graphs of $\hat{\mcb{G}}_{\C}(\mathcal{A},\mathcal{B},\mathcal{C})$ earlier. 
To do this, we associate to each $\eta \in \widehat{\mcb{P}}_{\C} \sqcup \mcb{L}_{\C}$ of length $n \ge 1$ a typed oriented graph $\mcb{g}(\eta)$ on the vertices $V = \{1,\dots, n\}$ where each vertex $v \in V$ is given a type  $\iota(i) \in \mbbE$.
We then set $S(\eta)$ to be the number of tog-automorphisms on $\mcb{g}(\eta)$ that preserve the types of vertices.

For $\gamma = (e_{1},\dots,e_{n}) \in \mcb{P}_{\C}$, we take $\mcb{g}(\gamma)$ to have oriented edges $\{(i,i+1)\}_{i=1}^{n-1}$ and assign, for each $1 \le i \le n$, the type of $i$ to be $s(i) = e_{i}$.
Note that we always have $S(\gamma) \eqdef S(\mcb{g}(\gamma))= 1$. 
 
For $\ell \in \mcb{L}_{\C}$, we choose some representative $ \widetilde{\mcb{P}}_{\C} \ni \gamma \in \ell$ in the equivalence class and define $\mcb{g}(\ell)$ to be the typed oriented graph obtained by adding to the oriented graph $\mcb{g}(\gamma)$ the oriented edge $(n,1)$. 
Note that the resulting value of $S(\ell)\eqdef S(\mcb{g}(\ell))$ is well-defined as it is independent of our choice of $\gamma$.

Note that the group of tog-automorphisms of $\ell$ is a subgroup of $C_n$ (cf.~Remark~\ref{rem:automorphisms of loops})
and it is possible that $S(\ell) \not = 1$. 
As an example, suppose we fix $\ell \in \mcb{L}_{\C}$ that contains $\gamma = (e_1,e_2,e_3,e_4,e_5,e_6) \in \widetilde{\mcb{P}}_{\C}$ with $e_1=e_{4}=(x,y) \in \mathbb{E}$, $e_2=e_5=(y,z) \in \mathbb{E}$, and $e_3=e_6=(z,x) \in \mathbb{E}$, where $x,y,z$ are distinct vertices.
Then $S(\ell) = 2$: one has the trivial tog-automorphism and the tog-automorphism that interchanges $e_{i} \leftrightarrow e_{i + 3}$ for $i \in \{1,2,3\}$.

\subsubsection{Loop expansion in complex case}

Recalling Notation~\ref{not:multisets}, we define $\mathbf{G}_{\C} = \multiset{\mcb{P}_{\C} \sqcup \mcb{L}_{\C}}$ to be the set of all finite multisets of lattice paths and loops. 
Given $G \in \mathbf{G}_{\C}$, we write\footnote{Note that we are doing a third overload of notation when writing $\mathcal{P}(G)$ and $\mathcal{L}(G)$, but point out the meaning of the notation is clear upon remembering whether $G \in \hat{\mcb{G}
}_{\C}(\mathcal{A},\mathcal{B},\mathcal{C})$, $G \in  \mcb{G}_{\C}(\mathcal{A},\mathcal{B},\mathcal{C})$ or $G \in \mathbf{G}_{\C}$.} $\mathcal{P}(G) \in \multiset{\mcb{P}_{\C}}$ for the multiset of paths in $G$ and $\mathcal{L}(G) \in \multiset{\mcb{L}_{\C}}$ for the multiset of loops in $G$. 
In particular, $G = \mathcal{P}(G) \sqcup \mathcal{L}(G)$. 

For $G \in \mathbf{G}_{\C}$ and $x \in \mbbV$, we write $G_{x} \in \N$ for the number of incidences $x$ has with edges in $G$ with multiplicity,
that is
\begin{equ}[eq:def_G_x]
G_{x} = 
\Bigg( 
\sum_{\gamma \in \mathcal{P}(G)} 
\sum_{e \in \gamma} 
+
\sum_{\ell \in \mathcal{L}(G)} 
\sum_{e \in \ell}
\Bigg) 
\big( \mathbf{1}\{\underline{e} = x\} +  \mathbf{1}\{\overline{e} = x\}\big)\;.
\end{equ}
Note that $G_x$ is well-defined in the sense that it is independent of the representative $\ell$ taken above for each element in $\mcL(G)$.
Note also that if $x$ is not the endpoint of any path in $G$, then we must have $G_{x} \in 2 \N$. 

\begin{definition}
An $\mathbb{E}$ operator assignment is a family of operators $M = (M_{e})_{e \in \mbbE}$ with $M_{e} \in L(H,H)$. 
\end{definition}
Our main result on a loop expansion in the complex case is the following. 

\begin{theorem}\label{thm:loop_expansion} 
For any $\mathbb{E}$ operator assignment $M$, 
\begin{equs}
Z(M) &\eqdef 
\int_{H^{\mbbV}}
\exp\Big(
\sum_{e \in \mbbE} \vscal{\phi_{\underline{e}},M_{e}\phi_{\overline{e}}}
 \Big)\prod_{x\in\mbbV} \mrd \rho^{\lambda_x}(\phi_x)\\
{}& \quad =
\sum_{\mathcal{L} \in \mathbf{L}_{\C}}
\frac{1}{\mathcal{L}!}
\Big(
\prod_{x \in \mbbV}
C^{\C,\lambda_{x}}_{\mathcal{L}_{x}/2}
\Big)
\prod_{\ell \in \mathcal{L}}
\hat{\ell}(M)\;,
\end{equs}
where $\mathbf{L}_{\C} = \multiset{\mcb{L}_{\C}}$ and $\hat{\ell}(M) =  \Trace \Big( \prod_{e \in \ell} M_{e} \Big)/S(\ell)$ with the product over edges ordered according to the orientation of the lattice loop.
\end{theorem}
Note that, for $\mathcal{L}_{x}\in \mathbf{L}_{\C}\subset \mathbf{G}_{\C}$ we always have $\mathcal{L}_{x} \in 2 \N$.
We will prove Theorem~\ref{thm:loop_expansion} using induction where we iteratively delete/integrate out sites in $\mbbV$.

Given $\bar{\mbbV} \subset \mbbV$, we define $\mathbf{G}_{\C}(\bar{\mbbV}) = \multiset{\mcb{P}_{\C}(\bar{\mbbV}) \sqcup \mcb{L}_{\C}(\bar{\mbbV})}\subset \mathbf{G}_{\C}$ where 
\begin{itemize}
\item $\mcb{P}_{\C}(\bar{\mbbV}) \subset \mcb{P}_{\C}$ consists of all lattice paths that have all of their intermediate vertices belonging to the set $\bar{\mathbb{V}}$ and both starting and endings vertices in $\mbbV \setminus \bar{\mathbb{V}}$
\item $\mcb{L}_{\C}(\bar{\mbbV}) \subset \mcb{L}_{\C}$ consists of all lattice loops that only traverse vertices in $\bar{\mathbb{V}}$. 
\end{itemize}
The set $\bar{\mathbb{V}}$ should be thought of as the set of integrated vertices. 
Note that for any $G \in \mathbf{G}_{\C}(\bar{\mbbV})$ and $x \in \bar\mbbV$ we must have $G_{x} \in 2 \N$. 

The core of the induction is to use Lemma~\ref{lemma:integrating_onesite_multiset} as a ``re-routing lemma''.
The sets of types  $\mcb{A}$, $\mcb{B}$, and $\mcb{C}$ will be subsets of lattice paths in $\mbbG$.
When integrating out the spin at $z \in \mbbV$, we will let $\CA$ be a multiset of lattice paths terminating at $z$ but not starting from $z$, let $\CB$ be a multiset of lattice paths both starting and ending at $z$, and finally let $\CC$ be a multiset of lattice paths that start at $z$ but do not terminate at $z$. 

We now give a lemma which contains the inductive argument and which proves Theorem~\ref{thm:loop_expansion} by taking $\bar{\mbbV} = \mbbV$.

\begin{lemma}\label{lem:full_loop_expansion}
For any $\mathbb{E}$ operator assignment $M$ and $\bar{\mbbV} \subset \mbbV$, we have 
\begin{equs}\label{eq:inductive_complex_case}
\int_{H^{\mbbV}}&
\exp\Big(
\sum_{e \in \mbbE} \vscal{\phi_{\underline{e}},M_{e}\phi_{\overline{e}}}
 \Big)\ 
\prod_{x\in\mbbV} \mrd  \rho^{\lambda_{x}} (\phi_x)
\\
=&
\sum_{ G \in \mathbf{G}_{\C}(\bar{\mbbV})}
\frac{
\prod_{x \in \bar{\mbbV}}
C^{\C,\lambda_{x}}_{G_{x}/2}}
{G!}
\prod_{\ell \in \mathcal{L}(G)}
\hat{\ell}(M)
\int_{H^{\mbbW}}
\prod_{\gamma \in \mathcal{P}(G)} \gamma(M,\phi)\;
\prod_{x\in \mbbW}  \mrd \rho^{\lambda_x}(\phi_x) \;,
\end{equs}
where we write $\mbbW = \mbbV \setminus \bar{\mbbV}$, and 
\[
 \gamma(M,\phi)  =  \vscal{\varphi_{\underline{\gamma}}, \prod_{e \in \gamma} M_{e} \phi_{\overline{\gamma}}}\;.
\]
\end{lemma}
\begin{proof}
The base case of the induction occurs when $\bar{\mbbV} = \emptyset$ in which case the identity just follows from expanding the exponential.
This is because elements of $\mathbf{G}_\C(\emptyset)$ have no lattice loops and their paths must be individuals edges in $\mbbE$, so \eqref{eq:inductive_complex_case} becomes 
\begin{equs}
\int_{H^{\mbbV}}&
\exp\Big(
\sum_{e \in \mbbE} \vscal{\phi_{\underline{e}},M_{e}\phi_{\overline{e}}}
 \Big)\ 
\prod_{x\in\mbbV} \mrd  \rho^{\lambda_{x}} (\phi_x)\\
{}&=
\sum_{ \mcb{n} \in \N^{\mathbb{E}}}
\frac{1}{\mcb{n}!}
\int_{H^{\mbbV}}
\prod_{e \in \mbbE}\vscal{\varphi_{\underline{e}},M_{e}\varphi_{\overline{e}}}^{\mcb{n}_{e}}\ 
 \prod_{x\in \mbbV}  \mrd \rho^{\lambda_x}(\phi_x) \;.
\end{equs}

For the inductive step, we fix $\bar{\mbbV} \subset \mathbb{V}$ and $z \in \mbbW = \mathbb{V} \setminus \bar{\mbbV}$. 
We define $\mcb{P}_{i}(z)$ to be the set of all lattice paths in $\mcb{P}_{\C}(\bar{\mbbV})$ that start at $z$ but do not end at $z$, $\mcb{P}_{f}(z)$ to be the set of all lattice paths in $\mcb{P}_{\C}(\bar{\mbbV})$ that end at $z$ but do not start at $z$, and $\mcb{P}_{l}(z)$ to be the set of all lattice paths in $\mcb{P}_{\C}(\bar{\mbbV})$ that both start and end at $z$. 
We write $\mathbf{P}_{i}(z) = \multiset{\mcb{P}_{i}(z)}$ and define $\mathbf{P}_{f}(z)$ and $\mathbf{P}_{l}(z)$ similarly.
We also define $\hat{\mcb{P}}(z)$ as the set of all lattice paths in $\mcb{P}_{\C}(\bar{\mbbV} \sqcup \{z\})$ that pass through $z$ and  $\hat{\mcb{L}}(z)$ to be the set of all lattice loops in $\mcb{L}_{\C}(\bar{\mbbV} \sqcup \{z\})$ that pass through $z$. 
We write $\hat{\mathbf{G}}(z) = \multiset{\hat{\mcb{P}}(z) \sqcup \hat{\mcb{L}}(z)}$. 
To prove our inductive step it then suffices to prove
\begin{equs}
\sum_{ \mathcal{P}_{i} \in \mathbf{P}_{i}(z)}&
\sum_{ \mathcal{P}_{l} \in \mathbf{P}_{l}(z)}
\sum_{ \mathcal{P}_{f} \in \mathbf{P}_{f}(z)}
\frac{1}
{\mathcal{P}_{f}! \mathcal{P}_{l}! \mathcal{P}_{i}!}\\
\times &
\int_{H}
\Big(
\prod_{\gamma^a \in \mathcal{P}_{f}} \gamma^a(M,\varphi)
\Big)
\Big(
\prod_{\gamma^b \in \mathcal{P}_{l}} \gamma^b(M,\varphi)
\Big)
\Big(
\prod_{\gamma^c \in \mathcal{P}_{i}} \gamma^c(M,\varphi)
\Big)\ 
\mrd \rho^{\lambda_z}(\varphi_z) \\
{}&=
\sum_{G \in \hat{\mathbf{G}}(z)}
\frac{C^{\C,\lambda_{z}}_{G_{z}/2}}{G!}
\prod_{\ell \in \mathcal{L}(G)}
\hat{\ell}(M)
\prod_{\gamma \in \CP(G)} \gamma(M,\phi)\;.
\end{equs}
To match the left and right-hand sides of the above equality, we apply Lemma~\ref{lemma:integrating_onesite_multiset} to the left side and give a bijection  $F\colon \mcb{G}_{\C}(z) \rightarrow  \hat{\mathbf{G}}(z)$ where
\[
\mcb{G}_{\C}(z) 
=
\bigsqcup_{\mathcal{P}_{f} \in \mathbf{P}_{f}(z)}
\bigsqcup_{\mathcal{P}_{l} \in \mathbf{P}_{l}(z)}
\bigsqcup_{\mathcal{P}_{i} \in \mathbf{P}_{i}(z)}
\mcb{G}_{\C}(\mathcal{P}_{f},\mathcal{P}_{l},\CP_{i})\;.
\]
In our application of Lemma~\ref{lemma:integrating_onesite_multiset}, we take $\mcb{A} = \mcb{P}_{f}(z)$, $\mcb{B} = \mcb{P}_{l}(z)$, and $\mcb{C} = \mcb{P}_{i}(z)$. 

For our bijection, given $\tilde{G} \in \mcb{G}_{\C}(z)$, we define $F(\tilde{G})$ by turning the loops and paths of $\tilde{G}$ into lattice paths and loops via concatenation.
Since we have imposed orientation, there is no ambiguity in how to concatenate elements of $\mathcal{P}_{l}$. 

To see that $F$ given above is indeed a bijection, first note that given $G \in \hat{\mathbf{G}}(z)$ we can obtain $\mathcal{P}_{f}(G) \in \mathbf{P}_{f}(z)$, $ \mathcal{P}_{l}(G) \in \mathbf{P}_{l}(z)$, and $\mathcal{P}_{i}(G) \in \mathbf{P}_{i}(z)$ by looking at lattice paths and loops in $G$  and ``snipping'' them whenever they pass through $z$. 
Injectivity and surjectivity follow from the fact that there is precisely one tog-isomorphism class $\tilde{G}  \in \mcb{G}_{\C}(\mathcal{P}_{f},\mathcal{P}_{l},\CP_{i})$ with $F(\tilde{G}) = G$. 

The desired equality is then obtained by observing that, for $\tilde{\ell} \in \mathcal{L}(\tilde{G})$, if one forms the lattice loop $\ell  \in \mathcal{L}(G)$ by concatenating the vertices $\gamma \in \tilde{\ell}$, then one has $S(\tilde{\ell}) = S(\ell)$. 
This is because tog-automorphisms of the loop $\tilde{\ell}$ are type preserving permutations of the $\gamma \in \tilde{\ell}$ that also preserve the oriented loop structure and there is a bijection between the set of such tog-automorphisms of $\tilde{\ell}$ and those of $\mcb{g}(\ell)$. 
\end{proof}

\begin{remark}\label{rem:self-loops_internal}
In our definition of $\mcb{P}_{\C}$ we enforced that self-loops cannot be terminal edges in any path of length greater than $1$. 
This property is true for the base case of our induction above and propagates through the induction. 
In particular, anytime the path $\gamma = (e)$ for a self-loop $e$ appears in our re-routing operation when integrating at $z$, one must have $\underline{e} = \overline{e} = z$ so $\gamma$ appears in $\mathcal{P}_{l}$ which doesn't allow $\gamma$ to be the first or last path in any of the new ``path of paths'' produced in that step.
\end{remark}

\subsection{The real case}\label{subsec:real_case}
In this subsection we assume $H$ is a finite dimensional real Hilbert space.
The following is an analogue of Lemma~\ref{lemma:spherical_integral}.
\begin{lemma}\label{lemma:spherical_integral_real}
For any $N \in \N$ and vectors $a_{1},\dots,a_{2N} \in H$, we have 
\begin{equ}
\int_{\hat{H}}  \prod_{i=1}^{2N} \vscal{\hat{\varphi},a_i}\ \mrd \hat{\varphi}
=
 K_{N}^{\R} \sum_{\theta \in \hat{\mfS}_{2N}} 
\prod_{ \{i,j\} \in \theta } \vscal{a_i,a_{j}}\;,
\end{equ}
with 
\begin{equ}
K^{\R}_{N} \eqdef  \frac{\pi^{\dim_\R(H)/2} }{2^{N-1}\Gamma \big(N+ \dim_\R(H)/2 \big)}\;.
\end{equ}
Here $\dim_\R(H)$ is the dimension of $H$ as a real vector space, $\Gamma(z) = \int_{0}^{\infty} t^{z-1} e^{-t}\ \mrd t$,  and $\hat{\mfS}_{2N}$ is the set of all partitions of $\{1,2,\dots,2N\}$ into pairs. 
\end{lemma}
\begin{proof}
For $Z$ a standard Gaussian on $H$, we have
\begin{equs}
\E\Big[ &
\prod_{i=1}^{2N} \vscal{a_i,Z}\ 
\Big]
=
\frac{1}{(2\pi)^{\dim_\R(H)/2}}
\int_{H} 
\prod_{i=1}^{2N} \vscal{a_i,\varphi} e^{-|\varphi|^{2}/2} \mrd \varphi\\
{}&=
\frac{1}{(2\pi)^{\dim_\R(H)/2}}
\int_{0}^{\infty}
r^{2N + \dim_\R(H) - 1} e^{-r^2/2}\  \mrd r \; \times
\int_{\hat{H}}  \prod_{i=1}^{2N} \vscal{a_i,\hat{\varphi}}\; \mrd \hat{\varphi} \;.
\end{equs}
We then use Wick's rule and the fact that 
\[
\int_{0}^{\infty}
r^{2N + \dim_\R(H) - 1} e^{-r^2/2}\  \mrd r  = 2^{N+\dim_{\R}(H)/2 - 1} \Gamma \big(N+ \dim_\R(H)/2 \big)\;.
\]
\end{proof}
We now introduce the graphical/combinatorial objects that will organize our expansion in the real case. 

\begin{definition}
Given a set $X$, we write $X^{(2)}$ for the collection of all two element subsets of $X$.
A subset $G \subset X^{(2)}$ is called a (unoriented) graph on $X^{(2)}$. 
A graph $R \subset X^{(2)}$ which is a partition\footnote{That is, each element of $X$ appears in an element of $R$ precisely once.} of $X$ is called a pairing of $X$. 
\end{definition}

Given two sets $A$, $B$ we write $\hat{\mcb{G}}_{\R}(A,B)$ to be the set of all pairings of the set $\mcb{V}(A,B) = A \sqcup \big(B \times \{\mcb{o},\mcb{i}\}\big)$.  
While an element $G \in \hat{\mcb{G}}_{\R}(A,B)$ is technically an unoriented graph, in the standard sense it doesn't have any paths or loops of length more than $1$.  

The ``paths'' and ``loops'' from $G$ that we will use in our formulae arise from collapsing the vertices $(b, \mcb{o})$ and $(b,\mcb{i})$ 
for each $b \in B$ to get an oriented graph $\tilde{G}$ on $A \sqcup B$, and then taking the paths and loops in $\tilde{G}$ along with the data of how they traverse elements of $B$.  
Recall that in $\tilde{G}$, each vertex in $A$ (resp. $B$) is one incident to precisely one edge (resp. two edges, or with a self-loop) - we also include the following additional data: for each $b$ and each incidence of $b$ to an edge $e$, one also keeps track of whether $(b,\mcb{i})$ or $(b,\mcb{o})$ is incident to $e$.
One can think of $\mcb{o}$ and $\mcb{i}$ as referring to distinguished ``outgoing'' and ``incoming'' half-edges for a given vertex in $B$. 

With this point of view, we can associate to $G \in \hat{\mcb{G}}_{\R}(A,B)$ a set of unoriented ``paths'' $\CP(G)$ and ``loops''\footnote{The self-loops described in these paragraph are treated as elements of $\CL(G)$} $\CL(G)$ which carry\footnote{A loop or a path in $G$ is itself concretely realized as an element of $\hat{\mcb{G}}_{\R}(A',B')$ for $A' \subset A$, $B' \subset B$, the sets of vertices that are incident to the path or loop.} with them the data of how they traverse the elements of $B$.
We give some examples below along with some useful notation for working with such paths and loops. 

Writing $A = \{a_{1},a_{2}\}$ and $B = \{b_{1},b_{2},b_{3}\}$ in the figure below we give two examples of paths that travel between the vertices $a_{1},a_{2}$ as endpoints and the three vertices of $B$ in between. 
If one reduces them to graphs on the vertex set $A \sqcup B$, the two paths below would be the same, but since we also keep track of how they traverse elements of $B$ they will be treated as distinct. 
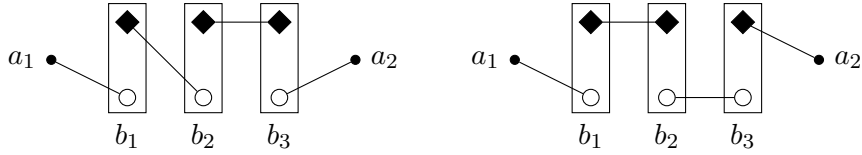
\begin{figure}[H]
\centering
\captionsetup[subfloat]{font=small}
\begin{tikzpicture}

\node [draw, dot, name=A1, label=left:{$a_1$}] at (0,1.5) {};
\node [draw, dot, name=A2,label=right:{$a_2$}] at (4,1.5) {};

\node [draw, circle, name=B1i, scale=0.6] at (1,1) {};
\node [draw, circle, name=B2i, scale=0.6] at (2,1) {};
\node [draw, circle, name=B3i, scale=0.6] at (3,1) {};

\node [draw, diamond, fill=black, name=B1o, scale=0.6] at (1,2) {};
\node [draw, diamond, fill=black, name=B2o, scale=0.6] at (2,2) {};
\node [draw, diamond, fill=black, name=B3o, scale=0.6] at (3,2) {};

\node [draw, fit=(B1i)(B1o), inner sep=2.5pt, label=below:{$b_1$}] {};
\node [draw, fit=(B2i)(B2o), inner sep=2.5pt, label=below:{$b_2$}] {};
\node [draw, fit=(B3i)(B3o), inner sep=2.5pt, label=below:{$b_3$}] {};

\draw (A1) -- (B1i);
\draw (B1o) -- (B2i);
\draw (B2o) -- (B3o);
\draw (B3i) -- (A2);
\end{tikzpicture}
\hspace{.5cm}
\begin{tikzpicture}

\node [draw, dot, name=A1, label=left:{$a_1$}] at (0,1.5) {};
\node [draw, dot, name=A2,label=right:{$a_2$}] at (4,1.5) {};

\node [draw, circle, name=B1i, scale=0.6] at (1,1) {};
\node [draw, circle, name=B2i, scale=0.6] at (2,1) {};
\node [draw, circle, name=B3i, scale=0.6] at (3,1) {};

\node [draw, diamond, fill=black, name=B1o, scale=0.6] at (1,2) {};
\node [draw, diamond, fill=black, name=B2o, scale=0.6] at (2,2) {};
\node [draw, diamond, fill=black, name=B3o, scale=0.6] at (3,2) {};

\node [draw, fit=(B1i)(B1o), inner sep=2.5pt, label=below:{$b_1$}] {};
\node [draw, fit=(B2i)(B2o), inner sep=2.5pt, label=below:{$b_2$}] {};
\node [draw, fit=(B3i)(B3o), inner sep=2.5pt, label=below:{$b_3$}] {};

\draw (A1) -- (B1i);
\draw (B1o) -- (B2o);
\draw (B2i) -- (B3i);
\draw (B3o) -- (A2);

\end{tikzpicture}
\caption{Circles indicate $\mcb{i}$ labels and diamonds represent $\mcb{o}$ labels.}\label{fig:path_real_case}
\end{figure}
The path on the left of Figure~\ref{fig:path_real_case} could be written as either
\begin{equ}\label{eq:path}
(a_{1}, b^{1}_1, b^1_2, b^{-1}_3, a_{2}) 
\quad
\text{or}
\quad 
(a_{2}, b^1_3, b^{-1}_2, b^{-1}_1, a_{1})\;.
\end{equ}
The superscripts of the $b$-sites, which can take the value $\pm 1$, indicate how we traverse the given $b$ site.
Here, in the first expression, we write $b^1_1$ to indicate that, travelling from $a_{1}$, the path first encounters $(b_1,\mcb{i})$, and similarly we write $b^{-1}_3$ to indicate that, travelling along this orientation, the path first encounters $(b_3,\mcb{o})$. 
There are two orientations for the path, which gives us the two expressions in \eqref{eq:path}, and both expressions describe the same path on the left of Figure~\ref{fig:path_real_case}.
On the other hand, the path on the right of Figure~\ref{fig:path_real_case} would be seen as a distinct path, and could be written $(a_{1}, b^1_1, b^{-1}_2, b^{1}_3, a_{2}) $ or $(a_{2}, b^{-1}_3, b^{1}_2, b^{-1}_1, a_{1})$.

Given a collection of vectors $v = (v_{a}: a \in A)$ with $v_{a} \in H$ along with a collection of operators $M = (M_{b}: b \in B)$ with $M_{b} \in L(H,H)$, and a path $\gamma$ written $(a,b_{1}^{p_{1}},\dots,b_{n}^{p_n},a')$ with $n \ge 0$ and $p_{j} = \pm 1$ for $1 \le j \le n$, we set 
\begin{equ}\label{eq:real_path_contrib}
\gamma(M,v) = \bvscal{v_{a}, \Big(\prod_{j=1}^{n} M_{j} \Big)v_{a'}}\;,
\end{equ}
where $M_{j} = M_{b_{j}}$ if $p_{j} = 1$ and $M_{j} = M^{T}_{b_{j}}$ if $p_{j}=-1$.  
As an example, for the path $\gamma$ as in \eqref{eq:path},
\begin{equ}
\gamma(M,v)
=
\vscal{ v_{a_{1}}, M_{b_1} M_{b_2} M^{T}_{b_3} v_{a_{2}} }
=
\vscal{ v_{a_{2}}, M_{b_3} M^{T}_{b_2} M^{T}_{b_1} v_{a_{1}} }\;.
\end{equ}
Note that $\gamma(M,v)$ is well-defined, i.e. different ways of writing the path give equivalent expressions. 

We now turn to loops. 
We set $B=\{b_{1},b_{2},b_{3},b_{4}\}$ and give two examples of loops that  look the same as loops on $B$ but which we view as distinct since they traverse the elements of $B$ differently. 
\begin{figure}[H]
\centering
\captionsetup[subfloat]{font=small}
\begin{tikzpicture}

\node [draw, circle, name=B1i, scale=0.6] at (1,1) {};
\node [draw, circle, name=B2i, scale=0.6] at (2,1) {};
\node [draw, circle, name=B3i, scale=0.6] at (3,1) {};
\node [draw, circle, name=B4i, scale=0.6] at (4,1) {};

\node [draw, diamond, fill=black, name=B1o, scale=0.6] at (1,2) {};
\node [draw, diamond, fill=black, name=B2o, scale=0.6] at (2,2) {};
\node [draw, diamond, fill=black, name=B3o, scale=0.6] at (3,2) {};
\node [draw, diamond, fill=black, name=B4o, scale=0.6] at (4,2) {};

\node [draw, fit=(B1i)(B1o), inner sep=2.5pt, label=below:{$b_1$}] {};
\node [draw, fit=(B2i)(B2o), inner sep=2.5pt, label=below:{$b_2$}] {};
\node [draw, fit=(B3i)(B3o), inner sep=2.5pt, label=below:{$b_3$}] {};
\node [draw, fit=(B4i)(B4o), inner sep=2.5pt, label=below:{$b_4$}] {};

\draw (B1i) -- (B2o);
\draw (B2i) -- (B3i);
\draw (B3o) -- (B4i);
\draw (B4o) to [bend right] (B1o);

\end{tikzpicture}
\hspace{1cm}
\begin{tikzpicture}

\node [draw, circle, name=B1i, scale=0.6] at (1,1) {};
\node [draw, circle, name=B2i, scale=0.6] at (2,1) {};
\node [draw, circle, name=B3i, scale=0.6] at (3,1) {};
\node [draw, circle, name=B4i, scale=0.6] at (4,1) {};

\node [draw, diamond, fill=black, name=B1o, scale=0.6] at (1,2) {};
\node [draw, diamond, fill=black, name=B2o, scale=0.6] at (2,2) {};
\node [draw, diamond, fill=black, name=B3o, scale=0.6] at (3,2) {};
\node [draw, diamond, fill=black, name=B4o, scale=0.6] at (4,2) {};

\node [draw, fit=(B1i)(B1o), inner sep=2.5pt, label=below:{$b_1$}] {};
\node [draw, fit=(B2i)(B2o), inner sep=2.5pt, label=below:{$b_2$}] {};
\node [draw, fit=(B3i)(B3o), inner sep=2.5pt, label=below:{$b_3$}] {};
\node [draw, fit=(B4i)(B4o), inner sep=2.5pt, label=below:{$b_4$}] {};

\draw (B1i) -- (B2i);
\draw (B2o) -- (B3i);
\draw (B3o) -- (B4o);
\draw (B4i) to (B1o);

\end{tikzpicture}
\caption{Circles indicate $\mcb{i}$ labels and diamonds represent $\mcb{o}$ labels.}
\label{fig:loop_real_case}
\end{figure}
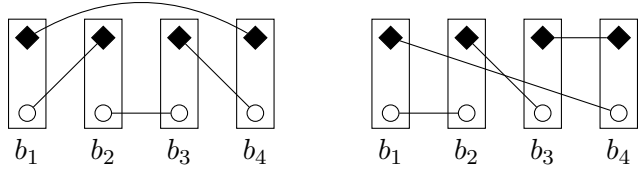
The loop $\ell$ on the left of Figure~\ref{fig:loop_real_case} could be written as
\begin{equ}\label{eq:loop}
(b^{-1}_{1},b^{-1}_{2},b^{1}_{3},b^{1}_{4})
\quad
\text{or}
\quad
(b^1_1,b_4^{-1},b_3^{-1},b_2^{1})\;,
\end{equ}
or any cyclic permutation of the above, such as $(b^{1}_{3},b^{1}_{4},b^{-1}_{1},b^{-1}_{2})$ or $(b^{1}_{2},b^{1}_{1},b^{-1}_{4},b^{-1}_{3})$.
The two expressions in \eqref{eq:loop} differ by orientation and the cyclic permutations change the ``basepoint''. 
One way of writing the path on the right of Figure~\ref{fig:loop_real_case} is $(b^{1}_{2},b^{1}_{3},b^{-1}_{4},b^{-1}_{1})$.

Analogously to \eqref{eq:real_path_contrib}, given $M$ as earlier,  and a loop $\ell$ written $(b_{1}^{p_{1}},\dots,b_{n}^{p_n})$ with $n \ge 1$, we set 
\begin{equ}
\ell(M) = \Trace \Big( \prod_{j=1}^{n} M_{j} \Big)\;,
\end{equ}
with $M_{j}$ defined as in  \eqref{eq:real_path_contrib}. 
For $\ell$ as on the left of Figure~\ref{fig:loop_real_case} we would have
\begin{equ}\label{eq:real_loop_contrib}
\ell(M) = \Trace( M_{b_1}^{T} M_{b_2}^{T} M_{b_3} M_{b_4})
=
\Trace( M_{b_1} M^{T}_{b_4} M^T_{b_3} M_{b_2})\;.
\end{equ}

The picture in Figure~\ref{fig:selfloop_real_case} gives a picture of a self-loop.
\begin{figure}[H]
\centering
\begin{tikzpicture}

\node [draw, circle, name=B1i, scale=0.6] at (1,1) {};

\node [draw, diamond, fill=black, name=B1o, scale=0.6] at (1,2) {};

\node [draw, fit=(B1i)(B1o), inner sep=2.5pt, label=below:{$b$}] {};

\draw (B1i) -- (B1o);

\end{tikzpicture}
\caption{A self loop $(b^1) = (b^{-1})$. }
\label{fig:selfloop_real_case}
\end{figure}
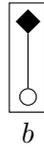
For $\ell$ as in Figure~\ref{fig:selfloop_real_case} one would write, in analogy with \eqref{eq:loop}, $(b^{1})$ or $(b^{-1})$ and set $\ell(M) = \Trace(M_b) = \Trace(M_b^T)$.  

Our analogue of Lemma~\ref{lem:integrate_higgs_onesite} is then the following. 
\begin{lemma}\label{lem:integrate_higgs_onesite_real}
Let $A,B$ be finite, disjoint sets. Suppose we fix a collection of vectors $v = (v_{a}: a \in A)$ with $v_{a} \in H$ along with a collection of operators $M = (M_{b}: b \in B)$ with $M_{b} \in L(H,H)$ .

We then have
\begin{equs}[eq:integrating_higgs_onesite_real]
\int_{H} &
\Big( 
\prod_{a \in A} \vscal{v_{a},\varphi}
\Big)
\Big(
\prod_{b \in B} \vscal{\varphi, M_{b} \varphi}
\Big)\ \mrd \rho^{\lambda}(\varphi) 
\\
{}&=
C^{\R,\lambda}_{|A|/2 + |B|}
\sum_{G \in \hat{\mcb{G}}_{\R}(A,B)}
\Big(
\prod_{\gamma \in \mathcal{P}(G)}
\gamma(M,v)
\Big)
\Big(
\prod_{\ell \in \mathcal{L}(G)}
\ell(M)
\Big)\;,
\end{equs}
understood as zero if $|A|$ is not even and
where, for $j \in \N$,
\[
C^{\R,\lambda}_{j} = K^{\R}_{j} \int_{0}^{\infty} s^{2j}\  \mrd \lambda(s) \;.
\] 
\end{lemma}
\begin{proof}
We fix an orthonormal basis $(e_{i})_{i=1}^{k}$ of $H$.
Without loss of generality, we may assume that, for $a \in A$, $v_{a} = e_{i_{a}}$ for some $1 \le i _{a} \le k$. 
We can then write the left-hand side of \eqref{eq:integrating_higgs_onesite_real}  as 
\begin{equs}
\sum_{ \substack{(i_{q}:  q \in B \times \{\mcb{i},\mcb{o}\})
\\
1 \le i_{q} \le k
}} &
\int_{0}^{\infty} 
\int_{\hat{H}}
s^{|A|+2|B|}
\prod_{a \in A} \vscal{e_{i_{a}},  \hat{\varphi} }\\
{}&
\quad \quad \quad
\times 
\prod_{b \in B} \vscal{e_{i_{(b,\mcb{i})}},\hat{\varphi}}
\vscal{e_{i_{(b,\mcb{o})}} ,\hat{\varphi}}
\vscal{ e_{i_{(b,\mcb{i})}}, M_{b} e_{i_{(b,\mcb{o})}}}\ 
\mrd \hat{\varphi}\ \mrd \lambda(s)\;.
\end{equs}
Above, we wrote $\varphi = s \hat{\varphi}$ and decomposed the $\varphi$ integration into integration over $s$ and $\hat{\varphi}$. 
We also expanded each factor of $\vscal{ \hat{\phi}, M_{b} \hat{\phi}}$ using our orthonormal basis. 
 
We can then apply Lemma~\ref{lemma:spherical_integral_real} to perform the integral over $\hat{\varphi}$ which gives us 
\begin{multline*}
K^{\R}_{|A|/2+|B|}
\sum_{ \substack{(i_{q}:  q \in B \times \{\mcb{i},\mcb{o}\})
\\
1 \le i_{q} \le k
}} 
\Bigg(
\prod_{b \in B} 
 \vscal{ e_{i_{(b,\mcb{i})}}, M_b e_{i_{(b,\mcb{o})}}}
 \Bigg)
\\
\times\sum_{G \in \hat{\mcb{G}}_{\R}(A,B)}
\Bigg[
\prod_{ \{w,u\}  \in \mcb{V}(A,B) } \vscal{e_{i_{u}} ,e_{i_w}}\
\Bigg]\;.
\end{multline*}
Note that the sum above is empty unless $|A|$ is even. 
The rest of the proof follows by summing indices over the paths and loops generated by $G$ similarly to the last part of the proof of Lemma~\ref{lem:integrate_higgs_onesite}. 
\end{proof}
\subsubsection{Typed pairing isomorphisms}
We now introduce the real case analogs of the isomorphisms and symmetry factors given in Section~\ref{subsec:tog-iso} for the complex case.
\begin{definition} A typed unoriented graph is an unoriented graph on a typed set $(X,\iota)$.
We define typed pairings similarly.
\end{definition}
As in the complex case, we have a multiset analogue of Lemma~\ref{lem:integrate_higgs_onesite_real}.
We fix two disjoint sets of types $\mcb{A}$ and $\mcb{B}$, and additionally fix a decomposition $\mcb{B} = \vec{\mcb{B}} \sqcup \mcb{B}^{o}$. 
For any $\CA \in \multiset{\mcb{A}}$, $\CB \in \multiset{\mcb{B}}$, recall that $\hat{\mcb{G}}_{\R}(\mathcal{A},\mathcal{B})$ 
is the set of all pairings of $\CA \sqcup \big(\CB \times \{\mcb{o},\mcb{i}\}\big)$.

Given the typing of $\CA$ by $\mcb{A}$ and a typing of $\CB$ by $\mcb{B}$, we define an associated typing of $\CA \sqcup \big(\CB \times \{\mcb{o},\mcb{i}\}\big)$ by a set of possible types $\mcb{A} \sqcup \big(\vec{\mcb{B}} \times \{\mcb{o},\mcb{i} \}\big) \sqcup \mcb{B}^{o}$ as follows.
The elements of $\CA$ are naturally typed by $\mcb{A}$.
Regarding $\CB$, we note there is a decomposition $\CB = \vec{\CB} \sqcup \CB^{o}$ with $\vec{\CB} \in \multiset{\vec{\mcb{B}}}$ and   $\CB^{o} \in \multiset{\mcb{B}^{o}}$. 
Then the elements of $\vec{\CB} \times \{\mcb{i},\mcb{o}\}$ are naturally typed by  $\vec{\CB} \times \{\mcb{i},\mcb{o}\}$, while we type the elements of $\CB^{o} \times \{\mcb{i},\mcb{o}\}$ by $\mcb{B}^{o}$ via ``forgetting'' the second component.
More precisely, writing $(\CA,\iota_{\CA})$, $(\CB,\iota_{\CB})$ for the typed sets taken as input, we define a type map $\iota$ on $\CA \sqcup \big(\CB \times \{\mcb{o},\mcb{i}\}\big)$
by setting, for $c \in \CA \sqcup \big(\CB \times \{\mcb{o},\mcb{i}\}\big)$,
\[
\iota(c) 
=
\begin{cases}
\iota_{\CA}(a) & \text{ if } c = a \in \mathcal{A},\\
\iota_{\CB}(b) & \text{ if } c = (b,\mcb{j}) \in \mathcal{B}^{o} \times \{\mcb{i},\mcb{o}\},\\
(\iota_{\CB}(b), \mcb{j}) & \text{ if } c = (b,\mcb{j}) \in \vec{\mathcal{B}} \times \{\mcb{i},\mcb{o}\}.
\end{cases}
\]

Before stating the key definitions of this subsection, we note that any bijection $f$ on $X$ induces a bijection $\mathbf{f}$ on $X^{(2)}$ by mapping $\{a,b\} \mapsto \mathbf{f}(\{a,b\}) = \{f(a),f(b)\}$. 

\begin{definition}\label{def:tg-iso}
For $\CA \in \multiset{\mcb{A}}, \CB \in \multiset{\mcb{B}}$, an $(\CA,\CB)$-type permutation is a type permutation $f$ on  $\Big(\CA \sqcup \big(\CB \times \{\mcb{o},\mcb{i}\}\big), \iota \Big)$ with the property that $\mathbf{f}$ maps the set $\Big\{ \{(b,\mcb{i}),(b,\mcb{o})\}: b \in \CB \Big\}$ to itself. 
\end{definition}

\begin{definition}\label{ref:typed_pairing_iso}
Given $G_1, G_2 \in \hat{\mcb{G}}_{\R}(\CA,\CB)$, an $(\CA,\CB)$-type permutation $f$ is called a tg-isomorphism from $G_1$ to $G_2$ if $\mathbf{f}$ maps $G_1$ to $G_2$. 
If such a tg-isomorphism exists, we call $G_1$ and $G_2$ tg-isomorphic. 

In the case where $G = G_1 = G_2$, we call $f$ a tg-automorphism.
\end{definition}
Given $G \in \hat{\mcb{G}}_{\R}(\CA,\CB)$, we define $S(G)$ to be the number of tg-automorphisms on $G$. 
We also define $\mcb{G}_{\R}(\mathcal{A},\mathcal{B})$ to be the set of tg-isomorphism classes of $\hat{\mcb{G}}_{\R}(\mathcal{A},\mathcal{B})$.
Note that will perform a similar overloading of notation as that described in Remark~\ref{rem:graph_overload}
 
%
%

As before, by identifying $(b,\mcb{o})$ and $(b,\mcb{i})$, there is a natural one-to-one correspondence between the set $\hat{\mcb{G}}_{\R}(\mathcal{A},\mathcal{B})$ and graphs with vertex set $\mathcal{A} \sqcup \mathcal{B} $ where each vertex in $\mathcal{A}$ is incident only to one edge and each vertex in $\mathcal{B}$ is either incident to precisely two distinct edges or just one self-loop, and where we keep track of how vertices in $\CB$ are traversed.
The notion of tg-isomorphism described on $\hat{\mcb{G}}_{\R}(\mathcal{A},\mathcal{B})$ viewed in this setting carries
the unoriented graph structure $\mathcal{A} \sqcup \mathcal{B}$ but where ``vertices of the same type are indistinguishable'' and we also keep track of the different ``incoming and outgoing'' connections for types in $\vec{\mcb{B}}$ but forget this for types in $\CB^{o}$. 

We also have notions of tg-isomorphism class of loops and paths, which we refer to again as types of paths and loops. 
Given $G \in \mcb{G}_{\R}(\mathcal{A},\mathcal{B})$, we write as before $\mathcal{P}(G)$ for the multiset of types of paths in $G$ and $\mathcal{L}(G)$ the multiset of types of loops in $G$.

An $\mcb{A}$-vector assignment is a tuple of vectors $v = (v_{u}: u \in \mcb{A})$ with $v_{u} \in H$.
A $\mcb{B}$-operator assignment is a tuple $M = (M_{b} : b \in \mcb{B})$ with $M_{b} \in L(H,H)$ for $b \in \CB$ and $M_{b}$ symmetric for $b \in \CB^{o}$, 
Given such assignments, we define, for any $\gamma \in \mathcal{P}(G)$ and $\ell \in \mathcal{L}(G)$, $\gamma(M,v)$ and $\ell(M)$ as in \eqref{eq:real_path_contrib} and \eqref{eq:real_loop_contrib} - note that these definitions are well-defined with respect to tg-isomorphism classes of paths and loops because we have enforced the necessary symmetry conditions on $M$.  

Additionally, given a fixed loop or path $\eta \in \mathcal{P}(G) \sqcup \mathcal{L}(G)$ we define $S(\eta)$ to be the number of tg-automorphisms of $\eta$ using the notion of tg-isomorphism we've defined above. 

\begin{remark}
\label{rem:automorphs_real_case}
For $\gamma \in \mathcal{P}(G)$, one has, in contrast with the complex case, $S(\gamma) \in \{1,2\}$, with $S(\gamma) = 2$ if and only if $\gamma$ a ``palindrome''. 
As an example (see Figure~\ref{fig:gamma_real_case} for illustrations), suppose $u,v$ both are of the same type in $\mcb{A}$.
Then if $\gamma$ is the path $u \leftrightarrow v$,
one would have $S(\gamma) =2$ since one can exchange $u$ and $v$. 
As another example, if one has additional vertices $w, w'$ of the same type in $\vec{\mcb{B}}$, then the path $\gamma'$ given by $u \leftrightarrow (w,\mcb{i}) - (w,\mcb{o}) \leftrightarrow (w',\mcb{o}) - (w',\mcb{i}) \leftrightarrow v$
also has $S(\gamma') = 2$ since one is allowed to simultaneously exchange $u$ and $v$ along with $w$ and $w'$ (exchanging only one of these pairs would not be allowed). 
On the other hand, if $\gamma''$ is given by $u \leftrightarrow (w,\mcb{i}) - (w,\mcb{o}) \leftrightarrow v$ we have $S(\gamma'') = 1$, in particular we cannot interchange $u$ and $v$. 
As a last example, if we have $\tilde{w}$ typed in $\CB^{o}$ then $\tilde{\gamma}$ given by $u \leftrightarrow (\tilde{w},\mcb{i}) - (\tilde{w},\mcb{o}) \leftrightarrow v$ would have $S(\tilde{\gamma}) = 2$ since one can simultaneously exchange $u$ and $v$ along with $(\tilde{w},\mcb{i})$ and $(\tilde{w},\mcb{o})$. 

Similarly for loops, in contrast to the complex case, the tg-automorphism group of $\ell\in\mathcal{L}(G)$ is a subgroup of $C_n\times \Z_2$ (cf. Remark~\ref{rem:automorphisms of loops}).
\end{remark}

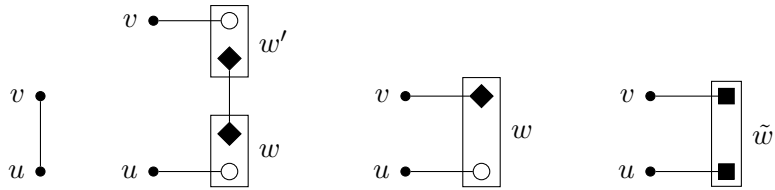
\begin{figure}[H]
\centering
\begin{tikzpicture}

\node [draw, dot, name=u, label=left:{$u$}] at (0,0) {};
\node [draw, dot, name=v,label=left:{$v$}] at (0,1) {};

\draw (u) -- (v);
\end{tikzpicture}
\qquad
\begin{tikzpicture}

\node [draw, dot, name=u, label=left:{$u$}] at (0,0) {};
\node [draw, dot, name=v, label=left:{$v$}] at (0,2) {};

\node [draw, circle, name=w1i, scale=0.6] at (1,0) {};
\node [draw, circle, name=w2i, scale=0.6] at (1,2) {};

\node [draw, diamond, fill=black, name=w1o, scale=0.6] at (1,0.5) {};
\node [draw, diamond, fill=black, name=w2o, scale=0.6] at (1,1.5) {};

\node [draw, fit=(w1i)(w1o), inner sep=2.5pt, label=right:{$w$}] {};
\node [draw, fit=(w2i)(w2o), inner sep=2.5pt, label=right:{$w'$}] {};

\draw (u) -- (w1i);
\draw (w1o) -- (w2o);
\draw (w2i) -- (v);

\end{tikzpicture}
\qquad
\begin{tikzpicture}

\node [draw, dot, name=u, label=left:{$u$}] at (0,0) {};
\node [draw, dot, name=v, label=left:{$v$}] at (0,1) {};

\node [draw, circle, name=w1i, scale=0.6] at (1,0) {};

\node [draw, diamond, fill=black, name=w1o, scale=0.6] at (1,1) {};

\node [draw, fit=(w1i)(w1o), inner sep=2.5pt, label=right:{$w$}] {};

\draw (u) -- (w1i);
\draw (w1o) -- (v);
\end{tikzpicture}
\qquad
\begin{tikzpicture}

\node [draw, dot, name=u, label=left:{$u$}] at (0,0) {};
\node [draw, dot, name=v, label=left:{$v$}] at (0,1) {};

\node [draw, rectangle, fill=black, name=w1i, scale=0.8] at (1,0) {};

\node [draw, rectangle, fill=black, name=w1o, scale=0.8] at (1,1) {};

\node [draw, fit=(w1i)(w1o), inner sep=2.5pt, label=right:{$\tilde w$}] {};

\draw (u) -- (w1i);
\draw (w1o) -- (v);
\end{tikzpicture}
\caption{The four examples of paths $\gamma,\gamma',\gamma'',\tilde\gamma$ from Remark~\ref{rem:automorphs_real_case} are presented from left to right. 
In the middle two figures, $w$ and $w'$ carry types in $\vec{\CB}$, and so we use circles to indicate $\mcb{i}$ labels and diamonds for $\mcb{o}$ labels. 
In the final figure, since $\tilde{w}$ is typed by an element of $\CB^{o}$, we use squares to indicate both $\mcb{i}$ or $\mcb{o}$ -- the two squares can be interchanged in a tg-automorphism of $\tilde\gamma$.}\label{fig:gamma_real_case}
\end{figure}

\begin{remark}
The special consideration of palindromes appears similarly
in the study of loop-soups, see, e.g.~\cite[Sec.~2.5.1]{Werner_Powell_21_GFF}.
\end{remark}

For loops $\ell$ and paths $\gamma$ we write $\hat{\ell}(M) = \ell(M) / S(\ell)$ along with $\hat{\gamma}(M,v) = \gamma(M,v)/S(\gamma)$. 
Note that any $G \in {\mcb{G}}_{\R}(\mathcal{A},\mathcal{B})$ has 
\[
\frac{\mathcal{A}!\mathcal{B}! 2^{|\CB^{o}|}}
{\mathcal{P}(G)! \mathcal{L}(G)! \prod_{\gamma \in \mathcal{P}(G)} S(\gamma) 
\prod_{\ell \in \mathcal{L}(G)} S(\ell)}\;, 
\]
representatives in $\hat{\mcb{G}}_{\R}(\mathcal{A},\mathcal{B})$.
Combining this observation with 
Lemma~\ref{lem:integrate_higgs_onesite_real}  immediately gives the following analogue of Lemma~\ref{lemma:integrating_onesite_multiset}. 

\begin{lemma}\label{lemma:real_integrating_onesite_multiset}
With the notation above, and for any $\mcb{A}$ vector assignment $v$ and $\mcb{B}$-operator assignment $M$, one has 
\begin{equs}
{}&\frac{1}{\mathcal{A}!\mathcal{B}!}
\int_{H} 
\Big( 
\prod_{a \in \mathcal{A}} \vscal{v_{a},\phi}
\Big)
\Big(
\prod_{b \in \vec{\mathcal{B}}} \vscal{\phi, M_{b} \phi}
\Big)
\Bigg(
\prod_{b \in \CB^{o}}
\frac{ \vscal{\phi, M_{b} \phi}}{2}
\Bigg)\ 
\mrd \rho^{\lambda}(\phi)\\
{}&=
C^{\R,\lambda}_{|\mathcal{A}|/2 + |\mathcal{B}|}
\sum_{G \in \mcb{G}_{\R}(\mathcal{A},\mathcal{B})}
\frac{1}{\mathcal{P}(G)! \mathcal{L}(G)!}
\Big(
\prod_{\gamma \in \mathcal{P}(G)}
\hat{\gamma}(M,v)
\Big)
\Big(
\prod_{\ell \in \mathcal{L}(G)}
\hat{\ell}(M)
\Big)\;.
\end{equs}	
\end{lemma}

\subsubsection{Lattice paths and loops in the real case}

Returning to our graph $\mathbb{G}$, we will write $\mcb{P}_{\R}$ for the set of all lattice paths. 
In the real case we enforce that
\begin{enumerate}[label=(\roman*)]
\item every lattice path contains at least one edge (and edges can be traversed multiple times),
\item lattice paths are unoriented and can traverse each edge in $\mathbb{E}$ either along or opposite to its orientation, and 
\item any lattice path consisting of more than one edge cannot start/end with an edge in $\mathbb{E}^{o}$. 
\end{enumerate}
To define $\mcb{P}_{\R}$ precisely, we need to introduce some notation.  
We define a new set of edges where $\vec{\mathbb{E}}$ is doubled, that is
\[
\mathbb{E}^{\leftrightarrow}
=
\big\{ 
e^{1} : e \in \vec{\mathbb{E}}
\big\}
\sqcup
\big\{ 
e^{-1} : e \in \vec{\mathbb{E}}
\big\}
\sqcup
\mbbE^{o}
\;.
\]
We often abuse notation and write, for $e \in \mbbE^{o}$, $e = e^1 = e^{-1} \in \mbbE^{\leftrightarrow}$. 
We fix endpoints for elements $e\in \mbbE^{\leftrightarrow}$ by $(\underline{e^{1}},\overline{e^1}) = (\underline{e},\overline{e})$ and $(\underline{e^{-1}},\overline{e^{-1}}) = (\overline{e},\underline{e})$. 

We write $\widehat{\mcb{P}}_{\R}$ for the set of finite, non-empty sequences $\omega = (e^{p_1}_{1},\dots,e^{p_n}_{n})$ of edges in $\mathbb{E}^{\leftrightarrow}$ with the property that for $1 \le i < n$, one has $\overline{e^{p_i}_{i}} = \underline{e^{p_{i+1}}_{i+1}}$. 

We then set $\mathring{\mcb{P}}_{\R}$ to be the set of all sequences in $\widehat{\mcb{P}}_{\R}$ which are either of length $1$ or which have the property that they neither start nor end with an edge $e \in \mathbb{E}^{o}$.

For $\omega \in \mathring{\mcb{P}}_{\R}$, we denote by $[\omega]$ the set of endpoints of $\omega$ in $\mbbV$, i.e. for $\omega$ as above, $[\omega] = \{\underline{e^{p_1}_{1}}, \overline{e^{p_n}_n}\}$.

We define a ``reversal'' involution $\iota$ on $\widehat{\mcb{P}}_{\R}$. 
More precisely, for $\omega$ as above we set $\iota(\omega) = (f^{q_1}_{1},\dots,f^{q_n}_{n})$ where, for $1 \le j \le n$, $f_{j} = e_{n - j +1}$ and $q_{j} = - p_{n-j+1}$. 
Finally the set of lattice paths $\mcb{P}_{\R}$ is defined as the quotient of $\mathring{\mcb{P}}_{\R}$ under $\iota$. 

We will write $\mcb{L}_{\R}$ for the set of all lattice loops.
In this subsection
\begin{enumerate}[label=(\roman*)]
\item every lattice loop consists of at least one edge (and edges can be traversed multiple times),
\item lattice loops do not have a distinguished basepoint, and
\item lattice loops are unoriented, in particular they can traverse edges in $\mbbE$  along or opposite to their orientation. 
\end{enumerate}
To precisely define $\mcb{L}_{\R}$, we let $\sim$ be the smallest equivalence relation on $\widehat{\mcb{P}}_{\R}$ closed under both the reversal $\iota$ and cyclic permutations. 
We then set $\mcb{L}_{\R} = \widetilde{\mcb{P}}_{\R} / \sim$ where $\widetilde{\mcb{P}}_{\R} = \{ \omega \in  \widehat{\mcb{P}}_{\R}: |[\omega]| = 1\}$. 

We introduce symmetry factors $S(\cdot)$ for lattice paths and loops.
Given a path $\gamma$ of length $1$, $\gamma$ is canonically associated to a single edge $e \in \mbbE$ so we write $\gamma = \gamma_{e}$ in this case, and we set $S(\gamma_{e}) = 1$ if $e \in \vec{\mbbE}$ and $S(\gamma_{e}) =2$ if $e \in \mbbE^o$. 
For larger lattice paths and loops we will refer to symmetry factors of associated typed pairings. 
Our sets of types $\mcb{A}$ and $\mcb{B}$ will be given by disjoint copies of $\mbbE$. 
Our decomposition of $\mcb{B}$ is given by setting $\vec{\mcb{B}} = \vec{\mbbE}$ and $\mcb{B}^{o} = \mbbE^{o}$. 

Given a path $\gamma \in \mcb{P}_{\R}$ of length $n \ge 2$ we fix $\mathring{\mcb{P}}_{\R} \ni \omega = (e^{p_1}_{1},\dots,e^{p_n}_{n}) \in \gamma$ and define a typed pairing $\mcb{g}(\omega) \in \hat{\mcb{G}}_{\R}(A,B)$
with $A = \{1,n\}$, $B=\{2,\dots,n-1\}$, and type map $j \mapsto e_{j}$, by setting
\begin{equ}
\mcb{g}(\omega)
=
\begin{cases}
\Big\{
\{ 1,\mcb{p}_{2}
\},
\{n,\mcb{q}_{n-1}\}
\Big\} 
\sqcup
\Big\{
\{\mcb{q}_{j},\mcb{p}_{j+1}\}
\Big\}_{j=2}^{n-2} & \text{ if } n > 2\;,\\
\Big\{ \{1,2\} \Big\} & \text{ if } n=2\;,
\end{cases}
\end{equ}
where,\footnote{Note that if $n =2$, we have $B = \emptyset$.} for $2 \le j \le n-2$, we set
\begin{equ}\label{eq:edge_to_vertex}
(\mcb{p}_{j},\mcb{q}_{j}) 
=
\begin{cases}
\big( (j,\mcb{i}), (j,\mcb{o}) \big) & \text{ if }p_{j} = 1 \text{ or }e_{j} \in \mbbE^{o}\;,\\
\big((j,\mcb{o}),(j,\mcb{i}) \big) & \text{ if }p_{j}=-1 \text{ and }e_j \in \vec{\mbbE}\;.
\end{cases}
\end{equ}
We write $\mcb{g}(\gamma)$ for the tg-isomorphism class of $\mcb{g}(\omega)$ which we note does not depend on our choice of $\omega$. 
We set $S(\gamma)$ to be the number of tg-automorphisms of $\mcb{g}(\gamma)$. 
Recall in this case that $S(\gamma)\in\{1,2\}$ (Remark~\ref{rem:automorphs_real_case}).

For a loop $\ell$ of length $n \ge 1$ we fix  $\widehat{\mcb{P}}_{\R} \ni \omega = (e^{p_1}_{1},\dots,e^{p_n}_{n}) \in \ell$ and define a typed pairing $\mcb{g}'(\omega) \in \hat{\mcb{G}}_{\R}(A,B)$, where $A = \emptyset$, $B = \{1,\dots,n\}$,  and $j \mapsto e_{j}$ is again the type map. 
This pairing is given by
\[
\mcb{g}'(\omega) = 
\Big\{ \{ \mcb{q}_{j},  \mcb{p}_{j}\} \Big\}_{j=1}^{n-1}
\sqcup
\Big\{
\{\mcb{q}_{n},\mcb{p}_{1}\}
\Big\}
\;,
\]
with $(\mcb{p}_{j},\mcb{q}_{j})$ defined as in \eqref{eq:edge_to_vertex} for $1 \le j \le n$. 
We write $\mcb{g}(\ell)$ for the tg-isomorphism class of $\mcb{g}'(\omega)$ and define $S(\ell)$ to be the number of tg-automorphisms of $\mcb{g}(\ell)$.

\begin{example}
We give example computations of symmetry factors for lattice loops. 
 
Suppose that $e \in \mathbb{E}^{o}$, then the loop $\ell_{e,n}$ of length $n$ that contains the sole representative $\omega = \underbrace{(e,\dots,e)}_{n \text{ times} }$ has $S(\ell_{e,n}) = |C_{n} \times \Z_{2}|$.
This is because one can perform any cyclic permutation on $B = \{1,2,\dots,n\}$
- that is mapping, for some fixed $k$ and for every $i$, both $(i ,\mcb{i}) \mapsto (i + k,\mcb{i})$ and $(i, \mcb{o}) \mapsto (i+k,\mcb{o})$ with addition modulo $n$ - and one can also compose any such a permutation with a simultaneous flip $(j,\mcb{i}) \leftrightarrow (j,\mcb{o})$ for every $1 \le j \le n$.

If we had a loop $\ell$ containing the representative $\omega =(e^1,\tilde{e},e^{-1})$ where $e \in \vec{\mbbE}$ and $\tilde{e} \in \mbbE^{o}$ then one would have $S(\ell) = 2$. 
Referring to $\mcb{g}'(\omega)$, the one non-trivial tg-automorphism would be given by flipping $(1,\mcb{i}) \leftrightarrow (3,\mcb{i})$, $(1,\mcb{o}) \leftrightarrow (3,\mcb{o})$, and $(2,\mcb{i}) \leftrightarrow (2,\mcb{o})$. 
\end{example}

\subsubsection{Loop expansion in the real case}
We set $\mathbf{G}_{\R} = \multiset{\mcb{P}_{\R} \sqcup \mcb{L}_{\R}}$ to be the set of all multisets of lattice paths and loops.
Given $G \in \mathbf{G}_{\R}$, we write $\mathcal{P}(G)$ for the multiset of paths in $G$ and $\mathcal{L}(G)$ for the multiset of loops in $G$. 
Given $x \in \mbbV$, we write $G_{x}$ for the number of times $x$ is incident to an edge in $G$ with multiplicity,
i.e. according to the formula~\eqref{eq:def_G_x},
which we note is again well-defined in the sense that it is independent of the representatives $\gamma,\ell$ chosen in the sums.

We now introduce a notion of $\mathbb{E}$ operator assignment for the real case. 
\begin{definition}
An $\mbbE$ operator assignment is a tuple $M = (M_{e})_{e \in \mbbE}$ with $M_{e} \in L(H,H)$ and where we additionally impose, for $e \in \mbbE^{o}$, that $M_{e}$ is symmetric. 
\end{definition}

Given an $\mbbE$ operator assignment $M$ and a lattice loop $\ell \in \mcb{L}_{\R}$ which, as an equivalence class, contains the sequence $(e^{i_1}_{1},\dots,e^{i_n}_{n})$, we define 
\[
\ell(M) = \Trace \Big( \prod_{j=1}^{n} M_{j} \Big)
\enskip
\text{ and }
\enskip
\hat{\ell}(M) = \ell(M)/S(\ell)
\]
where $M_{j} = M_{e_{j}}$ if $i_{j}=1$ and $M_{j} = M_{e_{j}}^{T}$ if $i_{j}=-1$.
Note that $\ell(M)$ and $\hat{\ell}(M)$ are well-defined in the sense that they are independent of the choice of sequence used.

The following is our main result on loop expansions in the real case that mirrors Theorem~\ref{thm:loop_expansion}.
\begin{theorem}\label{thm:real_loop_expansion} 
Let $M = (M_{e})_{e \in \mbbE}$ be an $\mbbE$ operator assignment.
Then
\begin{equs}
Z(M) \eqdef 
\int_{H^{\mbbV}}&
\exp\Big(
\sum_{e \in \vec{\mbbE }}
\vscal{\phi_{\underline{e}},M_{e}\phi_{\overline{e}}}
+
\frac{1}{2}
\sum_{e \in \mbbE^o}
\vscal{\phi_{\underline{e}},M_{e}\phi_{\overline{e}}}
 \Big)\ 
 \prod_{x\in\mbbV} \mrd  \rho^{\lambda_{x}} (\phi_x)\\
{}&=
\sum_{\mathcal{L} \in \mathbf{L}_{\R}}
\frac{1}{\mathcal{L}!}
\Big(
\prod_{x \in \mbbV}
C^{\R,\lambda_{x}}_{\mathcal{L}_{x}/2}
\Big)
\prod_{\ell \in \mathcal{L}}
\hat{\ell}(M)\;,
\end{equs}
where $\mathbf{L}_{\R} = \multiset{\mcb{L}_{\R}}$. 
\end{theorem}
As in the complex case we prove this theorem using an inductive integration of sites.
Given $\bar{\mbbV} \subset \mbbV$, we define $\mathbf{G}_{\R}(\bar{\mbbV}) = \multiset{\mcb{P}_{\R}(\bar{\mbbV}) \sqcup \mcb{L}_{\R}(\bar{\mbbV})}$ where $\mcb{P}_{\R}(\bar{\mbbV})$ and $\mcb{L}_{\R}(\bar{\mbbV})$ are defined analogously to the complex case, i.e. they are the set of lattice paths and lattice loops where endpoints (for the paths) must be in $\mbbV \setminus \bar{\mathbb{V}}$ and all other vertices visited must be in $\bar{\mbbV}$. 

We now state our inductive lemma for the real case, which gives Theorem~\ref{thm:real_loop_expansion} when we set $\bar{\mbbV} = \mbbV$. 

Given $\gamma \in \mcb{P}_{\R}$, an $\mbbE$ operator assignment $M$, and a vector $\phi = (\phi_{x})_{x \in A}$ for some $A \subset \mbbV$ containing the endpoints of $\gamma$, we define 
\[
\gamma(M,\phi) = \vscal{\varphi_{a},  M_{1} \cdots M_{n} \phi_{b}}
\enskip
\text{ and }
\enskip
\hat{\gamma}(M,\phi)  =  \gamma(M,\phi)  /S(\gamma)
\]
where $\mathring{\mcb{P}}_{\R} \ni (e_{1}^{i_{1}},\dots,e_{n}^{i_{n}}) \in \gamma$ is some choice of orientation for $\gamma$ and, for $1 \le j \le n$,  
 $M_{j} = M_{e_{j}}$ if $i_{j}=1$, $M_{j} = M_{e_{j}}^{T}$ if $i_{j}=-1$, $a = \underline{e^{i_1}_{1}}$, and $b = \overline{e^{i_n}_n}$. 
Observe that
$\gamma(M,\phi)$ and $\hat{\gamma}(M,\phi)$ are indeed well-defined in that they does not depend on our choice of orientation.

The core of the induction is to use Lemma~\ref{lemma:real_integrating_onesite_multiset} as a ``re-routing lemma'' where, when integrating out the spin at $z \in \mbbV$, we will have $\CA$ as a multiset of lattice paths that have two distinct endpoints with precisely one of them given by $z$, $\CB$ as a multiset of lattice paths for which both endpoints are given by $z$.
\begin{lemma}\label{lem:real_full_loop_expansion} 
Let $M$ be an $\mbbE$ operator assignment.
Then, for any $\bar{\mbbV} \subset \mbbV$, we have 
\begin{equs}
\int_{H^{\mbbV}}&
\exp\Big(
\sum_{e \in \vec{\mbbE }}
\vscal{\phi_{\underline{e}},M_{e}\phi_{\overline{e}}}
+
\frac12
\sum_{e \in \mbbE^o}
\vscal{\phi_{\underline{e}},M_{e}\phi_{\overline{e}}}
 \Big)\ 
\prod_{x\in\mbbV} \mrd  \rho^{\lambda_{x}} (\phi_x) \\
=&
\sum_{ G \in \mathbf{G}_{\R}(\bar{\mbbV})}
\frac{
\prod_{x \in \bar{\mbbV}}
C^{\R,\lambda_{x}}_{G_{x}/2}}
{G!}
\prod_{\ell \in \mathcal{L}(G)}
\hat{\ell}(M)
\int_{H^{\mbbW}}
\prod_{\gamma \in \mathcal{P}(G)} \hat{\gamma}(M,\phi)\ 
\prod_{x\in \mbbW}  \mrd \rho^{\lambda_x}(\phi_x) \;,
\end{equs}
where we write $\mbbW = \mbbV \setminus \bar{\mbbV}$. 
\end{lemma}
\begin{proof}
Similarly to the proof of Lemma~\ref{lem:full_loop_expansion}, the base case of the induction, when $\bar{\mbbV} = \emptyset$, follows from expanding the exponential and observing that $\mcb{P}_{\R}(\emptyset) = \big\{ \gamma_{e} = \{(e^1),(e^{-1})\}: e \in \mbbE \big\}$ and $\mcb{L}_{\R}(\emptyset)= \emptyset$. 

For the inductive step, we fix $\bar{\mbbV} \subset \mathbb{V}$ and $z \in \mbbW = \mathbb{V} \setminus \bar{\mbbV}$.
We define $\mcb{P}_{if}(z)$ to be the set of all lattice paths in $\mcb{P}_{\R}(\bar{\mbbV})$ which have $z$ as precisely one endpoint and $\mcb{P}_{l}(z)$ to be the set of all lattice paths in $\mcb{P}_{\R}(\bar{\mbbV})$ that have $z$ as both their endpoints.   

We again write $\mathbf{P}_{if}(z) = \multiset{\mcb{P}_{if}(z)}$ and $\mathbf{P}_{l}(z)= \multiset{\mcb{P}_{l}(z)}$.
We set $\hat{\mcb{P}}(z)$ to consist of all lattice paths in $\mcb{P}_{\R}(\bar{\mbbV} \sqcup \{z\})$ that pass through $z$,
$\hat{\mcb{L}}(z)$ to be the set of all lattice loops in $\mcb{L}_{\R}(\bar{\mbbV} \sqcup \{z\})$ that pass through $z$, and write $\hat{\mathbf{G}}(z) = \multiset{\hat{\mcb{P}}(z) \sqcup \hat{\mcb{L}}(z)}$. 
Proving our inductive step comes down to proving
\begin{equs}[eq:induc_equality_real]
{}&
\sum_{ \mathcal{P}_{if} \in \mathbf{P}_{if}(z)}
\sum_{ \mathcal{P}_{l} \in \mathbf{P}_{l}(z)}
\frac{1}
{\mathcal{P}_{if}! \mathcal{P}_{l}! }
\int_{H}
\Big(
 \prod_{\gamma^a \in \mathcal{P}_{if}} \gamma^a(M,\varphi)
\Big)
\Big( 
\prod_{\gamma^b \in \mathcal{P}_{l}} \widehat{\gamma^b}(M,\varphi) 
\Big)
\mrd \rho^{\lambda_z}(\varphi_z)  \\
{}& =
\sum_{G \in \hat{\mathbf{G}}(z)}
\frac{C^{\R,\lambda_{z}}_{G_{z}/2}}{G!}
\prod_{\ell \in \mathcal{L}(G)}
\hat{\ell}(M)
\prod_{\gamma \in \CP(G)} \hat{\gamma}(M,\phi)\;.
\end{equs}
Note that, for $\gamma^a \in \mathcal{P}_{if}$, $S(\gamma^{a}) = 1$. 

To match the left- and right-hand sides of \eqref{eq:induc_equality_real} , we apply  Lemma~\ref{lemma:real_integrating_onesite_multiset} to the left side and define a bijection  $F\colon \mcb{G}_{\R}(z) \rightarrow  \hat{\mathbf{G}}(z)$ where
\[
\mcb{G}_{\R}(z) 
=
\bigsqcup_{\mathcal{P}_{if} \in  \mathbf{P}_{if}(z)}
\bigsqcup_{ \mathcal{P}_{l} \in \mathbf{P}_{l}(z)}
\mcb{G}_{\R}(\mathcal{P}_{if},\mathcal{P}_{l})\;.
\]
In our application of Lemma~\ref{lemma:real_integrating_onesite_multiset}, we take $\mcb{A} = \mcb{P}_{if}(z)$ and $\mcb{B} = \mcb{P}_{l}(z)$, which we recall are sets of lattice paths (and not lattice loops).
We decompose $\mcb{B}$ into $\mcb{B}^{o} = \{ \gamma \in \mcb{P}_{l}(z): S(\gamma) = 2\}$ and $\vec{\mcb{B}} = \mcb{B} \setminus \mcb{B}^{o}$.
Moreover, for each $\gamma \in \vec{\mcb{B}}$, we make a choice of one of its two orientations $\mathring{\mcb{P}}_{\R} \ni \omega(\gamma) \in \gamma$.  

Given $\tilde{G} \in \mcb{G}_{\R}(z)$, we define $F(\tilde{G})$ by turning the loops and paths of $\tilde{G}$ into lattice paths and loops via concatenation. 
The paths in $\mathcal{P}_{if}$ have distinct endpoints so there is no ambiguity in how to concatenate them to adjacent paths.  
Regarding concatenating lattice paths associated to the elements of $\mathcal{P}_{l}$, there is no ambiguity for how to do this for paths types in $\mcb{B}^{o}$ since they only have one orientation.
On the other hand, while path types in $\gamma \in \vec{\mcb{B}}$ have two orientations we have chosen a preferred orientation $\omega(\gamma) \in \gamma$. 
Then our definition of $F$ is finished by adopting the following convention: we associate the traversal $\rightarrow (b,\mcb{i}) - (b,\mcb{o}) \rightarrow$ in $\tilde{G}$ to an insertion of $\omega(\gamma(b))$ into our path/loop of lattice paths, and similarly we associate the traversal $\rightarrow (b,\mcb{o}) - (b,\mcb{i}) \rightarrow$ to an insertion of $\iota(\omega(\gamma(b)))$. 
See Example~\ref{ex:real_clarification}.

To see that $F$ as described above is indeed a bijection, first note that, given $G \in \hat{\mathbf{G}}(z)$, we obtain $\mathcal{P}_{if}(G) \in \mathbf{P}_{if}(z)$ and $ \mathcal{P}_{l}(G) \in \mathbf{P}_{l}(z)$ by looking at lattice paths and loops in $G$  and
``snipping'' them whenever they pass through $z$. 
Injectivity and surjectivity follow from the fact that there is precisely one tg-isomorphism $\tilde{G} \in \mcb{G}_{\R}(\mathcal{P}_{if}(G),\mathcal{P}_{l}(G))$ with $F(\tilde{G}) = G$. 

The identity~\eqref{eq:induc_equality_real} then follows by observing that for any $\tilde{G} \in \mcb{G}_{\R}(z)$ and $\eta \in \mathcal{L}(\tilde{G}) \sqcup \mathcal{P}(\tilde{G})$, if we write $F(\eta)$ for the corresponding lattice path or loop in $\hat{\mcb{L}}(z) \sqcup \hat{\mcb{P}}(z)$, one indeed has $S(F(\eta)) = S(\eta)$ since there is a bijective correspondence of tg-automorphisms of $F(\eta)$ and $\mcb{g}(\eta)$. 
\end{proof}
\begin{example}\label{ex:real_clarification}
We give an example to further clarify the bijection $F$ in the proof above. 
Suppose $\mathcal{P}_{if} = \{a_1,a_2\}$ and $\mathcal{P}_{l} = \{b_1,b_2,b_3\}$, then Figure~\ref{fig:path_real_case} gives examples of possible $\tilde{G}$.
Fix $\gamma(a_1), \gamma(a_2) \in \mcb{P}_{if}(z)$ for the types of path of $a_1$ and $a_2$ and similarly $\gamma(b_1), \gamma(b_2), \gamma(b_3) \in \vec{\mcb{P}}_{l}(z)$ for the types of path of $b_1$, $b_2$, and $b_3$. 
Note that we could have $\gamma(a_i) = \gamma(a_j)$ or $\gamma(b_i) = \gamma(b_j)$ for $i \not = j$. 
We also fix an orientation for the types of path in $\vec{\mcb{P}}_{l}(z)$ as in the proof above. 
The $\tilde{G}$ on the left-hand side of Figure~\ref{fig:path_real_case} consists of a single path, and accordingly $F(\tilde{G})$ consists of a single lattice path $\gamma$ which can be constructed by concatenating  $ \gamma(a_1)$ with $\omega(\gamma(b_1))$ with $\omega(\gamma(b_2))$ with $\iota(\omega(\gamma(b_3)))$ with $\gamma(a_2)$.
Alternatively, one can obtain $\tilde{\gamma}$ by concatenating $\gamma(a_2)$ with $\omega(\gamma(b_3))$ with $\iota(\omega(\gamma(b_2))$ with $\iota(\omega(\gamma(b_1))$ with $\gamma(a_1)$. 

As mentioned earlier, there is no choice for the orientation for $\gamma(a_i)$ when doing this concatenation since the endpoints of these paths are distinct. 
Finally, we note that our choice of ``chosen orientation'' $\omega(\cdot)$ on the elements of $\vec{\mcb{P}}_{l}(z)$  doesn't really matter since in the end we sum over \emph{all} $\tilde{G} \in \mcb{G}_{\R}(\mathcal{P}_{if},\mathcal{P}_{l})$. 
\end{example}

\subsection{Qualitative diamagnetic inequality}

Although we do not use it later, we record the following inequality for the quantities $Z(M)$ defined in Theorems~\ref{thm:loop_expansion} and~\ref{thm:real_loop_expansion}. 
This result can be seen as a (qualitative) diamagnetic inequality.
\begin{corollary}\label{cor:weakDM}
Let $H$ be a complex or real finite dimensional Hilbert space, and suppose the single site measures $  ( \rho^{\lambda_{x}} : x \in \mathbb{V})$ on $H$ satisfy Assumption~\ref{assumption:single_site}.
Given an $\mathbb{E}$ operator assignment $M$,  we define another assignment $\widetilde{M}$ by setting, for each $e \in \mathbb{E}$, $\widetilde{M}_{e} = \|M_{e}\|  \mathrm{Id}_{H}$ where $\| \cdot \| $ denotes the operator norm. 

We then have  
\[
\big| Z(M) \big| \leq Z (\widetilde{M})\;.
\]
\end{corollary}
\begin{proof}
This is an immediate consequence of Theorems~\ref{thm:loop_expansion} and~\ref{thm:real_loop_expansion} and the fact that, for any loop $\ell$, $|\hat{\ell}(M)| \leq \hat{\ell}(\widetilde{M})$, which follows from the estimate $|\Trace(M_{e_1} \cdots M_{e_n}) | \leq \|M_{e_1}\cdots M_{e_n}\| \dim(H) \leq \|M_{e_1}\| \cdots \|M_{e_n}\| \dim(H)$.
\end{proof}
\begin{remark}
In the context of \textit{Abelian} gauge theory, Corollary~\ref{cor:weakDM} is a special case of~\cite[Thm.~4.1]{BFS79}. 
However, Corollary~\ref{cor:weakDM} applies to \textit{non-Abelian} models.
See also~\cite[Thm.~2.3]{BFS79} for a similar result in a general non-Abelian setting, although with specific periodic boundary conditions that we do not require.

Our proof strategy for Corollary~\ref{cor:weakDM} (and, more importantly, for the quantitative diamagnetic inequality of Theorem~\ref{thm:moment_bound_X} below),
is also different:
our proof relies on a loop expansion, while~\cite[Thm.~4.1]{BFS79} relies on a Fourier transform and~\cite[Thm.~2.3]{BFS79} relies on translation invariance and Osterwalder--Schrader positivity.
See also~\cite{HSU77, SchraderSeiler78,Simon76} for another strategy based on Kato's inequality.
The most similar proofs to ours that we could find are in~\cite[Sec.~3]{BFS79}, where loop expansions and Feynman--Kac formulae are employed, but appear restricted to quadratic potentials $V$.
(See also the start of Section~\ref{sec:loop} for a discussion of related results of~\cite{Seiler82} that we generalise.)
\end{remark}

\section{Diamagnetic inequality for Abelian gauge theories}
\label{sec:diamagnetic}

In this section, we interpret a pure Abelian gauge theory with Villain action as a Gaussian measure.
We show that, if this measure is weighted by a Radon--Nikodym derivative of `positive type' (see Remark~\ref{rem:generalisation}),
then it exhibits Gaussian tails of the same strength as the original unweighted measure.
Except for Lemma~\ref{lemma:specific_loop_expansion}, this section is independent of Section~\ref{sec:loop}.

\subsection{Gauge freedom and the axial gauge}
\label{subsec:discrete_GFF_def}

Recall the notation from Definition~\ref{def:loops}.
Since our structure group $U(1)$ is Abelian,
for every loop $\ell\in\mcb{L}$,
there exists a unique function $\ell \colon \mathbf{P}\to \Z$ (which we denote by the same symbol)
such that
\begin{equ}\label{eq:E_ell_def}
\hol(g,\ell) = \prod_{p \in \mathbf{P}} g(\d p)^{\ell(p)}\;,\quad \forall g\in\mcG\;.
\end{equ} 
Note that every $\ell=(\ell_0,\ldots,\ell_n)\in\mcb{L}$ can be identified with a continuous (unparametrised) path in $[0,1]^2$ of length $n2^{-N}$ that starts and finishes at $\ell_0=\ell_n$ and successively moves along the unique affine path connecting $\ell_i$ and $\ell_{i+1}$ for $i=0,\ldots, n-1$ in the respective order.
When viewed this way,
$\ell(p)$ in~\eqref{eq:E_ell_def} is the winding number of $\ell$ around any point enclosed by $p$.

The holonomies $\{\hol(g,\ell)\}_{\ell\in\mcb{L}}$ capture all the gauge-invariant information
about $g\in\mcG$, so we can identify the orbit space $\mcG/\mfG$ with the set of maps $U(1)^\mathbf{P}$.
To see this in another way, let us fix a maximal tree $\mfT$ in $\Lambda$.
See Figure~\ref{fig:max_tree_gammas} for
an example.
\begin{figure}[H]
\centering
\captionsetup[subfloat]{font=small}
\begin{tikzpicture}
[baseline=(zero.base),scale = 0.82]
\foreach \x in {1,...,5}{
  \foreach \y in {1,...,5}{
    \fill[black] (\x,\y) circle (0.1);
  }
}
\draw[very thick] (1,5)--++(0,-4);

\foreach \y in {1,...,5}{
\draw[very thick] (1,\y)--++(4,0);
}

\foreach \x in {2,...,5}{
\draw[very thick,dotted] (\x,1)--++(0,4);
}
\node[name=zero] at (1+1/4,1+1/4) {$0$};

%
\end{tikzpicture}
\caption{Example of $\mfT$ for $N=2$.
Solid lines indicate bonds in $\mfT$,
dotted lines indicate bonds not in $\mfT$.
Circles indicate points in $\Lambda$.
}\label{fig:max_tree_gammas}
\end{figure}
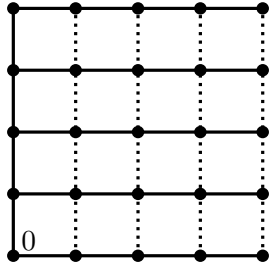
\begin{definition}
We say that $g\in\mcG$ is in the \emph{axial gauge} (with respect to $\mfT$) if $g_{xy}=1$ for all bonds $(x,y)$ in $\mfT$.
We denote by $\mcG_\mfT\subset \mcG$ the subset of gauge fields in the axial gauge.
\end{definition}
Observe that for every $g\in\mcG$, there exists a gauge transformation $u\in\mfG$ such that $g^u\in\mcG_\mfT$ (moreover $u$ is unique if we impose its value at one point, e.g. $u(0)=1$).
Then for any map $(X_p)_{p\in\mathbf{P}} \in U(1)^\mathbf{P}$, there exists a unique element $g\in\mcG_\mfT$
such that $g(\d p)=X_p$.
We can therefore identify $\mcG/\mfG$ with $\mcG_\mfT$ as sets.
Furthermore, we define
\begin{equ}[eq:Psi_def]
\Psi \colon \R^{\mathbf{P}} \to \mcG_\mfT\;,\quad \Psi\colon \{X_{p}\}_{p\in\mathbf{P}} \mapsto g
\end{equ}
where $g\in\mcG_\mfT$ satisfies $g(\d p) = \exp(\mbi X_p)$ for all $p\in\mathbf{P}$.

We will primarily factor out gauge invariance using the axial gauge, see also~\cite{Chatterjee16,Chevyrev19YM} and Section~\ref{subsec:axial} where this gauge is used for analytic purposes.
In the next subsection we introduce
the pure gauge theory (the discrete Yang--Mills measure) as a $\mcG_\mfT$-valued random variable.

\subsection{Bounds on holonomies}\label{subsec:hol_bounds}

Consider a family of independent centred Gaussians $Y=\{Y_p\}_{p\in \mathbf{P}}$ with variance $2^{-2N}$.
By a change of variable, one can see that, for all \textit{gauge-invariant} $f\colon \mcG\to \R$,
\begin{equ}\label{eq:pure_gauge}
\E f(\Psi Y) \propto \int_\mcG f(g) \prod_{p\in\mathbf{P}}Q(g(\partial p))\mrd g\;,
\end{equ}
where $Q$ is the  heat kernel on $U(1)$ at time $2^{-2N}$ as defined by~\eqref{eq:heat_kernel}.
Expression~\eqref{eq:pure_gauge} determines the classical Abelian pure gauge theory (the discrete Yang--Mills measure) of~\cite[Sec.~1.9]{Levy03} and~\cite[Thm.~1]{Levy06}.
We denote the law of $Y$ by $\nu$.

When studying the Abelian Higgs model~\eqref{eq:disc_YMH}, it is natural to introduce the $\R^\mathbf{P}$-valued random variable
$\{X_p\}_{p\in\mathbf{P}}$ with law $\mu$ given by
\begin{equ}\label{eq:mu_nu}
\mu(\mrd X) = Z^{-1}D(\Psi X)\nu(\mrd X)\;,
\end{equ}
where
\begin{equ}[eq:Dg_def]
D(g) =
\int_{\C^{\mathring\Lambda}} e^{\scal{\phi,\Delta_g\phi} - \int_{\mathring\Lambda} V(|\phi|)} \mrd \phi\;.
\end{equ}
The point in this definition is that,
since $\scal{\phi,\Delta_g\phi} - V(|\phi|)$ is gauge-invariant by~\eqref{eq:gauge_covar},
it follows from~\eqref{eq:pure_gauge} that, for any gauge-invariant function $f\colon\mcG\to \R$,
\begin{equ}\label{eq:law_g}
\E f(\g) = \E[f(\Psi X)]\;,
\end{equ}
where $\g$ is the random gauge field from~\eqref{eq:disc_YMH}.
For the next lemma, recall the definition of $\mathring{\mcb{L}}$ from Definition~\ref{def:loops}.
\begin{lemma}\label{lemma:specific_loop_expansion}
There exist real constants $(c_{\ell}: \ell \in \mathring{\mcb{L}})$ with $c_{\ell} \ge 0$ such that 
\begin{equ}\label{eq:D_def}
D(g) = \sum_{\ell\in\mathring{\mcb{L}}} c_\ell \Re [\hol(g,\ell)]\;.
\end{equ}
\end{lemma}
\begin{proof}
We apply Theorem~\ref{thm:loop_expansion} with $\mbbG = (\mbbV,\mbbE)$ given by $\mbbV=\Lambda$ and $\mbbE = \bonds$, the set of oriented edges in $\Lambda$ (there are two edges $(x,y),(y,x)\in\mbbE$ for every pair of nearest neighbours $x,y\in\Lambda$).
For $x\in \Lambda\setminus \mathring{\Lambda}$ we set $\rho^{\lambda_x} = \delta_0$, the Dirac mass at $0$, and, following~\eqref{eq:inner_prod_def}-\eqref{eq:Delta_g}, for $x\in\mathring{\Lambda}$ we set
$\mrd \rho^{\lambda_x} (\phi) = e^{-2^{-2N}V(|\phi|) -4 |\phi|^2}\mrd \phi$.
We further take the operators $M_e=g_e$, which corresponds to the off-diagonal (with respect to $\mathring \Lambda$) part of $\Delta_g$ (remark that there is no scaling factor in front of $g_e$ because $2^{2N}$ in~\eqref{eq:Delta_g} is cancelled by $2^{-2N}$ in~\eqref{eq:inner_prod_def}).
Remark that $\lambda_x$ for $x\in\Lambda$ has better than Gaussian tails by the assumption at the end of Section~\ref{sec:model},
so we are in the setting of Section~\ref{sec:loop}.

It follows from Theorem~\ref{thm:loop_expansion} that
$
D(g) = \sum_{\mcL \in \mathbf{L}_{\C}}
\tilde{c}_\mathcal{L}
\prod_{\ell \in \mathcal{L}} \hol(g,\ell)
$,
where
\begin{equ}
\tilde{c}_{\mathcal{L}} = \frac{1}{\mathcal{L}!}
\prod_{x \in \Lambda}
C^{\C,\lambda_{x}}_{\mathcal{L}_{x}/2}
\prod_{\ell \in \mathcal{L}} S(\ell)^{-1} \geq 0
\end{equ}
and where we used that $\Trace$ on $\C$ is the identity
and that $\prod_{e \in \ell} g_{e} = \hol(g,\ell)$ by definition.

Next, observe that $D(g)$ is real, which follows from the definition~\eqref{eq:Dg_def} and the fact that $\scal{\phi,\Delta_g\phi}$ and $V(|\phi|)$ are real.
Since $\tilde{c}_\mathcal{L}$ is also real, we obtain
\begin{equ}[eq:Dg_Re]
D(g) = \sum_{\mcL \in \mathbf{L}_{\C}}
\tilde{c}_\mathcal{L} 
\Re \Big[\prod_{\ell \in \mathcal{L}} \hol(g,\ell)\Big]\;.
\end{equ}
Remark that $C_{\mathcal{L}_x/2}^{\C,\lambda_x}=0$ whenever $x\in\Lambda\setminus\mathring\Lambda$ and $\mathcal{L}_x > 0$ because $\lambda_x = \delta_0$ (recall the definition of $C^{\C,\lambda}_j$ from Lemma~\ref{lem:integrate_higgs_onesite}).
Therefore, denoting by $\mathring{\mathbf{L}}_\C$ the set of multisets of loops $\mathcal{L}$ for which $\mcL_x=0$ for all $x\in\Lambda\setminus\mathring\Lambda$, it follows that $\tilde{c}_\mathcal{L}=0$ for all $\mathcal{L}\in \mathbf{L}_\C\setminus \mathring{\mathbf{L}}_\C$,
so we can restrict the sum in~\eqref{eq:Dg_Re} to $\mathcal{L} \in \mathring{\mathbf{L}}_\C$.

Finally, for every $\mathcal{L}\in  \mathring{\mathbf{L}}_\C$,
we can find a single loop $l=l (\mathcal{L}) \in\mathring{\mcL}$
such that $\hol(g,l)=\prod_{\bar\ell\in \mathcal{L}} \hol(g,\bar\ell)$ (take any $l\in\mathring{\mcb{L}}$ such that its winding number $l(p)$ is $\sum_{\bar\ell\in \mathcal{L}}\bar\ell(p)$ for every plaquette $p$; such $l$ exists but is not unique).
Choosing such $l (\mathcal{L}) \in\mathring{\mcb{L}}$ for every $\mathcal{L} \in  \mathring{\mathbf{L}}_\C$, we obtain the conclusion by setting, for any $\ell \in \mathring{\mcb{L}}$, 
 \[
c_{\ell} = \sum_{\mathcal{L} \in \mathring{\mathbf{L}}_\C}  \mathbf{1}\{ l(\mathcal{L}) = \ell\} \tilde{c}_{\mathcal{L}} \ge 0\;.
\]
\end{proof} 
\begin{remark}\label{rem:generalisation}
In what follows, the only property of $D$ that we will use is that it is of `positive type', i.e. that $D$ is of the form~\eqref{eq:D_def} for some $c_\ell\geq 0$.
In particular, all the results in the remainder of the article remain true
if $D$ is replaced by any integrable function of positive type and the law of $\g$ (up to gauge invariance)
is defined by~\eqref{eq:law_g} and~\eqref{eq:mu_nu}.
The inequality in Theorem~\ref{thm:moment_bound_X} furthermore holds if $Y=\{Y_p\}_{p\in\mathbf{P}}$ is taken as an arbitrary centred Gaussian vector.
\end{remark}
The main result of this section, Theorem~\ref{thm:moment_bound_X} and its Corollaries~\ref{cor:Gauss_tail},~\ref{cor:moment_bound_hol}, and~\ref{cor:moment_bound_plaq_sum},
give Gaussian tail bounds on sums of the form $\sum_p\ell(p)X_p$ together with (sums of) log-holonomies around 
loops of $\g$.
Although we are ultimately interested in $\g$ rather than $X$, it turns out simpler to bound $X$ first and use this bound to estimate $\g$.
These bounds are what we call \textit{quantitative} diamagnetic inequalities.
\begin{definition}
For a loop $\ell\in\mcb{L}$, we denote by $\omega(\ell)$ the squared $L^2$-norm of the winding number of $\ell$, i.e.
\begin{equ}
\omega(\ell) \eqdef 2^{-2N}\sum_{p\in\mathbf{P}} \ell(p)^2\;,
\end{equ}
where we recall that $\ell \colon \mathbf{P}\to\Z$ is the unique map satisfying~\eqref{eq:E_ell_def}.
\end{definition}
\begin{remark}\label{rem:geom}
$\omega(\ell)$ depends only on $\ell$ as a loop in $[0,1]^2$
and not on the choice of scale $N$.
For example, if $\ell$ is a simple loop,
then $\omega(\ell)$ is the area enclosed by $\ell$.
\end{remark}
\begin{theorem}\label{thm:moment_bound_X}
Let $\ell\in\mcb{L}$ and define the random variables $B=\sum_{p\in \mathbf{P}} \ell(p) Y_p$
and
$A=\sum_{p\in \mathbf{P}} \ell(p) X_p$.
Then $B$ is a centred Gaussian with variance $\omega(\ell)$
and for all $\eta\in[0,\frac12\omega(\ell)^{-1})$
\begin{equ}\label{eq:diamagnetic}
\E
\exp(\eta A^2)
\leq
\E\exp(\eta B^2) = 
(1-2\eta\omega(\ell))^{-1/2}\;.
\end{equ}
\end{theorem}
For the proof of Theorem~\ref{thm:moment_bound_X}, we record the following lemma.
\begin{lemma}\label{lem:decorrelation}
Let $(A,B)$ be an $\R^2$-valued centred multivariate Gaussian and denote $\sigma_A^2=\E[A^2]$, $\sigma_B^2=\E[B^2]$, and $\sigma_{AB} = \E[AB]$.
Then for all $\eta<\frac12\sigma_A^{-2}$
\begin{equ}
\E[e^{\eta A^2} \cos B] =
\exp
\Big[
\frac{\sigma_{AB}^2}{2\sigma_A^2}
\Big(
1-\frac{1}{1-2\eta\sigma_A^2}
\Big)
\Big]
\E[e^{\eta A^2}]\E[\cos B]\;.
\end{equ}
\end{lemma}

\begin{proof}
Since $B - A \sigma_{AB}/\sigma_A^2$ is a centred Gaussian independent of $A$ with variance $\sigma_B^2-\sigma_{AB}^2/\sigma_{A}^2$,
\begin{equ}\label{eq:A_cosB}
\E[e^{\eta A^2} \cos B ]
=
\E[e^{\eta A^2} e^{\mbi B}]
=
\E\big[
e^{\eta A^2} e^{\mbi A \sigma_{AB}/\sigma_A^2}\big]
e^{-\frac12
(\sigma_B^2-\sigma_{AB}^2/\sigma_A^2)}\;.
\end{equ}
For any bounded measurable $f\colon\R\to\R$, note that
\begin{equ}
\frac{\E [e^{\eta A^2}f(A)]}{\E[e^{\eta A^2}]}=
\mcZ^{-1}\int_\R e^{(\eta -\sigma^{-2}_A/2) x^2}f(x)\mrd x = \E[f(Z)]\;,
\end{equ}
where $\mcZ$ is such that the middle term is $1$ when $f\equiv 1$
and from which the second equality follows with
$Z\sim\mathcal{N}(0,1/(\sigma^{-2}_A-2\eta))$.
Hence
\begin{equ}\label{eq:A_Laplace}
\E[
e^{\eta A^2} e^{\mbi A \sigma_{AB}/\sigma_A^2}]
=
\E [e^{\eta A^2}] \exp\Big(-\frac{\sigma_{AB}^2}{2\sigma_A^2(1-2\eta\sigma_A^2)}
\Big)\;.
\end{equ}
The conclusion follows by combining~\eqref{eq:A_cosB} and~\eqref{eq:A_Laplace}.
\end{proof}

\begin{proof}[of Theorem~\ref{thm:moment_bound_X}]
The fact that $B$ is a centred Gaussian with $\Var(B)=\omega(\ell)$ follows from the fact that $\{Y_p\}_{p\in\mathbf{P}}$ is a family of independent centred Gaussians with variance $2^{-2N}$.
The equality in~\eqref{eq:diamagnetic} follows immediately.

To prove the inequality in~\eqref{eq:diamagnetic},
remark that, for any $x\in \R^\mathbf{P}$, one has $\Re [\hol(\Psi x,\ell)] = \cos(\sum_{p}\ell(p)x_p)$,
which follows from the identity~\eqref{eq:E_ell_def} and the definition of $\Psi$ in~\eqref{eq:Psi_def}.
Therefore, by~\eqref{eq:mu_nu} and Lemma~\ref{lemma:specific_loop_expansion},
\begin{equ}
\E e^{\eta A^2}
=
Z^{-1} \E\big[e^{\eta B^2} D(\Psi Y)\big]
=
Z^{-1} \E\Big[e^{\eta B^2}
\sum_{\ell'\in\mathring{\mcb{L}}}  c_{\ell'}\cos H_{\ell'} \Big]\;,
\end{equ}
where $H_{\ell'} = \sum_{p\in\mathbf{P}} \ell'(p)Y_p$ is a centred Gaussian.
By Lemma~\ref{lem:decorrelation},
\begin{equ}
\E[e^{\eta B^2}\cos H_{\ell'} ]\leq \E[e^{\eta B^2}]\E[\cos H_{\ell'}]\;.
\end{equ}
Using that $Z^{-1}\E \big[\sum_{\ell'\in\mathring{\mcb{L}}} c_{\ell'}\cos H_{\ell'} \big]= Z^{-1}\E D(\Psi Y) =1$, where the final equality is due to the fact that $\mu$
in~\eqref{eq:mu_nu} is a probability measure, the conclusion follows.
\end{proof}
We now state several corollaries of Theorem~\ref{thm:moment_bound_X}.
\begin{corollary}[Gaussian moments]\label{cor:Gauss_tail}
Let $\ell\in\mcb{L}$ and denote $ A = \sum_\ell \ell(p)X_p$.
Then for all $x\geq0$,
\begin{equ}\label{eq:Gauss_prob}
\P[| A|\geq x] \leq \sqrt 2 e^{-x^2/(4\omega(\ell))}\;.
\end{equ}
Furthermore, there exists a universal constant $C>0$ such that for all $q\geq 1$
\begin{equ}\label{eq:Gauss_moments}
\E[|A|^q]^{1/q} \leq C \sqrt{q\omega(\ell)}\;.
\end{equ}
\end{corollary}

\begin{proof}
By Theorem~\ref{thm:moment_bound_X}, $\E e^{\eta A^2} \leq (1-2\eta \omega(\ell))^{-1/2}$ for all $\eta\in[0,\frac12\omega(\ell)^{-1})$.
Taking $\eta=\frac14 \omega(\ell)^{-1}$, we obtain $\P[|A|>x] \leq  \sqrt 2 e^{-\eta x^2}$ by Markov's inequality, which proves~\eqref{eq:Gauss_prob}.
Moreover, expanding $e^{\eta A^2}$ in a power series
for the same choice of $\eta$, we obtain $\E[ A^{2k}] \leq \sqrt 2 (4\omega(\ell))^k k!$ for all $k\geq1$,
and~\eqref{eq:Gauss_moments} follows by Stirling's formula.
\end{proof}
We now come to consequences of Theorem~\ref{thm:moment_bound_X} for the gauge field marginal $\g$.
The following corollary shows that the $\log$ holonomies of $\g$ behave essentially no worse than those of pure gauge theory~\eqref{eq:pure_gauge} (we do not use this corollary later however).
\begin{corollary}\label{cor:moment_bound_hol}
There exists a universal constant $C>0$ such that, for all $q \geq 1$ and $\ell\in\mcb{L}$, 
\begin{equ}
\E[|\log \hol(\g,\ell)|^{q}]^{1/q}
\leq C\sqrt{q\omega(\ell)}\;.
\end{equ}
\end{corollary}

\begin{proof}
Immediate from the bound $|\log e^{\mbi x}| \leq |x|$ and~\eqref{eq:Gauss_moments}.
\end{proof}
The following corollary controls an important gauge-invariant function of $\g$ and is the way that we use Theorem~\ref{thm:moment_bound_X} below.
%
\begin{corollary}\label{cor:moment_bound_plaq_sum}
There exists a constant $C>0$ such that for every $q \geq 1$ and $\ell\in\mcb{L}$
\begin{equ}\label{eq:sum_log_hol_bound}
\E\Big[
\Big|
\sum_{p\in\mathbf{P}} \ell(p) \log \g(\partial p)
\Big|^{q}
\Big]^{1/q}
\leq C q\sqrt{\omega(\ell)}\;.
\end{equ}
\end{corollary}

\begin{proof}
Recalling that $\g\eqlaw \Psi X$ by~\eqref{eq:law_g},
observe that
$\g(\partial p) = \exp(\mbi X_p)$ and thus
$\log \g(\partial p) = X_p$ for all $p\in\mathbf{P}$ on the event $\max_{p\in\mathbf{P}}|X_p| < \pi$.
Furthermore, for $r \geq 1$,
\begin{equs}
\P\Big[
\max_{p\in\mathbf{P}}
|X_p| \geq \pi
\Big]
\leq \sum_{p\in\mathbf{P}}\P[|X_p|\geq \pi]
&\leq \pi^{-r}\sum_{p\in\mathbf{P}}\E[|X_p|^{r}]
\\
&\leq \pi^{-r}C^r r^{r/2} 2^{N(2- r)}\;,
\end{equs}
where we used~\eqref{eq:Gauss_moments} and that $\#\mathbf{P} = 2^{2N}$ in the final bound.
Observe that
\begin{equ}
2^{-2N}\Big|\sum_{p\in\mathbf{P}} \ell(p)\Big| 
\leq \omega(\ell)^{1/2}
\end{equ}
by the power mean inequality, and therefore
\begin{equs}
\E\Big[
\Big|
\sum_{p\in\mathbf{P}} \ell(p) \log \g(\partial p)
\Big|^{q}
\Big]
&\leq
\E\Big[
\Big|
\sum_{p\in\mathbf{P}} \ell(p) X_p
\Big|^{q}
\Big]
\\
&\qquad +
\pi^{q}2^{2Nq}\omega(\ell)^{q/2}\P
\Big[
\max_{p\in\mathbf{P}} |X_p|>\pi
\Big]
\\
&\leq
C^q q^{q/2} \omega(\ell)^{q/2}
+ C^r \pi^{q-r}r^{r/2}\omega(\ell)^{q/2} 2^{N(2q+2-r)}\;,
\end{equs}
where we again used~\eqref{eq:Gauss_moments}
in the final bound.
Taking $r=2q+2$ concludes the proof.
\end{proof}

\begin{remark}
We expect that $q$ in the bound~\eqref{eq:sum_log_hol_bound} can be replaced by $\sqrt q$.
\end{remark}

\section{Lattice \texorpdfstring{$1$}{1}-forms}
\label{sec:1-forms}

In this section we recall the norms on lattice $1$-forms introduced in~\cite{Chevyrev19YM}.
See also~\cite{CCHS2d} for a continuum analogue of these spaces.

Let $\mcR^{(N)}$ denote the set of rectangles $r\subset [0,1]^2$ with corners in $\Lambda^{(N)}$.
We treat every $r\in\mcR^{(N)}$ as a subset of $\mathbf{P}^{(N)}$ containing $|r|2^{2N}$ plaquettes, where $|r|$ is the area of $r$.
We denote by $\partial r\in\mcb{L}$
the loop around the boundary of $r$
oriented counter-clockwise and based at the lower-left corner.
As for plaquettes, we use the shorthand $g(\partial r) = \hol(g,\partial r)$ for which we note $g(\partial r) = \prod_{p\in r} g(\partial p)$.

For $0 \leq n \leq N$,
we denote by $\mcT^{(n)}\subset\mcR^{(N)}$ the 
set of rectangles with corners in $\Lambda^{(n)} = \{k2^{-n} \,:\, k= 0,\ldots, 2^{n}\}^2$ and with dimensions $2^{-n}\times k2^{-n}$
or $k2^{-n}\times 2^{-n}$ for $1\leq k \leq 2^n$.
Elements of $\mcT^{(n)}$ are called $n$-thin rectangles.

Denoting by $e_1,e_2$ the canonical basis vectors of $\R^2$, let
\begin{equ}
\mcX^{(N)} = \{(x,h)\,:\,
h = k2^{-N}e_i\,,\,
i\in\{1,2\}\,,\,
k\in \{0,\ldots, 2^{N}\}\,,\,
x,x+h\in \Lambda^{(N)}\}
\end{equ}
denote the set of all (positively oriented) axis-parallel line segments with endpoints in $\Lambda^{(N)}$.
For $\ell=(x,h)=(x,k2^{-N}e_i)$, we denote by
$|\ell|\eqdef |h|$ the length of $\ell$ and we call $e_i$ the direction of $\ell$.

Let $\obonds^{(N)}\subset \bonds^{(N)}$ denote the subset of bonds of the form $(x,x+2^{-N}e_i)$ for some $i\in\{1,2\}$.
Every $\ell=(x,k2^{-N}e_i)\in\mcX^{(N)}$ defines a subset of bonds $\{b_j\}_{1\leq j \leq |\ell|2^{N}} \subset\obonds^{(N)}$
by
$b_j = (x+(j-1)2^{-N}e_i,x+j2^{-N}e_i)$ and this subset determines $\ell$ uniquely.
In addition, $\ell=(x,h)\in\mcX^{(N)}$ can be identified with the
one-dimensional subset $\{x+th\,:\,t\in[0,1]\}\subset [0,1]^2$.
We will often identify $\ell=(x,k2^{-N}e_i)=(x,h)$ either with the set of bonds $\{b_j\}_{1\leq j \leq |\ell|2^{N}}$ or with the set $\{x+th\,:\,t\in[0,1]\}$, which will be clear from the context.

We say that $\ell,\bar\ell\in\mcX^{(N)}$ are parallel
and write $\ell\parallel\bar\ell$ if they have the same direction $e_i$ for some $i\in \{1,2\}$ and if
$\pi_i(\ell)=\pi_i(\bar\ell)$ where $\pi_i \colon [0,1]^2 \to [0,1]$
is the canonical projection onto the $i$-th coordinate,
see Figure~\ref{fig:parallel_rho}.
Note that $\pi_i(\ell)=\pi_i(\bar\ell)$ is equivalent to
$d(x,\bar\ell)=d(y,\ell) = d(\ell,\bar\ell)$ for all $x\in\ell, y\in\bar \ell$,
where for $z\in [0,1]^2$ and $A,B\subset [0,1]^2$,
$d(z,A) = \inf_{a\in A} |z-a|$ and $d(A,B) = \inf_{a\in A} d(a,B)$.

For parallel $\ell,\bar\ell\in\mcX^{(N)}$, we define
\begin{equ}
\rho(\ell,\bar\ell) = |\ell|^{1/2}d(\ell,\bar\ell)^{1/2}\;.
\end{equ}
Note that $\rho(\ell,\bar\ell)^2$ is the area of the rectangle with two of its sides $\ell,\bar\ell$, see Figure~\ref{fig:parallel_rho}.
\begin{figure}
\centering
\begin{tikzpicture}[scale = 0.75]

\draw[very thick] (0,0)--++(0,5)--++(5,0)--++(0,-5)--++(-5,0);

\draw[fill=gray!30,draw opacity=0]  (1,1)--++(0,3) -- ++(1,0) -- ++(0,-3);

\draw[thick, -{Latex[length=1.5mm]}] (1,1)--++(0,3)
node[draw=none, midway, left=0.2]{$\ell$}
node[draw=none, pos=0, below]{\footnotesize $x$}
node[draw=none, pos=1, above]{\footnotesize $x+h$};
%

\draw[thick, -{Latex[length=1.5mm]}] (2,1)--++(0,3)
node[draw=none, midway, right=0.2]{$\bar\ell$};

\draw[thick, -{Latex[length=1.5mm]}] (3.3,0.4)--++(0,3)
node[draw=none, midway, left=0.2]{$\hat\ell$};

\draw[thick, -{Latex[length=1.5mm]}] (4.3,1.5)--++(0,2)
node[draw=none, midway, left=0.2]{$\check\ell$};
\end{tikzpicture}
\caption{Lines $\ell=(x,h),\bar\ell$ are parallel,
but $\hat\ell$ and $\check\ell$ are not parallel to $\ell$. The shaded region has area $\rho(\ell,\bar\ell)^2$.
}\label{fig:parallel_rho}
\end{figure}
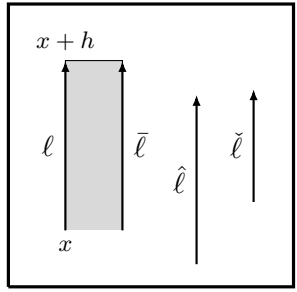

We denote by $\Omega^{(N)}$ the vector space of functions $A\colon\obonds^{(N)}\to \R$.
We canonically extend the definition of $A\in\Omega^{(N)}$ to a function $A\colon\mcX^{(N)}\to \R$ by $A(\ell) = \sum_{b\in\ell}A(\ell)$ where we recall that $\ell$ can be identified with a subset of $\obonds^{(N)}$
(understood as $A(\ell)=0$ if $|\ell|=0$).

For $\alpha\in [0,1]$,
we define on $\Omega_{N}$ the norm
\begin{equ}
|A|_{\gr\alpha} = \sup_{\substack{\ell \in \mcX^{(N)}\\|\ell|>0}} \frac{|A(\ell)|}{|\ell|^\alpha}
\end{equ}
and
the semi-norm
\begin{equ}
|A|_{\rnorm\alpha} = \sup_{\ell\parallel\bar\ell} \frac{|A(\ell)-A(\bar\ell)|}{\rho(\ell,\bar\ell)^{\alpha}}\;,
\end{equ}
where the supremum is taken over all distinct parallel $\ell,\bar\ell\in\mcX^{(N)}$.
We then denote by $\Omega^{(N)}_\alpha$ the space $\Omega^{(N)}$
equipped with the norm
\begin{equ}
|A|_{\alpha} \eqdef |A|_{\gr\alpha} + |A|_{\rnorm\alpha}\;.
\end{equ}

\begin{remark}
While the definition of $|\cdot|_{\rnorm\alpha}$ only involves parallel $\ell,\bar\ell$, we in fact have for all $\ell,\bar\ell\in\mcX^{(N)}$
\begin{equ}
|A(\ell)-A(\bar\ell)|  \lesssim_\alpha |A|_{\alpha} d_{\Haus}(\ell,\bar\ell)^{\alpha/2}\;,
\end{equ}
where $d_\Haus(\ell,\bar\ell)$ is the Hausdorff distance between $\ell,\bar\ell$ treated as subsets of $[0,1]^2$ (see~\cite[Prop.~3.9]{Chevyrev19YM}).
\end{remark}

\section{Gauge fixing}
\label{sec:gauge_fix}

In this section we adapt and improve several 
results of~\cite[Sec.~4]{Chevyrev19YM} in the special case that the structure group is Abelian;
our results and proofs, however, are self-contained and can be read without knowledge of~\cite{Chevyrev19YM}.
The main result of this section is Theorem~\ref{thm:gauge_fix}, which gives conditions under which a gauge field $g\in\mcG=\mcG^{(N)}$ can be gauge transformed to an element $g^u \in\mcG$ such that $|\log g^u|_\alpha$ is small.
For the remainder of this section we fix $N \geq 1$ and suppress it from our notation.

We use the following gauge-invariant function to measure the non-flatness of a gauge field $g$.

\begin{definition}\label{def:[g]}
For $g\in\mcG$ and $\alpha\geq0$ define
\begin{equ}\label{eq:[g]_def}
[g]_\alpha
=
\sup_{r\in\mcR} |r|^{-\alpha/2}
\Big|
\sum_{p\in\mathbf{P}\cap r} \log g(\partial p)
\Big|\;.
\end{equ}
\end{definition}
Recall the space of gauge transformations $\mfG$ from Definition~\ref{def:gauge_transforms}.
\begin{theorem}\label{thm:gauge_fix}
There exists $K>0$ with the following property.
Suppose $g\in\mcG$, $\alpha>0$, and $m \leq N$ are such that $2^m > (\frac{8}{\pi}[g]_\alpha)^{2/\alpha}\vee 8$.
Then there exists $u\in\mfG$
such that
for all $\beta\in[0,1]$ and $\kappa>0$
\begin{equ}
|\log g^u|_{\beta}
\leq
K (2^m + (1-2^{\kappa})^{-1}2^{-m\kappa}[g]_{\beta+\kappa})
\;.
\end{equ}
\end{theorem}
We give the proof of Theorem~\ref{thm:gauge_fix}
at the end of this section.
%
%
\begin{remark}\label{rem:trivial_bounds}
In addition to the bounds in Theorem~\ref{thm:gauge_fix}, we have, for all $g\in\mcG$ and $\beta\in[0,1]$, the trivial bound 
\begin{equ}
|\log g|_{\beta}
\leq
2\pi 2^{N(1+\beta/2)}\;.
\end{equ}
Indeed,
$|(\log g)(\ell)| \leq \pi 2^N|\ell|\leq \pi2^N|\ell|^{\beta}$ for all $\ell\in\mcX^{(N)}$.
Likewise, for distinct and parallel $\ell,\bar\ell\in\mcX^{(N)}$,
$\rho(\ell,\bar\ell)\geq 2^{-N/2}|\ell|^{1/2}$, 
and thus
\begin{equ}
|(\log g)(\ell)-(\log g)(\bar\ell)| \leq 2\pi 2^{N}|\ell| \leq 2\pi\rho(\ell,\bar\ell)^{\beta}2^{N(1+\beta/2)}\;.
\end{equ}
\end{remark}

\subsection{Landau gauge}
\label{subsec:Landau}
We first construct a discrete version of the 
Landau (i.e. Coulomb) gauge applied at dyadic scales.
Throughout this subsection, let us fix $1\leq m \leq N$ and $g\in\mcG$.
We define a gauge transformation $u\in\mfG$ inductively as follows.
Starting at scale $m$, we set $u(x) = 1$ for all $x\in\Lambda^{(m)}$.
Suppose now that we have defined $u$ on $\Lambda^{(n-1)}$ for $m < n \leq N$.
To define $u$ on $\Lambda^{(n)}$, consider the sub-lattice $\Lambda_1^{(n-1)} \subset \Lambda^{(n)}$ which consists of all points in $\Lambda^{(n-1)}$ together with the midpoints of bonds of $\Lambda^{(n-1)}$, see Figure~\ref{fig:plaquettes}.
We let $\bonds_1^{(n-1)}\subset\bonds^{(n)}$ denote the bonds of $\Lambda_1^{(n-1)}$.
For $x\in\Lambda^{(n-1)}_1\setminus\Lambda^{(n-1)}$ which is the midpoint of $(a,b)\in\bonds^{(n-1)}$, we set $u(x)$ as the unique element of $U(1)$ such that
\begin{equ}
\log g_{ax}^u = \log g_{xb}^u = \frac12\log g^u_{ab}\;.
\end{equ}
We then extend the definition of $u$ to the remaining points in $\Lambda^{(n)}$ as follows.
Consider $x\in \Lambda^{(n)}\setminus\Lambda_1^{(n-1)}$
and let $p_1,\ldots, p_4\in\mathbf{P}^{(n)}$ be the four plaquettes touching $x$ oriented anti-clockwise with $p_1$ at the top-right corner, see Figure~\ref{fig:plaquettes}.
For $i=1,\ldots, 4$ consider
\begin{equ}
\beta_i = \log g(\partial p_i) - \log g^u(b_{2i-1}) - \log g^u(b_{2i})\;,
\end{equ}
where $b_1,\ldots, b_8\in \bonds_{1}^{(n-1)}$ are the eight bonds encircling $x$ as in Figure~\ref{fig:plaquettes}.
Observe that
\begin{equ}\label{eq:beta_sum_zero}
\sum_{i=1}^4 \beta_i = 0 \mod 2\pi\;.
\end{equ}
We say that $g^u$ is \emph{small around $x$} if $\sum_{i=1}^4 \beta_i = 0$.
Observe that there is a bijection between choices for $u(x)\in U(1)$
and $\alpha_1,\ldots,\alpha_4\in \R \mod 2\pi$
such that
\begin{equ}\label{eq:alpha_condition}
\alpha_i-\alpha_{i+1}
= \beta_i \mod 2\pi \;,\quad i=1,\ldots, 4
\end{equ}
(indices are mod $4$ so that $\alpha_5=\alpha_1$),
which is given by the relation
\begin{equ}\label{eq:def_u_x}
e^{\mbi\alpha_i} = g^u_{xy_i} = u(x) g_{xy_i}u(y_i)^{-1}\;,
\end{equ}
where $y_1,\ldots y_4\in\Lambda_1^{(n-1)}$ as in Figure~\ref{fig:plaquettes}.

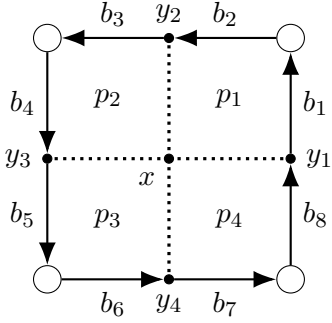
\begin{figure}[t]
\centering
\begin{tikzpicture}[scale = 1.6]
\foreach \x in {0,2}{
  \foreach \y in {0,2}{
    \node [draw, circle, name=c\x\y] at (\x,\y) {};
  }
}

\node [draw, dot, name=d10,label=below:$y_4$] at (1,0) {};
\node [draw, dot, name=d12,label=above:$y_2$] at (1,2) {};
\node [draw, dot, name=d01,label=left:$y_3$] at (0,1) {};
\node [draw, dot, name=d21,label=right:$y_1$] at (2,1) {};

\draw[thick, -{Latex[length=3mm]}] (c00) -- (d10) node[draw=none, midway, below=0.5]{$b_6$};
\draw[thick, -{Latex[length=3mm]}] (d10) -- (c20)
node[draw=none, midway, below=0.5]{$b_7$};
\draw[thick, -{Latex[length=3mm]}] (c20) -- (d21)
node[draw=none, midway, right=0.5]{$b_8$};
\draw[thick, -{Latex[length=3mm]}] (d21) -- (c22)
node[draw=none, midway, right=0.5]{$b_1$};
\draw[thick, -{Latex[length=3mm]}] (c22) -- (d12)
node[draw=none, midway, above=0.5]{$b_2$};
\draw[thick, -{Latex[length=3mm]}] (d12) -- (c02)
node[draw=none, midway, above=0.5]{$b_3$};
\draw[thick, -{Latex[length=3mm]}] (c02) -- (d01)
node[draw=none, midway, left=0.5]{$b_4$};
\draw[thick, -{Latex[length=3mm]}] (d01) -- (c00)
node[draw=none, midway, left=0.5]{$b_5$};

\draw[very thick, dotted] (1,1)--++(0,1);
\draw[very thick, dotted] (1,1)--++(0,-1);
\draw[very thick, dotted] (1,1)--++(1,0);
\draw[very thick, dotted] (1,1)--++(-1,0);
\node (p1) at (1.5,1.5) {$p_1$};
\node (p2) at (0.5,1.5) {$p_2$};
\node (p3) at (0.5,0.5) {$p_3$};
\node (p4) at (1.5,0.5) {$p_4$};

\node [draw,dot,name=x,label={-135:$x$}] at (1,1) {};

\end{tikzpicture}
\caption{
Large circles are points of $\Lambda^{(n-1)}$, small dots are points of $\Lambda^{(n)}$.
The points $y_1,\ldots, y_4$ along with the large circles are in $\Lambda^{(n-1)}_1$ and 
the point $x$ is in $\Lambda^{(n)}\setminus\Lambda^{(n-1)}_1$.
Arrows labelled by $b_1,\ldots, b_8$ are bonds in $\bonds^{(n-1)}_1$
and dotted lines represent bonds in $\bonds^{(n)}\setminus\bonds^{(n-1)}_1$.
}
\label{fig:plaquettes}
\end{figure}

If $g^u$ is small around $x$, observe that
\begin{equ}\label{eq:alpha_in_Landau}
\alpha_i \eqdef \frac{3}{8}(\beta_i - \beta_{i-1}) + \frac{1}{8}(\beta_{i+1}-\beta_{i+2})
\;,\quad i=1,\ldots, 4
\end{equ}
(indices are again mod $4$ so that
$\beta_0=\beta_4$, $\beta_5=\beta_1$, and $\beta_6=\beta_2$)
is a solution to~\eqref{eq:alpha_condition} (without mod $2\pi$ even),
and we define in this case $u(x)$ as the unique element
of $U(1)$ which determined by~\eqref{eq:def_u_x}-\eqref{eq:alpha_in_Landau}.
Note also that in this case, since $\log g^u(b_{j}) = \log g^u(b_{j+1})$
for all even $j=2,\ldots, 8$
by definition of $u$ on $\Lambda_1^{(n-1)}$ (with $b_9=b_1$), we have
\begin{equation}\label{eq:alpha_def_2}
\begin{split}
\alpha_i = \frac{\log g^u(b_{2i-3})-\log g^u(b_{2i})}{2}
&+
\frac{3}8
\big(
\log g(\partial p_i)-\log g(\partial p_{i-1})\big)
\\
&+ \frac{1}8
\big(
\log g(\partial p_{i+1}) - \log g(\partial p_{i+2})
\big)
\;.
\end{split}
\end{equation}
If $g^u$ is not small around $x$, then we simply define $u(x)=1$ (but only the case that $g^u$ is small around $x$ will ultimately be relevant to us).
The above procedure uniquely defines $u\in\mfG$ by induction.

\begin{remark}
As expected from gauge-invariance, the system of equations~\eqref{eq:alpha_condition}
has many solutions.
We break this gauge-invariance by the  auxiliary condition $\sum_{i=1}^4 \alpha_i = 0$
which~\eqref{eq:alpha_in_Landau} satisfies.
This is the lattice analogue of the Landau (Coulomb) gauge and corresponds to minimising the (squared) $L^2$-norm $\sum_{i=1}^4\alpha_i^2$.
\end{remark}

\begin{lemma}[Landau-type gauge]
\label{lem:Landau}
Suppose that for some $\alpha\geq 0$
\begin{equ}\label{eq:simple_bound}
[g]_\alpha 2^{-(m+1) \alpha/2}<\pi\;,
\end{equ}
and for some $c\in [0,\frac\pi6]$
\begin{equ}\label{eq:dyadic_bonds_bound}
\max
\Big\{
[g]_\alpha 2^{-(m+1)\alpha} ,
\max_{b\in\bonds^{(m)}} |\log g_b|/2
\Big\}
< c\;.
\end{equ}
Then
\begin{enumerate}[label=(\alph*)]
\item\label{pt:bonds_bound} 
\eqref{eq:dyadic_bonds_bound} holds with $m$ replaced by any $n\in\{m,\ldots, N\}$ and $g$ replaced by $g^u$,

\item\label{pt:gr_bound} for all $\beta\in[0,1]$ and $\kappa>0$
\begin{equ}\label{eq:A_bar_alpha}
|(\log g^u)|_{\gr\beta} \leq
c 2^{m+1} + 4[g]_{\beta+\kappa}2^{-(m+1)\kappa}(1-2^{\kappa})^{-1}\;.
\end{equ}
\end{enumerate}
\end{lemma}

\begin{proof}
Observe that~\eqref{eq:simple_bound} implies
$\log g(\partial r) = \sum_{p\in\mathbf{P}\cap r} \log g(\partial p)$
for any $r\in\mcR$ with $|r|\leq 2^{-(m+1)}$.
Hence, for every $\gamma\geq 0$, $m<n\leq N$, and $r\in\mcT^{(n)}$
\begin{equ}\label{eq:plaq_r_bound}
\Big|\sum_{p\in\mathbf{P}^{(n)}\cap r} \log g(\partial p)\Big|
\leq
[g]_\gamma|r|^{\gamma/2}\;,
\end{equ}
where we canonically treat $\mathbf{P}^{(n)}$ as a subset of $\mcR$ to define $g(\partial p)$ and treat $r$ as a rectangle in $\mcR^{(n)}$ to make sense of $\mathbf{P}^{(n)}\cap r$.
We proceed to prove~\ref{pt:bonds_bound} together with the fact that
\begin{equ}\label{eq:small_claim}
\text{$g^u$ is small around $x$ for all $x \in \Lambda^{(n)}\setminus \Lambda_1^{(n-1)}$ and $m < n\leq N$.}
\end{equ}
Consider first any $x\in \Lambda^{(m+1)}\setminus\Lambda_1^{(m)}$.
Using the notation from above with $n=m+1$,
note that~\eqref{eq:plaq_r_bound} with $\gamma=\alpha$ and $r\in\{p_1,\ldots, p_4\}$ and~\eqref{eq:dyadic_bonds_bound} imply $|\log g(\partial p_i)| < c$.
Since $c\leq \frac\pi6$, it follows that $|\beta_i| < \frac\pi2$, and thus $\sum_{i=1}^4\beta_i=0$ due to~\eqref{eq:beta_sum_zero}.
Hence $g^u$ is small around $x$.
Using again~\eqref{eq:plaq_r_bound} and the expression~\eqref{eq:alpha_def_2}
we obtain
\begin{equ}
|\alpha_i| \leq [g]_\alpha 2^{-(m+1) \alpha} + \frac{1}{2}\max_{b\in\bonds^{(m)}}|\log g_b|\;.
\end{equ}
It follows from~\eqref{eq:dyadic_bonds_bound}
that $|\alpha_i| < 2c$, and thus~\eqref{eq:dyadic_bonds_bound}
holds with $m$ replaced by $m+1$ and $g_b$ replaced by $g^u_b$.
By induction,~\eqref{eq:dyadic_bonds_bound} holds with $m$ replaced by any $m \leq n \leq N$ and $g_b$ replaced by $g^u_b$, 
which simultaneously proves~\ref{pt:bonds_bound} and~\eqref{eq:small_claim}.

We now prove~\ref{pt:gr_bound}
by induction on $m \leq n \leq N$.
Let us denote $A=\log g^u\in\Omega^{(N)}$.
For the base case,
note that for all $\ell\in\mcX^{(m)}$, $|A(\ell)| \leq 2c |\ell|2^{m} \leq c2^{m+1} |\ell|^{\beta}$.
Suppose now that $m<n\leq N$ and $|A(\ell)| \leq P_{n-1}|\ell|^{\beta}$ for some $P_{n-1}>0$ and all $\ell\in\mcX^{(n-1)}$.
Consider $\ell\in \mcX^{(n)}$ and suppose first that $\ell\subset \bonds_1^{(n-1)}$.
Consider $\bar\ell$ and $\underline\ell$ as the shortest and longest line in $\mcX^{(n-1)}$
which contains and is contained in $\ell$ respectively.
Note that $A(\ell) = \frac12(A(\bar\ell)+A(\underline{\ell}))$ by the choice of $u$ on $\Lambda_1^{(n-1)}$.
Hence, by the inductive hypothesis and concavity of $x\mapsto x^{\beta}$,
\begin{equs}
|A(\ell)|
= \frac{|A(\bar\ell) + A(\underline{\ell})|}{2}
\leq P_{n-1}\frac{|\bar\ell|^{\beta}+|\underline\ell|^{\beta}}{2}
\leq P_{n-1}\Big(\frac{|\bar\ell|+|\underline\ell|}{2}\Big)^{\beta}
= P_{n-1}|\ell|^{\beta}\;.
\end{equs}
Suppose now that $\ell \subset \bonds^{(n)}\setminus\bonds_1^{(n-1)}$.
Consider the two lines $\ell_1,\ell_2\in\mcX^{(n)}$ for which $\ell_1,\ell_2 \subset \bonds_1^{(n-1)}$ and which are parallel to and at distance $2^{-n}$ from $\ell$.
Then, since $g^u$ is small around all $x\in\Lambda^{(n)}\setminus\Lambda_1^{(n-1)}$,
using the expression~\eqref{eq:alpha_def_2},
we have 
$A(\ell) = \frac12(A(\ell_1)+A(\ell_2)) + \Delta$
where
\begin{equs}
|\Delta|
&\leq \Big| \sum_{p} \log g (\partial p) \Big| + 2[g]_{\beta+\kappa} 2^{-n(\beta+\kappa)}
\\
&\leq 2[g]_{\beta+\kappa} (|\ell|^{(\beta+\kappa)/2}2^{-(\beta+\kappa) n/2} + 2^{-n(\beta+\kappa)})
\\
&\leq 4[g]_{\beta+\kappa}|\ell|^{\beta}2^{-n\kappa}
\end{equs}
where the sum is over all plaquettes $p\in\mathbf{P}^{(n)}$ with two corners touching $\underline \ell$, the largest line contained in $\ell$ whose respective parallel lines $\underline\ell_1$ and $\underline\ell_2$ are in $\mcX^{(n-1)}$,
and where the second inequality follows from~\eqref{eq:plaq_r_bound}.
It follows that $|A(\ell)| \leq P_n|\ell|^{\beta}$ with $P_n = P_{n-1}+4[g]_{\beta+\kappa}2^{-n\kappa}$, from which~\eqref{eq:A_bar_alpha} follows.
\end{proof}

\subsection{Axial gauge}
\label{subsec:axial}
In order to apply the regularising Landau gauge above, 
we need to ensure that we can find $1\leq m \leq N$
for which the conditions of Lemma~\ref{lem:Landau} are satisfied.
As in the previous subsection, we fix $1\leq m \leq N$.

\begin{lemma}[Axial gauge]\label{lem:axial}
Suppose $\alpha, C\geq 0$ and $g\in\mcG^{(m)}$
such that
\begin{equ}\label{eq:hol_r_bound}
\sup_{r\in\mcT^{(m)}} |r|^{-\alpha/2}|\log g(\partial r)| \leq C\;.
\end{equ}
Then there exists $u\in\mfG^{(m)}$ such that
\begin{equ}\label{eq:axial_bound}
\max_{b\in\bonds^{(m)}} |\log g^u_b| \leq C 2^{-\alpha m/2}\;.
\end{equ}
\end{lemma}


\begin{proof}
Consider the maximal tree 
$\mfT$ as in Figure~\ref{fig:max_tree_gammas} (with $m$ in place of $N$).
There exists $u\in\mfG^{(m)}$ such that $g^u\in\mcG_\mfT$.
In particular, every bond $b\in\obonds^{(m)}$ not in $\mfT$ corresponds to a unique $m$-thin rectangle $r\in\mcT^{(m)}$ such that $g(\partial r) = g^u_b$, and thus $|\log g^u_b|\leq C2^{-\alpha m/2}$ by~\eqref{eq:hol_r_bound}.
\end{proof}

\begin{remark}\label{rem:periodic}
We see here a simplification of working on $[0,1]^2$ vs. $\T^2$ (i.e. with free vs. periodic boundary conditions):
for $\T^2$, Lemma~\ref{lem:axial} is not true due to global holonomies (though a version of it is true if the structure group is simply connected, see~\cite[Prop.~4.15]{Chevyrev19YM} and~\cite[Lem.~9.7]{CS23}).
In fact, our main result, Theorem~\ref{thm:moments_YM}, is \textit{not} true for $\T^2$;
this is because the measure~\eqref{eq:disc_YMH} in the continuum disintegrates over the isomorphism classes of $U(1)$-principal bundles
and realisations of the gauge field from non-trivial bundles cannot be represented by a global $1$-form.
To recover a version of our main result for $\T^2$, one could either adjust the measure~\eqref{eq:disc_YMH} to remain on the trivial bundle (see~\cite[Sec.~2.3]{Levy06} where this is done for the pure Yang--Mills model)
or restrict to bounds on $\log \g^\mbu$ in simply connected domains.
\end{remark}

\subsection{Proof of Theorem~\ref{thm:gauge_fix}}

\begin{proof}[of Theorem~\ref{thm:gauge_fix}]
Consider the smallest integer $m\geq 4$ such that $2^m > (\frac{8}{\pi}[g]_\alpha)^{2/\alpha}$.
By assumption, $m \leq N$.
For this choice of $m$, we readily see that
\begin{equ}{}
[g]_\alpha 2^{-(m+1)\alpha/2}<\pi\;,\quad [g]_\alpha 2^{-(m+1)\alpha}<\frac\pi8\;, \quad
[g]_\alpha 2^{-\alpha m/2}<\frac\pi4\;.
\end{equ}
Moreover, note that~\eqref{eq:hol_r_bound} holds with $C=[g]_\alpha$ by definition~\eqref{eq:[g]_def}.
Thus, by Lemma~\ref{lem:axial},~\eqref{eq:dyadic_bonds_bound}
holds with $c=\frac\pi8$ and with $g$ replaced by $g^u$
for some $u\in\mfG$.
By Lemma~\ref{lem:Landau}\ref{pt:gr_bound}, there exists $u\in\mfG$
such that for all $\beta\in[0,1]$ and $\kappa>0$
\begin{equ}\label{eq:gr_bar_alpha_bound}
|\log g^u|_{\gr{\beta}}
\lesssim
2^m + (1-2^{\kappa})^{-1}2^{-m\kappa}[g]_{\beta+\kappa} \;.
\end{equ}
Furthermore, Lemma~\ref{lem:Landau}\ref{pt:bonds_bound}
implies that $\max_{b\in\bonds^{(N)}}|\log g^u_b| < \frac\pi4$,
from which it follows by the discrete Stokes' theorem that for every $r\in\mcR$
\begin{equ}\label{eq:r_plaq_decomp}
\sum_{b \in \partial r} \log g^u_b
=
\sum_{p\in\mathbf{P}\cap r} \log g(\partial p)\;. 
\end{equ}
Recall that for parallel $\ell,\bar\ell\in\mcX$, we have
$\rho(\ell,\bar\ell)^2=|r|$ for the $r\in\mcR$
such that two sides of $r$ are $\ell,\bar\ell$, and the other two sides have length $d(\ell,\bar\ell)$.
By considering the cases $d(\ell,\bar\ell) \leq |\ell|$
and $d(\ell,\bar\ell)>|\ell|$,
it readily follows from~\eqref{eq:gr_bar_alpha_bound} 
and~\eqref{eq:r_plaq_decomp} that
\begin{equ}
|\log g^u|_{\rnorm{\beta}}
\lesssim
2^m + (1-2^{\kappa})^{-1}2^{-m\kappa}[g]_{\beta+\kappa}\;.
\end{equ}
\end{proof}

\section{Moment bounds on the gauge field marginal}
\label{sec:moments}

In this section we combine the results of Sections~\ref{sec:diamagnetic} and~\ref{sec:gauge_fix} to prove  Theorem~\ref{thm:moments_YM}.
We again fix $N\geq 1$ and suppress it from our notation.
We begin with two lemmas, the first of which
is purely combinatorial.

\begin{lemma}\label{lem:rect_decomp}
For every $r\in\mcR$ there exists a subset of thin rectangles $\mft \subset \cup_{n=1}^N\mcT^{(n)}$
such that $\mft$ is a partition of $r$ and such 
that $\sum_{t\in \mft} |t|^\alpha \lesssim_\alpha |r|^\alpha$
for all $\alpha \in (0,1]$.
\end{lemma}

\begin{proof}
It suffices to consider $r\in\mcR$ with side-lengths at most $\frac12$.
For an interval $[a,b] \subset [0,1]$ with $b-a\leq \frac12$
and $a,b\in 2^{-N}\Z$, consider the smallest $0 \leq n\leq N$
such that $z\eqdef k2^{-n}\in [a,b]$ for some $0 \leq k \leq 2^n$
(note that $k$ is unique).

Consider now the dyadic decomposition $a=h_{-u}<\ldots <h_{-1} < h_0 = z < h_1 \ldots<h_v=b$, 
where $h_{i}-h_{i-1}$ for $i\leq 0$
(resp. $i>0$) is the largest power of $2$ which fits into $[a,h_i]$
(resp. $[h_{i-1},b]$).
Note that every power of $2$ appears at most twice in the sequence $\{h_i-h_{i-1}\}_{-u<i\leq v}$.

Suppose $r$ has horizontal coordinates $[a,b]$ and vertical coordinates $[\bar a,\bar b]$.
Denoting by $h_i$ and $\bar h_j$ the respective decompositions of the two intervals, we note that every pair of subintervals $[h_{i-1},h_{i}], [\bar h_{j-1},\bar h_j]$ forms a thin rectangle.
We let $\mft \subset \cup_{n=1}^N\mcT^{(n)}$ denote the collection of these thin rectangles.
Then $\mft$ is a partition $r$.
Let $\bar t \in \mft$ such that $|\bar t|=\max_{t\in \mft} |t|$.
There are at most $4$ elements $t\in \mft$
of any given dimensions, thus the number of $t\in \mft$
with area $|\bar t| 2^{-m}$ is at most $4(m+1)$ for any $m\geq 0$.
It follows that
$\sum_{t\in \mft}|t|^\alpha \leq \sum_{m=0}^\infty (4m+1) |\bar t|^\alpha 2^{-\alpha m} \lesssim_\alpha |r|^\alpha$.
\end{proof}
The following lemma is a Kolmogorov-type bound on $[\g]_\alpha$.
\begin{lemma}\label{lem:Kolm_bound}
Let $\alpha\in [0,1)$ and $q > 0$.
Then 
\begin{equ}\label{eq:all_rect_bound}
\E\big[
[\g]_\alpha^{2q}
\big]
\lesssim_{\alpha,q} 1\;.
\end{equ}
\end{lemma}

\begin{proof}
It suffices to consider $q>\frac2{1-\alpha}$.
It holds that
\begin{equs}
\E\Big[
\sup_{1\leq n \leq N} \sup_{r \in \mcT^{(n)}}
\Big| \frac{\sum_{p\in\mathbf{P}\cap r} \log \g(\partial p)}{|r|^{\alpha /2}}\Big|^{2q}
\Big]
&\lesssim_{q} \sum_{n=1}^N \sum_{x\in\Lambda^{(n)}}
\sum_{k=1}^{2^n} (k2^{-2n})^{q(1-\alpha)}
\\
&\leq 4\sum_{n=0}^\infty 2^{n(2-q(1-\alpha))} \lesssim_{\alpha,q} 1\;,
\end{equs}
where in the first bound we used Corollary~\ref{cor:moment_bound_plaq_sum} with $\ell=\partial r$, and in the second bound we used $|\Lambda^{(n)}| \leq 2^{2n+2}$.
By Lemma~\ref{lem:rect_decomp}, we can replace $\sup_{1\leq n \leq N} \sup_{r \in \mcT^{(n)}}$ by $\sup_{r\in\mcR}$ on the left-hand side, which completes the proof.
\end{proof}

\begin{proof}[of Theorem~\ref{thm:moments_YM}]
If $N \leq 3$, then $|\log g|_{\beta} \lesssim 1$ for all $g\in\mcG$ by Remark~\ref{rem:trivial_bounds},
so we suppose $N>3$.
Let us fix $\alpha\in(0,1)$.
By Lemma~\ref{lem:Kolm_bound}, $\E[[\g]_\alpha^q]\lesssim_{\alpha,q} 1$ for all $q>0$.
By Theorem~\ref{thm:gauge_fix},
on the event $2^N>(\frac{8}{\pi}[\g]_\alpha)^{2/\alpha}$,
there exists $\mbu\in\mfG$ such that
for all $\beta\in[0,1]$ and $\kappa>0$
\begin{equ}
|\log \g^\mbu|_{\beta} \lesssim_{\alpha,\kappa} 1+[\g]_\alpha^{2/\alpha} + [\g]_{\beta+\kappa}\;.
\end{equ}
On the event $(\frac{8}{\pi}[\g]_\alpha)^{2/\alpha}\geq 2^N$,
we set $\mbu\equiv 1\in\mcG$, and note that 
$|\log \g^\mbu|_{\beta} \leq 2\pi 2^{N(1+\beta/2)}$
by Remark~\ref{rem:trivial_bounds}.
Hence, for all $q,r>0$,
\begin{equation}\label{eq:log_g_u_bound}
\begin{split}
\E
\big[
|\log \g^{\mbu}|_{\beta}^q
\big] 
&\lesssim_{\alpha,\kappa,q}
\E
\big[
1+[\g]_\alpha^{2q/\alpha} + [\g]_{\beta+\kappa}^q
\big]
\\
&\qquad\qquad + 2^{Nq(1+\beta/2)}
\P
\big[
(16\pi^{-1}[\g]_\alpha)^{2/\alpha}\geq 2^N
\big]
\\
&\lesssim_{\alpha,\kappa,q,r}
\E
\big[
1+[\g]_{\beta+\kappa}^{q}
\big]
+ 2^{Nq(1+\beta/2)}2^{-N r}
\E
\big[
[\g]_\alpha^{2r/\alpha}
\big]
\;.
\end{split}
\end{equation}
Choosing $r = q(1+\beta/2)$ and $\kappa>0$ such that $\beta+\kappa<1$ yields the desired estimate $\E[|\log \g^{\mbu}|_{\beta}^q] \lesssim 1$.

To conclude, we remark that $\mbu$ is measurable with respect to $\g$ since both the Landau gauge from Section~\ref{subsec:Landau} and axial gauge from the proof of Lemma~\ref{lem:axial},
which are used in the proof of Theorem~\ref{thm:gauge_fix},
are measurable with respect to the underlying gauge field.
\end{proof}

\appendix

\endappendix

\noindent\textbf{Acknowledgements:} The authors thank the anonymous referee for their helpful comments and suggestions.
I.C. acknowledges support from the EPSRC via the New Investigator Award EP/X015688/1.

\medskip

\noindent\textbf{Statement on competing interests:} The authors declare they have no competing interests. 
\bibliographystyle{./Martin}
\bibliography{./refs}

\end{document}